\documentclass[english]{article}
\usepackage[T1]{fontenc}
\usepackage[latin9]{inputenc}
\usepackage{geometry}
\usepackage{xcolor}
\usepackage{pdfcolmk}
\usepackage{babel}
\usepackage{amsmath}
\usepackage{amsthm}
\usepackage{amssymb}
\usepackage{graphicx}
\usepackage{wasysym}
\usepackage[authoryear]{natbib}
\PassOptionsToPackage{normalem}{ulem}
\usepackage{ulem}

\makeatletter

\providecolor{lyxadded}{rgb}{0,0,1}
\providecolor{lyxdeleted}{rgb}{1,0,0}

\DeclareRobustCommand{\lyxsout}[1]{\ifx\\#1\else\sout{#1}\fi}

\theoremstyle{plain}
\newtheorem{thm}{\protect\theoremname}
\theoremstyle{plain}
\newtheorem{prop}[thm]{\protect\propositionname}
\theoremstyle{plain}
\newtheorem{cor}[thm]{\protect\corollaryname}
\theoremstyle{plain}
\newtheorem{lem}[thm]{\protect\lemmaname}

\usepackage{nips15submit_e,times}
\usepackage{hyperref}
\usepackage{url}

\usepackage{ccaption}

\nipsfinalcopy

\makeatother

\providecommand{\corollaryname}{Corollary}
\providecommand{\lemmaname}{Lemma}
\providecommand{\propositionname}{Proposition}
\providecommand{\theoremname}{Theorem}

\begin{document}
\global\long\def\E{\mathbb{E}}%
\global\long\def\R{\mathbb{R}}%
\global\long\def\cov{\text{Cov}}%

\title{A deterministic and computable Bernstein-von Mises theorem}
\author{Guillaume Dehaene\\
Ecole Polytechnique Fédérale de Lausanne\\
\texttt{guillaume.dehaene@gmail.com}}
\maketitle
\begin{abstract}
Bernstein-von Mises results (BvM) establish that the Laplace approximation
is asymptotically correct in the large-data limit. However, these
results are inappropriate for computational purposes since they only
hold over most, and not all, datasets and involve hard-to-estimate
constants. In this article, I present a new BvM theorem which bounds
the Kullback-Leibler (KL) divergence between a fixed log-concave density
$f\left(\boldsymbol{\theta}\right)$ and its Laplace approximation.
The bound goes to $0$ as the higher-derivatives of $f\left(\boldsymbol{\theta}\right)$
tend to $0$ and $f\left(\boldsymbol{\theta}\right)$ becomes increasingly
Gaussian. The classical BvM theorem in the IID large-data asymptote
is recovered as a corollary.

Critically, this theorem further suggests a number of computable approximations
of the KL divergence with the most promising being:
\[
KL\left(g_{LAP},f\right)\approx\frac{1}{2}\text{Var}_{\boldsymbol{\theta}\sim g\left(\boldsymbol{\theta}\right)}\left(\log\left[f\left(\boldsymbol{\theta}\right)\right]-\log\left[g_{LAP}\left(\boldsymbol{\theta}\right)\right]\right)
\]
An empirical investigation of these bounds in the logistic classification
model reveals that these approximations are great surrogates for the
KL divergence. This result, and future results of a similar nature,
could provide a path towards rigorously controlling the error due
to the Laplace approximation and more modern approximation methods.
\end{abstract}

\section*{Introduction}

Bayesian inference is intrinsically plagued by computational problems
due to the fact that its key object, the posterior distribution $f\left(\boldsymbol{\theta}|\text{Data}\right)$,
is computationally hard to approximate. Solutions to this challenge
can be roughly decomposed into two classes. Sampling methods, dating
back to the famed Metropolis-Hastings algorithm (\citet{metropolis1953equation,hastings1970monte}),
provide a first possible solution: approximate the posterior using
a large number of samples whose marginal distribution is the posterior
distribution (\citet{brooks2011handbook}). Any expected value under
the posterior can then be approximated using the corresponding empirical
mean in the samples. A second solution is provided by \emph{Variational}
methods which aim to find the member of a parametric family $g\left(\boldsymbol{\theta};\lambda\right)$
which is the closest (in some sense) to the posterior $f\left(\boldsymbol{\theta}|\text{Data}\right)$
(\citet{blei2017variational}). For example, the historical Laplace
approximation (\citet{laplace1820theorie}) proposes to approximate
the posterior by a Gaussian centered at the Maximum A Posterior value
while the more modern Gaussian Variational Approximation finds a Gaussian
which minimizes the reverse Kullback-Leibler divergence (KL) (\citet{opper2009variational}).
Computationally, the choice between the two corresponds to a trade-off
between accuracy (the error of sampling methods typically converges
at speed $s^{-1/2}$ where $s$ is the number of samples, while Variational
methods will always have some residual error) and speed (Variational
methods tend to quickly and cheaply find the best approximation; e.g.
\citet{nickisch2008approximations}). Currently, Variational methods
are furthermore held back by the fact that they are perceived to be
unrigorous approximations due to the absence of results guaranteeing
their precision.

One limit for which Variational methods are particularly interesting
is for large datasets. Indeed, as the number of datapoints grows,
sampling methods struggle computationally since the calculation of
one additional sample requires a pass through the whole dataset (but
see \citet{bardenet2014towards,chen2014stochastic,maclaurin2015firefly,bardenet2017markov}
for modern attempts at tackling this issue). In contrast, Variational
methods are still able to tackle these large datasets since they can
leverage the computational prowess of optimization. For example, it
is straightforward to solve an optimization problem while accessing
only subsets (or batches) of the data at any one time, thus minimizing
the memory cost of the method while evaluation of the Metropolis-Hastings
ratio on subsets of the data is much trickier.

Furthermore, in the large-data limit, it is known that the posterior
distribution becomes simple: it tends to be almost equal to its Laplace
approximation (with total variation error typically scaling as $\mathcal{O}_{P}\left(n^{-1/2}\right)$
if the dimensionality $p$ is fixed), a result known as the Bernstein-von
Mises theorem (BvM; \citet{van2000asymptotic,kleijn2012bernstein}).
Intuitively, we should thus expect that Variational methods would
typically have smaller error as $n$ grows larger.

However, existing variants of the BvM theorem fail to be useful in
characterizing whether, for a given posterior $f\left(\boldsymbol{\theta}|\text{Data}\right)$
and a given approximation $g\left(\boldsymbol{\theta}\right)$, this
approximation is good enough or not. This might be due to historical
reasons since the BvM theorems might have aimed instead at proving
that Bayesian methods are valid in the frequentist paradigm (under
correct model specification; \citet{kleijn2012bernstein}).

In this article which expands upon the preliminary work in \citet{dehaene2017computing},
I propose to extend the scope of BvM theorems by proving (Th.\ref{thm:A-deterministic-BvM theorem}
and Cor.\ref{cor:Computable-approximations-of KL(g,f)}) that the
distance (measured using the KL divergence) between a \uline{log-concave}
density $f\left(\boldsymbol{\theta}\right)$ on $\R^{p}$ and its
Laplace approximation $g_{LAP}\left(\boldsymbol{\theta}\right)$ can
be (roughly and with caveats) approximated using the ``Kullback-Leibler
variance'':
\begin{equation}
KL\left(g_{LAP},f\right)\approx\frac{1}{2}\text{Var}_{\boldsymbol{\theta}\sim g\left(\boldsymbol{\theta}\right)}\left(\log\left[f\left(\boldsymbol{\theta}\right)\right]-\log\left[g_{LAP}\left(\boldsymbol{\theta}\right)\right]\right)
\end{equation}
thus yielding a computable quantity assessment of whether, in a given
problem, the Laplace approximation is good enough or not. Furthermore,
the KL divergence scales at most as:
\begin{equation}
KL\left(g_{LAP},f\right)=\mathcal{O}\left(\left(\Delta_{3}\right)^{2}p^{3}\right)
\end{equation}
where the scalar $\Delta_{3}$ (defined in eq.\ref{eq: controling the third derivative}
below) measures the relative strength of the third-derivative compared
to the second-derivative of $\log\left[f\left(\boldsymbol{\theta}\right)\right]$.
In the conventional large-data limit, $\Delta_{3}$ scales as $\mathcal{O}_{P}\left(n^{-1}\right)$
(Cor.\ref{cor:Bernstein-von-Mises:-IID case}) and this result thus
recovers existing BvM theorems since the total variational distance
scales as the square-root of the KL divergence due to Pinsker's inequality.

This article is organized in the following way. Section \ref{sec:Notations-and-background}
introduces key notations and gives a short review of existing BvM
results. Section \ref{sec:Three-limited-deterministic BvM propositions}
then introduces three preliminary propositions which give deterministic
approximations of $KL\left(g_{LAP},f\right)$. These approximations
are limited, but provide a key stepping stone towards Th.\ref{thm:A-deterministic-BvM theorem}.
Section \ref{sec:Gaussian-approximations-of log-concave distributions}
then presents the main result, Th.\ref{thm:A-deterministic-BvM theorem},
and shows that the classical BvM theorem in the IID large-data limit
is recovered as Cor.\ref{cor:Bernstein-von-Mises:-IID case}. Next,
I show that Th.\ref{cor:Bernstein-von-Mises:-IID case} yields computable
approximations of $KL\left(g_{LAP},f\right)$ and investigate them
empirically in the linear logistic classification model. Finally,
Section \ref{sec:Discussion} discusses the significance of these
findings and how they can be used to derive rigorous Bayesian inferences
from Laplace approximations of the posterior distribution.

Note that, due to the length of the proofs, I only give sketches of
the proofs in the main text. Fully rigorous proofs are given in the
appendix.

\section{Notations and background\label{sec:Notations-and-background}}

\subsection{Notations}

Throughout this article, let $\boldsymbol{\theta}\in\R^{p}$ denote
a $p$-dimensional random variable. I investigate the problem of approximating
a probability density:
\begin{equation}
f\left(\boldsymbol{\theta}\right)=\exp\left(-\phi_{f}\left(\boldsymbol{\theta}\right)-\log\left(Z_{f}\right)\right)
\end{equation}
using its Laplace approximation:
\begin{align}
g_{LAP}\left(\boldsymbol{\theta}\right) & =\exp\left(-\phi_{g}\left(\boldsymbol{\theta}\right)-\log\left(Z_{g}\right)\right)\\
 & =\exp\left(-\frac{1}{2}\left(\boldsymbol{\theta}-\boldsymbol{\mu}\right)^{T}\Sigma^{-1}\left(\boldsymbol{\theta}-\boldsymbol{\mu}\right)-\log\left(Z_{g}\right)\right)
\end{align}
I will distinguish whether variables have distribution $f$ or $g$
through the use of indices (e.g. $\boldsymbol{\theta}_{f},\boldsymbol{\theta}_{g}$)
or through more explicit notation. Recall that the mean and covariance
of $g_{LAP}$ are as follows: $\boldsymbol{\mu}$ is the maximum of
$f\left(\boldsymbol{\theta}\right)$ and $\Sigma$ is the inverse
of the negative log-Hessian of $f\left(\boldsymbol{\theta}\right)$:
\begin{align}
\boldsymbol{\mu} & =\text{argmin}_{\boldsymbol{\theta}}\left[\phi_{f}\left(\boldsymbol{\theta}\right)\right]\\
\Sigma^{-1} & =H\phi_{f}\left(\boldsymbol{\mu}\right)
\end{align}
The Laplace approximation thus corresponds to a quadratic Taylor approximation
of $\phi_{f}\left(\boldsymbol{\theta}\right)$ around the maximum
$\boldsymbol{\mu}$ of $f\left(\boldsymbol{\theta}\right)$.

I will furthermore assume that $\phi_{f}$ is a strictly convex function
that can be differentiated at least three times. This ensures that
$\boldsymbol{\mu}$ is unique, straightforward to compute through
gradient descent and $f\left(\boldsymbol{\theta}\right)$ is a log-concave
density function, a family with many interesting properties (see \citet{saumard2014log}
for a thorough review). Note that, in a Bayesian context, it is straightforward
to guarantee that the posterior is log-concave since it follows from
the prior and all likelihoods being log-concave. These assumptions
could be weakened but at the cost of a sharp loss of clarity.

I will measure the distance between $f$ and $g$ using the Kullback-Leibler
divergence:
\begin{align}
KL\left(g,f\right) & =\E_{\boldsymbol{\theta}\sim g\left(\boldsymbol{\theta}\right)}\left(\log\left[\frac{g\left(\boldsymbol{\theta}\right)}{f\left(\boldsymbol{\theta}\right)}\right]\right)\\
KL\left(f,g\right) & =\E_{\boldsymbol{\theta}\sim f\left(\boldsymbol{\theta}\right)}\left(\log\left[\frac{f\left(\boldsymbol{\theta}\right)}{g\left(\boldsymbol{\theta}\right)}\right]\right)
\end{align}
I will refer to these as the forward (from truth to approximation)
KL divergence $KL\left(f,g\right)$ and the reverse (approximation
to truth) KL divergence $KL\left(g,f\right)$ in order to distinguish
them.

The KL divergence is a positive quantity which upper-bounds the total-variation
distance through the Pinsker inequality:
\begin{equation}
\left[d_{TV}\left(g,f\right)\right]^{2}\leq\frac{1}{2}KL\left(g,f\right)
\end{equation}
Contrary to the total-variation distance, the KL divergence is not
symmetric nor does it generally respect a triangle inequality.

In order to state convergence results, I will use the probabilistic
big-O and small-o notations. Recall that $X_{n}=\mathcal{O}_{P}\left(a_{n}\right)$
is equivalent to the sequence $X_{n}/a_{n}$ being bounded in probability:
for any $\epsilon>0$, there exists a threshold $\delta$ such that:
\begin{equation}
\forall n\ \mathbb{P}\left(\left|\frac{X_{n}}{a_{n}}\right|>\delta\right)\leq\epsilon
\end{equation}
while $X_{n}=o_{P}\left(a_{n}\right)$ is equivalent to the sequence
$X_{n}/a_{n}$ converging weakly to $0$, i.e. for any threshold $\delta$:
\begin{equation}
\mathbb{P}\left(\left|\frac{X_{n}}{a_{n}}\right|>\delta\right)\rightarrow0
\end{equation}

Finally, my analysis strongly relies on three key changes of variables.
First, I will denote with $\tilde{\boldsymbol{\theta}}$ the affine
transformation of $\boldsymbol{\theta}$ such that $g$ is the standard
Gaussian distribution:
\begin{equation}
\tilde{\boldsymbol{\theta}}=\Sigma_{g}^{-1/2}\left(\boldsymbol{\theta}-\boldsymbol{\mu}\right)
\end{equation}
Densities and log-densities on the $\tilde{\boldsymbol{\theta}}$
space will be denoted as $\tilde{f},\tilde{g}_{LAP},\tilde{\phi}_{f},\tilde{\phi}_{g}$.\\
Second, I will further consider a switch to spherical coordinates
in the $\tilde{\boldsymbol{\theta}}$ referential:\begin{subequations}
\begin{align}
\tilde{\boldsymbol{\theta}} & =r\boldsymbol{e} & r\in\R^{+},\ \boldsymbol{e}\in S^{p-1}\\
r & =\left\Vert \tilde{\boldsymbol{\theta}}\right\Vert _{2}\\
\boldsymbol{e} & =\frac{\tilde{\boldsymbol{\theta}}}{\left\Vert \tilde{\boldsymbol{\theta}}\right\Vert _{2}}
\end{align}
\end{subequations} where $S^{p-1}$ is the $p$-dimensional unit
sphere. Recall that, under the Gaussian distribution $g_{LAP}$, $r,\boldsymbol{e}$
are independent, $r$ is a $\chi_{p}$ random variable and $\boldsymbol{e}$
is a uniform random variable over $S^{p-1}$.\\
Finally, for technical reasons, I do not work with the radius $r$
but with its cubic root $c$:
\begin{equation}
r=c^{3}
\end{equation}
This change of variable plays a key role in the derivation of Theorem
\ref{thm:A-deterministic-BvM theorem} as detailed in section \ref{sec:Gaussian-approximations-of log-concave distributions}.

\subsection{The Bernstein-von Mises theorem\label{subsec: intro to The-Bernstein-von-Mises}}

The Bernstein-von Mises theorem asserts that the Laplace approximation
is asymptotically correct. While it was already intuited by Laplace
(\citet{laplace1820theorie}), the first explicit formulations date
back to separate foundational contributions by \citet{bernsteinTheory}
and \citet{mises1931wahrscheinlichkeitsrechnung}. The first rigorous
proof was given by \citet{doob1949application}.

We focus here on the IID setting, which is the easiest to understand
but the result can be derived in the very general Locally Asymptotically
Normal setup (LAN, \citet{van2000asymptotic,kleijn2012bernstein}).

Consider the problem of analyzing $n$ IID datapoints $X_{i}$, with
common density $d\left(x\right)$, according to some Bayesian model,
composed of a prior $f_{0}\left(\boldsymbol{\theta}\right)=\exp\left(-\phi_{0}\left(\boldsymbol{\theta}\right)\right)$
and a conditional model describing the conditional IID distribution
of $X|\boldsymbol{\theta}$ with density $d\left(x|\boldsymbol{\theta}\right)$.
The posterior is then:
\begin{align}
f_{n}\left(\boldsymbol{\theta}\right) & =f\left(\boldsymbol{\theta}|x_{1}\dots x_{n}\right)\\
 & \propto f_{0}\left(\boldsymbol{\theta}\right)\prod_{i=1}^{n}d\left(X_{i}=x_{i}|\boldsymbol{\theta}\right)
\end{align}

Noting $NLL_{i}=-\log\left[\left(X_{i}=x_{i}|\boldsymbol{\theta}\right)\right]$
the negative log-density due to the $i^{th}$ datapoint, and $\phi_{n}\left(\boldsymbol{\theta}\right)$
the negative log-density of the joint distribution $f_{n}\left(\boldsymbol{\theta},x_{1}\dots x_{n}\right)$,
we have:
\begin{align}
\phi_{n}\left(\boldsymbol{\theta}\right) & =\phi_{0}\left(\boldsymbol{\theta}\right)+\sum_{i=1}^{n}NLL_{i}\left(\boldsymbol{\theta}\right)\\
 & =n\left[\frac{1}{n}\phi_{0}\left(\boldsymbol{\theta}\right)+\frac{1}{n}\sum_{i=1}^{n}NLL_{i}\left(\boldsymbol{\theta}\right)\right]\label{eq: the log-density is like a mean}
\end{align}
The trivial rewriting in eq.\eqref{eq: the log-density is like a mean}
makes explicit that $\phi_{n}\left(\boldsymbol{\theta}\right)$ is
somewhat akin to a biased empirical mean of the $NLL_{i}\left(\boldsymbol{\theta}\right)$
which are IID function-valued random variables, a structure they inherit
from the $X_{i}$ being IID.

Critically, the BvM theorem does not require correct model specification,
i.e. for the data-generating density $d\left(x\right)$ to correspond
to one of the inference-model densities $d\left(x|\boldsymbol{\theta}\right)$.
Note that this requires us to reconsider what is the goal of the analysis
since it means that there is no true value $\boldsymbol{\theta}_{0}$
such that $d\left(x|\boldsymbol{\theta}_{0}\right)=d\left(x\right)$.
As an alternative, we might try to recover the value of $\boldsymbol{\theta}$
such that the conditional density is the closest to the truth. Both
Maximum Likelihood Estimation (MLE) and Bayesian inference behave
in this manner and aim at recovering the parameter $\boldsymbol{\theta}_{0}$
which minimizes the KL divergence:
\begin{equation}
\boldsymbol{\theta}_{0}=\text{argmin}_{\boldsymbol{\theta}}\left[KL\left(d\left(x\right),d\left(x|\boldsymbol{\theta}\right)\right)\right]
\end{equation}
which is equivalent to minimizing the theoretical average of the $NLL_{i}$:
\begin{equation}
\boldsymbol{\theta}_{0}=\text{argmin}_{\boldsymbol{\theta}}\left[\E_{X}\left(NLL\left(\boldsymbol{\theta}\right)\right)\right]
\end{equation}
The connection is particularly immediate for the MLE since it aims
at recovering the minimum of the theoretical average by using the
minimum of the empirical average $\frac{1}{n}\sum_{i=1}^{n}NLL_{i}\left(\boldsymbol{\theta}\right)$.

Three technical conditions need to hold in order for the BvM theorem
to apply in the IID setting:
\begin{enumerate}
\item The prior needs to put mass on all neighborhoods of $\boldsymbol{\theta}_{0}$.\\
If not, then the prior either rules out $\boldsymbol{\theta}_{0}$
or $\boldsymbol{\theta}_{0}$ is on the edge of the prior. Both possibilities
modify the limit behavior heavily.
\item The posterior needs to concentrate around $\boldsymbol{\theta}_{0}$
as $n\rightarrow0$.\\
More precisely, for any $\epsilon>0$, the posterior probability of
the event $\left\Vert \boldsymbol{\theta}-\boldsymbol{\theta}_{0}\right\Vert _{2}\leq\epsilon$
needs to converge to $1$. Recall that the posterior is random since
it inherits its randomness from the sequence $X_{i}$, this is thus
a convergence in probability:
\begin{equation}
\mathbb{P}_{\boldsymbol{\theta}\sim f_{n}\left(\boldsymbol{\theta}\right)}\left[\left\Vert \boldsymbol{\theta}-\boldsymbol{\theta}_{0}\right\Vert _{2}\leq\epsilon\right]=1+o_{P}\left(1\right)
\end{equation}
This is a technical condition that ensures that the posterior is \emph{consistent
}in recovering $\boldsymbol{\theta}_{0}$ as $n$ grows. This is very
comparable to the necessity of assuming that the MLE is consistent
in order for it to exhibit Gaussian limit behavior (\citet{van2000asymptotic}).
\item The function-valued random variables $NLL$ needs to have a regular
distribution. More precisely:
\begin{enumerate}
\item The theoretical average function: $\boldsymbol{\theta}\rightarrow\E\left(NLL\left(\boldsymbol{\theta}\right)\right)$
needs to have a strictly positive Hessian at $\boldsymbol{\theta}_{0}$.
\item The gradient of $NLL$ at $\boldsymbol{\theta}_{0}$ must have finite
variance.
\item In a local neighborhood around $\boldsymbol{\theta}_{0}$, $NLL$
is $m$-Lipschitz where $m$ is a random variable such that $\E\left(m^{2}\right)$
is finite.
\end{enumerate}
\end{enumerate}
Under such assumptions, then the posterior distribution becomes asymptotically
Gaussian, in that the total-variation distance between the posterior
$f_{n}\left(\boldsymbol{\theta}\right)$ and its Laplace approximation
$g_{LAP,n}\left(\boldsymbol{\theta}\right)$ converges to $0$ as:
\begin{equation}
d_{TV}\left(f_{n},g_{LAP,n}\right)=\mathcal{O}_{P}\left(\frac{1}{\sqrt{n}}\right)
\end{equation}
Indeed, both $f_{n}$ and $g_{LAP,n}$ inherit their randomness from
the random sequence $X_{i}$ and the distance between the two is thus
random. The BvM theorem establishes furthermore that Bayesian inference
is a valid procedure for frequentist inference in that it is equivalent
in the large $n$ limit to Maximum Likelihood Estimation. Furthermore,
\uline{under correct model specification}, Bayesian credible intervals
are also confidence intervals (\citet{kleijn2012bernstein}).

However, BvM theorems of this nature are limited in several ways.
One trivial but worrying flaw is that they ignore the prior. If we
consider for example a ``strong'' Gaussian prior with small variance,
then when $n$ is small, the likelihood will have negligible influence
on the posterior, and we intuitively expect that the posterior will
be almost equal to its Laplace approximation. Existing theorems fail
to establish this. Furthermore, such BvM theorems are of a different
probabilistic nature than normal Bayesian inference. Indeed, in BvM
theorems, the probability statement is made a-priori from the data
$X_{i}$. The result thus applies to the typical posterior and not
to any specific one. This complicates the application of such results
in a Bayesian setting where the focus is instead in making probabilistic
statements conditional on the data. Finally, these theorems are also
of limited applicability due to their reliance on theoretical quantities
that are inaccessible directly and hard to estimate, like the random
Lipschitz constant $m$ and its second moment $\E\left(m^{2}\right)$.

In this article, I will prove a theorem which addresses these limits
and gives a sharp approximation of the KL divergence between the Laplace
approximation $g_{LAP}\left(\boldsymbol{\theta}\right)$ and a fixed
log-concave target $f\left(\boldsymbol{\theta}\right)$. This bound
recovers the classical BvM theorem in the classical large-data limit
setup that we have presented here (see Section \ref{subsec:Recovering-the-classical IID theorem}).
Finally, this bound identifies both the small $n$ and large $n$
regimes in which the posterior is approximately Gaussian (see Section
\ref{subsec:Empirical-results-in}).

\section{Three limited deterministic BvM propositions\label{sec:Three-limited-deterministic BvM propositions}}

In this section, I present three simple propositions that are almost-interesting
deterministic BvM results. However, all three fall short due to either
relying on assumptions that are too restrictive or referring to quantities
that are unwieldy. Even then, they still are useful both in building
intuition on BvM results and as a gentle introduction into important
tools for the proof of Theorem \ref{thm:A-deterministic-BvM theorem}.

Note that, while they are restricted in other ways, the following
propositions hold under more general assumptions than only for log-concave
$f\left(\boldsymbol{\theta}\right)$ and for its Laplace approximation
$g_{LAP}\left(\boldsymbol{\theta}\right)$, as emphasized in every
section and in the statement of the propositions.

\subsection{The ``Kullback-Leibler variance''}

Throughout this subsection, let $f\left(\boldsymbol{\theta}\right)$
and $g\left(\boldsymbol{\theta}\right)$ be any probability densities.

The KL divergence between $g$ and $f$ corresponds to the first moment
of the random variable $\phi_{f}\left(\boldsymbol{\theta}\right)-\phi_{g}\left(\boldsymbol{\theta}\right)+\log\left(Z_{f}\right)-\log\left(Z_{g}\right)$
under the density $g\left(\boldsymbol{\theta}\right)$:
\begin{align}
KL\left(g,f\right) & =\E_{\boldsymbol{\theta}\sim g\left(\boldsymbol{\theta}\right)}\left[\log\left[\frac{g\left(\boldsymbol{\theta}\right)}{f\left(\boldsymbol{\theta}\right)}\right]\right]\\
 & =\E_{\boldsymbol{\theta}\sim g\left(\boldsymbol{\theta}\right)}\left[\phi_{f}\left(\boldsymbol{\theta}\right)-\phi_{g}\left(\boldsymbol{\theta}\right)+\log\left(Z_{f}\right)-\log\left(Z_{g}\right)\right]
\end{align}
It is very hard to say anything about this quantity since it requires
knowledge of both normalization constants $Z_{f}$ and $Z_{g}$. Outside
of the handful of probability distributions for which the integral
is known, normalization constants prove to be a major obstacle towards
theoretical or computational results of any kind.

The following proposition offers a possible solution: the variance
of $\phi_{f}\left(\boldsymbol{\theta}\right)-\phi_{g}\left(\boldsymbol{\theta}\right)+\log\left(Z_{f}\right)-\log\left(Z_{g}\right)$
can be used to approximate the KL divergence. I will refer to this
quantity as the \emph{KL variance} defined as:
\begin{equation}
KL_{var}\left(g,f\right)=\text{Var}_{\boldsymbol{\theta}\sim g\left(\boldsymbol{\theta}\right)}\left[\phi_{f}\left(\boldsymbol{\theta}\right)-\phi_{g}\left(\boldsymbol{\theta}\right)\right]
\end{equation}
Critically, $KL_{var}$ does not require knowledge of the normalization
constants because they do not modify the variance.
\begin{prop}
Restrictive approximation of $KL\left(g,f\right)$.\label{prop: KL approx KL_var}

For any densities $f\left(\boldsymbol{\theta}\right)$ and $g\left(\boldsymbol{\theta}\right)$,
define the exponential family:
\begin{equation}
h\left(\boldsymbol{\theta};\lambda\right)=g\left(\boldsymbol{\theta}\right)\left[\frac{f\left(\boldsymbol{\theta}\right)}{g\left(\boldsymbol{\theta}\right)}\right]^{\lambda}\exp\left[-C\left(\lambda\right)\right]\ \ \lambda\in\left[0,1\right]
\end{equation}
then the reverse KL divergence can be approximated:
\begin{equation}
KL\left(g,f\right)=\frac{1}{2}KL_{var}\left(g,f\right)+\left[\int_{0}^{1}\frac{\lambda^{2}}{2}k_{3}\left(\lambda\right)d\lambda\right]
\end{equation}
where $k_{3}\left(\lambda\right)$ is the third cumulant of $\phi_{f}\left(\boldsymbol{\theta}\right)-\phi_{g}\left(\boldsymbol{\theta}\right)$
under $h\left(\boldsymbol{\theta};\lambda\right)$:
\begin{align}
k_{1}\left(\lambda\right) & =\E_{\boldsymbol{\theta}\sim h\left(\boldsymbol{\theta};\lambda\right)}\left[\phi_{f}\left(\boldsymbol{\theta}\right)-\phi_{g}\left(\boldsymbol{\theta}\right)\right]\\
k_{3}\left(\lambda\right) & =\E_{\boldsymbol{\theta}\sim h\left(\boldsymbol{\theta};\lambda\right)}\left[\left\{ \phi_{f}\left(\boldsymbol{\theta}\right)-\phi_{g}\left(\boldsymbol{\theta}\right)-k_{1}\left(\lambda\right)\right\} ^{3}\right]
\end{align}

If the difference between the log-densities is bounded, then the difference
between $KL$ and $KL_{var}$ is bounded too:
\begin{align}
\left|\phi_{f}\left(\boldsymbol{\theta}\right)-\phi_{g}\left(\boldsymbol{\theta}\right)\right| & \leq M\\
\left|KL\left(g,f\right)-\frac{1}{2}KL_{var}\left(g,f\right)\right| & \leq\frac{M^{3}}{6}
\end{align}
\end{prop}

This proposition offers the $KL_{var}$ as a great alternative to
the KL divergence that is both computable and tractable. However,
it falls short due to the very important limitation that $\phi_{f}\left(\boldsymbol{\theta}\right)-\phi_{g}\left(\boldsymbol{\theta}\right)$
must be bounded in order to achieve precise control of the error of
the approximation. This means that Prop.\ref{prop: KL approx KL_var}
will not be able to provide a rigorous general BvM theorem. Indeed,
the Laplace approximation is a purely local approximation of $f\left(\boldsymbol{\theta}\right)$
and the statement $\phi_{g}\left(\boldsymbol{\theta}\right)\approx\phi_{f}\left(\boldsymbol{\theta}\right)$
is only valid in a small neighborhood around the MAP value $\boldsymbol{\mu}$.
As $\left\Vert \boldsymbol{\theta}-\boldsymbol{\mu}\right\Vert _{2}$
becomes large, the error typically tends to infinity.

\subsection{The log-Sobolev inequality}

Throughout this subsection, $g\left(\boldsymbol{\theta}\right)$ is
not restricted to being the Laplace approximation of $f\left(\boldsymbol{\theta}\right)$
and is not necessarily restricted to being Gaussian.

Another possible path towards avoiding the normalization constants
in the KL divergence is given by the Log-Sobolev Inequality (LSI;
\citet{bakry1985diffusions,otto2000generalization}). The LSI relates
the KL divergence, which requires knowledge of the normalization constants,
to the \emph{relative Fisher information} defined as:
\begin{equation}
RI_{F}\left(g,f\right)=\E_{\boldsymbol{\theta}\sim g\left(\boldsymbol{\theta}\right)}\left[\left\Vert \nabla\phi_{f}\left(\boldsymbol{\theta}\right)-\nabla\phi_{g}\left(\boldsymbol{\theta}\right)\right\Vert _{2}^{2}\right]
\end{equation}
which does not require the normalization constants due to the derivative.

A key theorem by \citet{bakry1985diffusions} shows that the relationship
between $KL$ and $RI_{F}$ can be controlled by the minimal curvature
of $\phi_{f}$, i.e. the smallest eigenvalue of the matrices $H\phi_{f}\left(\boldsymbol{\theta}\right)$.
In order for the LSI to hold, this minimal curvature needs to be strictly
positive, i.e. $\phi_{f}\left(\boldsymbol{\theta}\right)$ needs to
be strongly log-concave (\citet{saumard2014log}). The LSI is usually
stated in the following form.
\begin{thm}
LSI (\citet{bakry1985diffusions}; \citet{otto2000generalization}
Th.2).

If $f\left(\boldsymbol{\theta}\right)$ is a $\lambda$-strongly log-concave
density, i.e.
\begin{equation}
\forall\boldsymbol{\theta}\ H\phi_{f}\left(\boldsymbol{\theta}\right)\geq\lambda I_{p}
\end{equation}
 then, for any $g\left(\boldsymbol{\theta}\right)$, the reverse KL
divergence is bounded:
\begin{equation}
KL\left(g,f\right)\leq\frac{1}{2\lambda}RI_{F}\left(g,f\right)
\end{equation}
\end{thm}

There are two possible ways to transform the LSI into a BvM result:
either we can assume that $f\left(\boldsymbol{\theta}\right)$ is
strongly log-concave or we can use the fact that a Gaussian distribution
is strongly log-concave. The following two propositions correspond
to these two direct applications of the LSI.
\begin{prop}
Loose/restrictive LSI bound on $KL\left(g,f\right)$. \label{prop:LSI-bound-on KL(g,f)}

If $f\left(\boldsymbol{\theta}\right)$ is strongly log-concave so
that there exists a strictly positive matrix $H_{min}$ such that:
\begin{equation}
\forall\boldsymbol{\theta}\ H\phi_{f}\left(\boldsymbol{\theta}\right)\geq H_{min}
\end{equation}
then, for any approximation $g\left(\boldsymbol{\theta}\right)$,
the reverse KL divergence is bounded:
\begin{equation}
KL\left(g,f\right)\leq\frac{1}{2}\E_{\boldsymbol{\theta}\sim g\left(\boldsymbol{\theta}\right)}\left[\left\{ \nabla\phi_{f}\left(\boldsymbol{\theta}\right)-\nabla\phi_{g}\left(\boldsymbol{\theta}\right)\right\} ^{T}\left(H_{min}\right)^{-1}\left\{ \nabla\phi_{f}\left(\boldsymbol{\theta}\right)-\nabla\phi_{g}\left(\boldsymbol{\theta}\right)\right\} \right]
\end{equation}
\end{prop}

\begin{prop}
Unwieldy LSI bound on $KL\left(f,g\right)$. \label{prop:LSI-bound-on KL(f,g)}

If $g\left(\boldsymbol{\theta}\right)$ is a Gaussian with covariance
$\Sigma$, then for any $f\left(\boldsymbol{\theta}\right)$ the forward
KL divergence is bounded:
\begin{equation}
KL\left(f,g\right)\leq\frac{1}{2}\E_{\boldsymbol{\theta}\sim f\left(\boldsymbol{\theta}\right)}\left[\left\{ \nabla\phi_{f}\left(\boldsymbol{\theta}\right)-\nabla\phi_{g}\left(\boldsymbol{\theta}\right)\right\} ^{T}\Sigma\left\{ \nabla\phi_{f}\left(\boldsymbol{\theta}\right)-\nabla\phi_{g}\left(\boldsymbol{\theta}\right)\right\} \right]
\end{equation}
\end{prop}

Both of these propositions offer an interesting alternative to the
KL divergence which can be useful to build heuristic understanding
of the BvM result but they fail to be useful in practice.

Prop.\ref{prop:LSI-bound-on KL(g,f)} falls short due to the severity
of assuming that the target density $f\left(\boldsymbol{\theta}\right)$,
i.e. the posterior in a Bayesian context, is strongly log-concave.
Indeed, in order for the bound of Prop.\ref{prop:LSI-bound-on KL(g,f)}
to be tight, the minimal curvature $H_{min}$ needs to be comparable
to the curvature at the MAP value: $H\phi_{f}\left(\boldsymbol{\theta}_{MAP}\right)$,
or more generally to be representative of the typical curvature. For
a typical posterior distribution, these quantities widely differ with
the mode being much more peaked than any other region of the posterior
distribution, or at least much more peaked than the region with minimum
curvature.

Prop.\ref{prop:LSI-bound-on KL(f,g)} falls short for very different
reasons. Indeed, it comes with absolutely no restrictions on the target
density $f\left(\boldsymbol{\theta}\right)$. However, it bounds the
KL divergence with an expected value of a complicated quantity under
the target $f\left(\boldsymbol{\theta}\right)$. Dealing with an expected
value under $f\left(\boldsymbol{\theta}\right)$ is as complicated
as tackling the normalization constant $Z_{f}$ for theoretical or
computational purposes which makes Prop.\ref{prop:LSI-bound-on KL(f,g)}
inapplicable.

\section{Gaussian approximations of simply log-concave distributions\label{sec:Gaussian-approximations-of log-concave distributions}}

This section first details the assumptions required and then presents
the main result of this article, Theorem \ref{thm:A-deterministic-BvM theorem},
establishing an approximation of $KL\left(g,f\right)$ that is well-suited
to computational and theoretical investigations of the quality of
the Laplace approximation. Finally, I show how to recover the classical
IID BvM Theorem as a corollary.

\subsection{Assumptions\label{subsec:Assumptions}}

In this first subsection, I present the assumptions required for Theorem
\ref{thm:A-deterministic-BvM theorem} to hold and discuss the reasons
why they are necessary. I differ the discussion of their applicability
in a Bayesian context to Section \ref{sec:Discussion}.

In order for the Laplace approximation $g_{LAP}\left(\boldsymbol{\theta}\right)$
to be close to $f\left(\boldsymbol{\theta}\right)$, we need assumptions
that control $f\left(\boldsymbol{\theta}\right)$ at two qualitatively
different levels.

First, we need global control over the shape of $f\left(\boldsymbol{\theta}\right)$
so that we can avoid trivial counter-examples such as the following
mixture of two Gaussians where the second Gaussian has a very high-variance:
\begin{equation}
f\left(\boldsymbol{\theta}\right)=\frac{1}{2}\Phi\left(\boldsymbol{\theta};0,I_{p}\right)+\frac{1}{2}\Phi\left(\boldsymbol{\theta};0,10^{100}I_{p}\right)
\end{equation}
where $\Phi$ denotes the density of the Gaussian with a given mean
and covariance. Around the mode $\boldsymbol{\mu}=0$, the density
$f\left(\boldsymbol{\theta}\right)$ is completely dominated by the
component with the small variance and the Laplace approximation would
completely ignore the second component. The Laplace approximation
would then miss half of the mass of $f\left(\boldsymbol{\theta}\right)$
and would thus provide a very poor approximation. Less artificial
examples are straightforward to construct with fat tails instead of
a mixture component.

In the present article, I propose to achieve global control by assuming
that $f\left(\boldsymbol{\theta}\right)$ is log-concave. Log-concave
distributions are not only unimodal but they have tails that decay
at least exponentially. As such, it is thus impossible for a lot of
mass of $f\left(\boldsymbol{\theta}\right)$ to be hidden in the tails
and the Laplace approximation thus necessarily captures the global
shape of $f\left(\boldsymbol{\theta}\right)$. Please see \citet{saumard2014log}
for a thorough review of their properties.

However, global control is not enough. Since the Laplace approximation
is based on a Taylor expansion of $\phi_{f}\left(\boldsymbol{\theta}\right)$
to second order, we need local control on the regularity of $\phi_{f}\left(\boldsymbol{\theta}\right)$.
We will do so by assuming that $\phi_{f}$ is differentiable three
times and by controlling the third derivative of $\phi_{f}\left(\boldsymbol{\theta}\right)$
which is an order-3 tensor: a linear operator which takes as input
three vectors $v_{1},v_{2},v_{3}$ and returns a scalar:
\begin{equation}
\phi_{f}^{\left(3\right)}\left(\boldsymbol{\theta}\right)\left[v_{1},v_{2},v_{3}\right]=\frac{\partial^{3}}{\partial\alpha_{1}\partial\alpha_{2}\partial\alpha_{3}}\phi_{f}\left(\boldsymbol{\theta}+\alpha_{1}v_{1}+\alpha_{2}v_{2}+\alpha_{3}v_{3}\right)
\end{equation}

The size of the third derivative can be measured by using the max
norm which measures the maximum relative size of the inputted vectors
and outputted scalar:
\begin{equation}
\left\Vert \phi_{f}^{\left(3\right)}\left(\boldsymbol{\theta}\right)\right\Vert _{\text{max}}=\max_{v_{1},v_{2},v_{3}}\frac{\left|\phi_{f}^{\left(3\right)}\left(\boldsymbol{\theta}\right)\left[v_{1},v_{2},v_{3}\right]\right|}{\left\Vert v_{1}\right\Vert _{2}\left\Vert v_{2}\right\Vert _{2}\left\Vert v_{3}\right\Vert _{2}}
\end{equation}

The smaller the third-derivative, the more accurate the Taylor expansion
of $\phi_{f}\left(\boldsymbol{\theta}\right)$ to second order. However,
a key technical point is that the absolute size of the third-derivative
is not what matters. What matters instead is the relative size of
the third-derivative to the second, and more precisely to the Hessian
matrix at the mode: $\Sigma^{-1}=H\phi_{f}\left(\boldsymbol{\mu}\right)$.
One possibility to measure this relative size is to compute the third-derivative
of the log-density on the standardized space $\tilde{\boldsymbol{\theta}}$:
$\tilde{\phi}_{f}\left(\tilde{\boldsymbol{\theta}}\right)$. Theorem
\ref{thm:A-deterministic-BvM theorem} establishes that the maximum
of the max-norms as we vary $\boldsymbol{\theta}$ is the key component
that controls the distance between $f\left(\boldsymbol{\theta}\right)$
and $g_{LAP}\left(\boldsymbol{\theta}\right)$. We will denote this
key quantity as $\Delta_{3}$:
\begin{align}
\Delta_{3} & =\max_{\boldsymbol{\theta}}\left\Vert \tilde{\phi}_{f}^{\left(3\right)}\left(\boldsymbol{\theta}\right)\right\Vert _{\text{max}}\label{eq: controling the third derivative}\\
 & =\max_{\boldsymbol{\theta},v_{1},v_{2},v_{3}}\frac{\left|\phi_{f}^{\left(3\right)}\left(\boldsymbol{\theta}\right)\left[\Sigma^{-1/2}v_{1},\Sigma^{-1/2}v_{2},\Sigma^{-1/2}v_{3}\right]\right|}{\left\Vert v_{1}\right\Vert _{2}\left\Vert v_{2}\right\Vert _{2}\left\Vert v_{3}\right\Vert _{2}}
\end{align}

Note that it is already NP-hard to compute exactly the max norm of
an order 3 tensor (\citet{hillar2013most}). $\Delta_{3}$ is thus
potentially computationally tricky to compute. However, this value
is only required in order to control the higher-order error terms
and, for computational concerns, we will be able to focus instead
on the dominating term. The precise value of $\Delta_{3}$ can then
be bounded very roughly or ignored.

\subsection{A general deterministic bound}

\subsubsection{Structure of the proof}

Given the two assumptions that the target density is log-concave and
regular, as measured by the scalar $\Delta_{3}$, it is possible to
precisely bound the KL divergence $KL\left(g_{LAP},f\right)$. The
bound is actually derived from a direct application of Prop.\ref{prop: KL approx KL_var}
and Prop.\ref{prop:LSI-bound-on KL(g,f)}, despite their flaws. This
subsection gives an overview of the structure of the proof. Full details
are presented in the appendix.

The key step is the use of a tricky change of variables from $\boldsymbol{\theta}$
to the standardized variable $\tilde{\boldsymbol{\theta}}$ and finally
to spherical coordinates in $\tilde{\boldsymbol{\theta}}$:
\begin{align}
\boldsymbol{\theta} & =\boldsymbol{\mu}+\Sigma^{1/2}\tilde{\boldsymbol{\theta}}\\
\tilde{\boldsymbol{\theta}} & =r\boldsymbol{e} & r\in\R^{+},\ \boldsymbol{e}\in S^{p-1}
\end{align}
where $\boldsymbol{e}$ takes values over the $p$-dimensional unit
sphere $S^{p-1}$.

Under the distribution $g_{LAP}$, the pair of random variables $r_{g},\boldsymbol{e}_{g}$
is simple. Indeed, it is a basic result of statistics that they are
independent and that $\boldsymbol{e}$ is a uniform random variable
over $S^{p-1}$ while $r_{g}$ is a $\chi_{p}$ random variable (Appendix
Lemma \ref{lem: APPENDIX Distribution-of-r_g,e_g}). Expected values
under $g_{LAP}$ are thus still simple to compute in the $\left(r,\boldsymbol{e}\right)$
parameterization. Furthermore, the KL divergence has the nice feature
that it is invariant to bijective changes of variables. This change
of variable thus decomposes the KL divergence of interest: $KL\left(g_{LAP},f\right)$
into the KL divergence due to $\boldsymbol{e}$ and that due to $r$
(Appendix Lemma \ref{lem: APPENDIX Decomposition-of-KL(g,f)}). More
precisely, we have:
\begin{equation}
KL\left(g_{LAP},f\right)=KL\left(\boldsymbol{e}_{g},\boldsymbol{e}_{f}\right)+\E_{\boldsymbol{e}\sim g_{LAP}\left(\boldsymbol{e}\right)}\left[KL\left(r_{g},r_{f}|\boldsymbol{e}\right)\right]
\end{equation}
where $KL\left(\boldsymbol{e}_{g},\boldsymbol{e}_{f}\right)$ is the
KL divergence between the marginal distributions of $\boldsymbol{e}$
under $g_{LAP}$ and $f$ and $KL\left(r_{g},r_{f}|\boldsymbol{e}\right)$
is the KL divergence between the conditional of distribution of $r$
under $g_{LAP}$ and $f$.

This decomposition already resurrects Proposition \ref{prop: KL approx KL_var}.
Indeed, the random direction $\boldsymbol{e}$ takes values over a
compact region of space. There is thus going to be a maximum for the
difference between the log-densities $\log\left[f\left(\boldsymbol{e}\right)\right]-\log\left[g_{LAP}\left(\boldsymbol{e}\right)\right]$
which means that we will be able to approximate $KL\left(\boldsymbol{e}_{g},\boldsymbol{e}_{f}\right)$
using the KL variance $KL_{var}\left(\boldsymbol{e}_{g},\boldsymbol{e}_{f}\right)$
and precisely bound the error.

Computation of the $KL_{var}$ is further simplified because $g\left(\boldsymbol{e}\right)$
is constant. Thus, the variance only comes from $f\left(\boldsymbol{e}\right)$.
However, computing $\log\left[f\left(\boldsymbol{e}\right)\right]$
would be slightly tricky here since that would involve a slightly
unwieldy integral against $r$. An additional useful trick thus consists
in replacing $\log\left[f\left(\boldsymbol{e}\right)\right]$ with
its Evidence Lower Bound (ELBO) computed under the approximation distribution
$g_{LAP}\left(r\right)$:
\begin{align}
\log\left[f\left(\boldsymbol{e}\right)\right] & \approx ELBO\left(\boldsymbol{e}\right)\\
ELBO\left(\boldsymbol{e}\right) & =-\E_{r\sim g_{LAP}\left(r\right)}\left[\tilde{\phi}_{f}\left(r\boldsymbol{e}\right)-\frac{1}{2}r^{2}\right]+C
\end{align}
where $C$ is a constant that does not depend on $\boldsymbol{e}$
and thus does not affect the KL variance. This error of this approximation
of $\log\left[f\left(\boldsymbol{e}\right)\right]$ is precisely equal
to $KL\left(r_{g},r_{f}|\boldsymbol{e}\right)$ which we now turn
to bounding.

The KL divergence between $r_{g}$ and $r_{f}|\boldsymbol{e}$ is
still challenging. While $r_{g}$ is strongly log-concave, $r_{f}|\boldsymbol{e}$
cannot be guaranteed to be more than simply log-concave without additional
highly-damaging assumptions. This is due to the tail of $r_{f}|\boldsymbol{e}$
where the curvature could tend to $0$. An additional trick comes
into play here: changing the variable from the radius $r$ into its
cubic-root $c$:
\begin{equation}
r=c^{3}
\end{equation}
Once more, this does not modify the KL divergence, but it does modify
the properties of the problem in interesting ways. Indeed, we can
now prove that $f\left(c|\boldsymbol{e}\right)$ is necessarily strongly
log-concave and that its minimum log-curvature tends to the minimum
log-curvature of $g_{LAP}\left(c\right)$. This unlocks applying the
LSI in the interesting direction of Proposition \ref{prop:LSI-bound-on KL(g,f)}
thus yielding a bound on the KL divergence:
\begin{align}
KL\left(r_{g},r_{f}|\boldsymbol{e}\right) & =KL\left(c_{g},c_{f}|\boldsymbol{e}\right)\\
 & \leq\frac{1}{2\lambda}RI_{F}\left[g\left(c\right),f\left(c|\boldsymbol{e}\right)\right]
\end{align}

\subsubsection{Statement of the Theorem}

By combining the approximations of both halves of $KL\left(g_{LAP},f\right)$,
we obtain the following Theorem.
\begin{thm}
A deterministic BvM theorem.\label{thm:A-deterministic-BvM theorem}

For any log-concave distribution $f\left(\boldsymbol{\theta}\right)$
supported on $\R^{p}$ and its Laplace approximation $g_{LAP}\left(\boldsymbol{\theta}\right)$,
in the asymptote $\Delta_{3}p^{3/2}\rightarrow0$ the reverse KL divergence
scales as:
\begin{align}
KL\left(g_{LAP},f\right) & =KL\left(\boldsymbol{e}_{g},\boldsymbol{e}_{f}\right)+\E_{\boldsymbol{e}\sim g_{LAP}\left(\boldsymbol{e}\right)}\left[KL\left(r_{g},r_{f}|\boldsymbol{e}\right)\right]\\
KL\left(r_{g},r_{f}|\boldsymbol{e}\right) & \leq\mathcal{O}\left(\left(\Delta_{3}\right)^{2}p^{2}\right)\\
 & \leq\frac{1}{2p^{2/3}}\E_{r\sim g\left(r\right)}\left[r^{4/3}\left(\boldsymbol{e}^{T}\nabla\tilde{\phi}_{f}\left(r\boldsymbol{e}\right)-r\right)^{2}\right]+\mathcal{O}\left(\left(\Delta_{3}\right)^{3}p^{5/2}\right)\\
KL\left(\boldsymbol{e}_{g},\boldsymbol{e}_{f}\right) & =\mathcal{O}\left(\left(\Delta_{3}\right)^{2}p^{3}\right)\\
 & =\frac{1}{2}\text{Var}_{\boldsymbol{e}\sim g_{LAP}\left(\boldsymbol{e}\right)}\left[\E_{r\sim g_{LAP}\left(r\right)}\left[\tilde{\phi}_{f}\left(r\boldsymbol{e}\right)-\frac{1}{2}r^{2}\right]\right]+\mathcal{O}\left(\left(\Delta_{3}\right)^{3}p^{9/2}\right)\\
 & \leq\frac{1}{2}KL_{var}\left(g_{LAP},f\right)+\mathcal{O}\left(\left(\Delta_{3}\right)^{3}p^{9/2}\right)
\end{align}
The reverse KL thus goes to $0$ as:
\begin{equation}
KL\left(g_{LAP},f\right)=\mathcal{O}\left(\left(\Delta_{3}\right)^{2}p^{3}\right)
\end{equation}
\end{thm}

This theorem establishes an upper-bound on the rate at which the KL
divergence goes to $0$ in the large data limit. Critically, this
convergence is slower for higher-dimensional probability distributions.
It is unclear whether this rate is optimal or pessimistic. We discuss
this point further in section \ref{sec:Discussion} where we give
an example which scales at this rate but is much less regular than
assumed by the theorem since it is continuous in $r$ but not in $\boldsymbol{e}$.
However, a careful investigation of the proof reveals that the the
proof of the theorem does not make use of the $\boldsymbol{e}$-continuity
of $f$ and this example is thus still very interesting in establishing
one worst case of the theorem.

Another important aspect of Theorem \ref{thm:A-deterministic-BvM theorem}
to note is that the first order term of the error is only expressed
as expected values under $g_{LAP}$. It is thus straightforward to
approximate this term by sampling from $g_{LAP}$. I take advantage
of this property in section \ref{sec:Computable-approximations-of KL}
in order to give computable asymptotically-correct approximations
of the KL divergence.

\subsection{Recovering the classical IID theorem\label{subsec:Recovering-the-classical IID theorem}}

Theorem \ref{thm:A-deterministic-BvM theorem} can be used to recover
the classical BvM Theorem that we have presented in Section \ref{subsec: intro to The-Bernstein-von-Mises}.
I only briefly sketch the analysis here while I provide a fully rigorous
proof in Appendix section \ref{sec: APPENDIX General-deterministic-Bound}.

First, we need to extend the setup of Section \ref{subsec: intro to The-Bernstein-von-Mises}
so that the assumptions of Theorem \ref{thm:A-deterministic-BvM theorem}
hold. This only requires two additional assumptions. First, we need
the negative log-prior and all negative log-likelihoods $NLL_{i}$
to be strictly convex so that the posterior is guaranteed to be log-concave.
Second, we need to control the size of the third log-derivative of
the prior and the typical size of the third derivative of the negative
log-likelihoods $NLL_{i}$.

More precisely, let $\boldsymbol{\theta}_{0}$ denote the pseudo-true
parameter value, and let $J$ be the Hessian-based version of the
Fisher information:
\begin{equation}
J=\E_{NLL}\left[H\ NLL\left(\boldsymbol{\theta}_{0}\right)\right]
\end{equation}
Then, let $\Delta^{\left(i\right)}$ denote the (random) maximum of
the third derivative of $NLL_{i}$, normalized by $J^{1/2}$:
\begin{equation}
\Delta^{\left(i\right)}=\max_{\boldsymbol{\theta},v_{1},v_{2},v_{3}}\frac{\left|NLL_{i}^{\left(3\right)}\left(\boldsymbol{\theta}\right)\left[J^{1/2}v_{1},J^{1/2}v_{2},J^{1/2}v_{3}\right]\right|}{\left\Vert v_{1}\right\Vert _{2}\left\Vert v_{2}\right\Vert _{2}\left\Vert v_{3}\right\Vert _{2}}
\end{equation}
Note that this normalization by $J$ is comparable to the normalization
by $\Sigma^{-1/2}$ in the definition of $\Delta_{3}$ (eq.\ref{eq: controling the third derivative})
since asymptotically we will have $\Sigma^{-1}\approx nJ$. We will
assume that $\E\left(\Delta^{\left(i\right)}\right)$ is finite. We
will similarly assume that the equivalent value computed on the prior
term $\phi_{0}$ is finite:
\begin{equation}
\Delta^{\left(0\right)}=\max_{\boldsymbol{\theta},v_{1},v_{2},v_{3}}\frac{\left|\phi_{0}^{\left(3\right)}\left(\boldsymbol{\theta}\right)\left[J^{1/2}v_{1},J^{1/2}v_{2},J^{1/2}v_{3}\right]\right|}{\left\Vert v_{1}\right\Vert _{2}\left\Vert v_{2}\right\Vert _{2}\left\Vert v_{3}\right\Vert _{2}}
\end{equation}

Under these assumptions, the MAP estimator constructed from the first
$n$ likelihoods, $\hat{\boldsymbol{\theta}}_{n}$ is a consistent
estimator of $\boldsymbol{\theta}_{0}$ and asymptotically Gaussian
with variance decaying at rate $n^{-1}$ (Appendix Lemma \ref{lem: APPENDIX MAP is consistent},
Appendix Prop.\ref{prop: APPENDIX MAP is gaussian}). As a consequence,
the variance of the $n^{th}$ Laplace approximation $g_{LAP,n}$ scales
as (using a law-of-large-numbers approximation of the sum):
\begin{align}
\left(\Sigma_{n}\right)^{-1} & =H\phi_{0}\left(\hat{\boldsymbol{\theta}}_{n}\right)+\sum_{i=1}^{n}H\ NLL_{i}\left(\hat{\boldsymbol{\theta}}_{n}\right)\\
 & =H\phi_{0}\left(\boldsymbol{\theta}_{0}\right)+\sum_{i=1}^{n}H\ NLL_{i}\left(\boldsymbol{\theta}_{0}\right)+\mathcal{O}_{P}\left(1\right)\\
 & =H\phi_{0}\left(\boldsymbol{\theta}_{0}\right)+nJ+\mathcal{O}_{P}\left(1\right)\\
 & =nJ+\mathcal{O}_{P}\left(1\right)
\end{align}
The scaled third log-derivative of the log-posterior then scales approximately
as (using the sub-additivity of maxima):
\begin{align}
\Delta_{3,n} & =\max_{\boldsymbol{\theta},v_{1},v_{2},v_{3}}\frac{\left|\phi_{n}^{\left(3\right)}\left(\boldsymbol{\theta}\right)\left[\Sigma_{n}^{-1/2}v_{1},\Sigma_{n}^{-1/2}v_{2},\Sigma_{n}^{-1/2}v_{3}\right]\right|}{\left\Vert v_{1}\right\Vert _{2}\left\Vert v_{2}\right\Vert _{2}\left\Vert v_{3}\right\Vert _{2}}\\
 & \approx n^{-3/2}\max_{\boldsymbol{\theta},v_{1},v_{2},v_{3}}\frac{\left|\phi_{n}^{\left(3\right)}\left(\boldsymbol{\theta}\right)\left[J^{1/2}v_{1},J^{1/2}v_{2},J^{1/2}v_{3}\right]\right|}{\left\Vert v_{1}\right\Vert _{2}\left\Vert v_{2}\right\Vert _{2}\left\Vert v_{3}\right\Vert _{2}}\\
 & \lesssim n^{-3/2}\left(\Delta^{\left(0\right)}+\sum_{i=1}^{n}\Delta^{\left(i\right)}\right)
\end{align}

From a law-of-large-numbers argument, the sum $\sum_{i=1}^{n}\Delta^{\left(i\right)}$
scales approximately as $n\E\left(\Delta^{\left(i\right)}\right)$
and $\Delta_{3,n}$ thus scales at most as $n^{-1/2}$. It is furthermore
straightforward to establish through a law-of-large-numbers argument
that this rate cannot generally be improved if any $\boldsymbol{\theta}$
is such that $\E\left(NLL_{i}^{\left(3\right)}\right)$ is non-zero.

The following corollary of Theorem \ref{thm:A-deterministic-BvM theorem}
summarizes this chain of reasoning.
\begin{cor}
Bernstein-von Mises: IID case.\label{cor:Bernstein-von-Mises:-IID case}

Let $f_{n}\left(\boldsymbol{\theta}\right)$ be the posterior distribution:
\begin{equation}
f_{n}\left(\boldsymbol{\theta}\right)\propto\exp\left(-\phi_{0}\left(\boldsymbol{\theta}\right)-\sum_{i=1}^{n}NLL_{i}\left(\boldsymbol{\theta}\right)\right)
\end{equation}
where the $NLL_{i}$ are IID function-valued random variables. Let
$g_{LAP,n}\left(\boldsymbol{\theta}\right)$ be the Laplace approximation
of $f_{n}\left(\boldsymbol{\theta}\right)$ and $J$ be the Hessian-based
Fisher information:
\begin{equation}
J=\E_{NLL}\left[H\ NLL\left(\boldsymbol{\theta}_{0}\right)\right]
\end{equation}

If $J>0$ and $\phi_{0}$ and all $NLL_{i}$ are log-concave and have
controlled third-derivatives:
\begin{align}
\Delta^{\left(0\right)} & =\max_{\boldsymbol{\theta},v_{1},v_{2},v_{3}}\frac{\left|\phi_{0}^{\left(3\right)}\left(\boldsymbol{\theta}\right)\left[J^{1/2}v_{1},J^{1/2}v_{2},J^{1/2}v_{3}\right]\right|}{\left\Vert v_{1}\right\Vert _{2}\left\Vert v_{2}\right\Vert _{2}\left\Vert v_{3}\right\Vert _{2}}\label{eq: control on the third derivatives of the prior}\\
\Delta^{\left(i\right)} & =\max_{\boldsymbol{\theta},v_{1},v_{2},v_{3}}\frac{\left|NLL_{i}^{\left(3\right)}\left(\boldsymbol{\theta}\right)\left[J^{1/2}v_{1},J^{1/2}v_{2},J^{1/2}v_{3}\right]\right|}{\left\Vert v_{1}\right\Vert _{2}\left\Vert v_{2}\right\Vert _{2}\left\Vert v_{3}\right\Vert _{2}}
\end{align}
with $\Delta^{\left(0\right)}<\infty$ and $\E\left(\Delta^{\left(i\right)}\right)<\infty$.

Then, as $n\rightarrow\infty$, $\Delta_{3}p^{3/2}=\mathcal{O}_{P}\left(1/\sqrt{n}\right)$
and $f_{n}\left(\boldsymbol{\theta}\right)$ and $g_{LAP,n}\left(\boldsymbol{\theta}\right)$
converge to one-another as:
\begin{equation}
KL\left(g_{LAP,n},f_{n}\right)=\mathcal{O}_{P}\left(\frac{1}{n}\right)
\end{equation}
\end{cor}

Note that this line of reasoning is straightforward to extend to other
models for which the $NLL_{i}$ might have a more complicated dependence
structure. For example, consider linear models with a fixed design.
Under such models, the $NLL_{i}$ are independent but not identically
distributed. However, it is still straightforward to establish, as
long as the design is balanced so that no one observation retains
influence asymptotically, that $\left(\Sigma_{n}\right)^{-1}$ scales
linearly asymptotically and that $\Delta_{3,n}$ would scale as $n^{-1/2}$.
Similarly, the result could be extended to a situation in which the
$NLL_{i}$ have a time-series dependency. As long as there is a high-enough
degree of independence inside the chain, an Ergodic-theorem type argument
would yield that $\left(\Sigma_{n}\right)^{-1}$ scales linearly and
$\Delta_{3,n}$ as $n^{-1/2}$. More generally, I conjecture that
conditions yielding a Locally Asymptotically Normal likelihood (LAN;
\citet{van2000asymptotic}) and a log-concave posterior would result
in Theorem \ref{thm:A-deterministic-BvM theorem} being applicable.
I will address this question is future work.

\section{Computable approximations of $KL\left(g_{LAP},f\right)$\label{sec:Computable-approximations-of KL}}

In this section, I detail how to approximate the KL divergence $KL\left(g_{LAP},f\right)$
by using Theorem \ref{thm:A-deterministic-BvM theorem}. I further
demonstrate the bounds in a simple example: logistic regression.

\subsection{Proposed approximations}

At face value, Theorem \ref{thm:A-deterministic-BvM theorem} offers
two main approximations:
\begin{align}
KL\left(r_{g},r_{f}|\boldsymbol{e}\right) & \apprle LSI\left(r_{g},r_{f}|\boldsymbol{e}\right)=\frac{1}{2p^{2/3}}\E_{r\sim g\left(r\right)}\left[r^{4/3}\left(\boldsymbol{e}^{T}\nabla\tilde{\phi}_{f}\left(r\boldsymbol{e}\right)-r\right)^{2}\right]\\
KL\left(\boldsymbol{e}_{g},\boldsymbol{e}_{f}\right) & \approx\frac{1}{2}\text{Var}_{\boldsymbol{e}\sim g_{LAP}\left(\boldsymbol{e}\right)}\left[\E_{r\sim g_{LAP}\left(r\right)}\left[\tilde{\phi}_{f}\left(r\boldsymbol{e}\right)-\frac{1}{2}r^{2}\right]\right]
\end{align}
which can be combined to yield an approximation of $KL\left(g_{LAP},f\right)$
which we will denote as ``LSI+VarELBO'' (recall that $ELBO\left(\boldsymbol{e}\right)=\E_{r\sim g_{LAP}\left(r\right)}\left[\tilde{\phi}_{f}\left(r\boldsymbol{e}\right)-\frac{1}{2}r^{2}\right]+C$).

We can similarly use the $KL_{var}$-based bound of $KL\left(\boldsymbol{e}_{g},\boldsymbol{e}_{f}\right)$:
\begin{equation}
KL\left(\boldsymbol{e}_{g},\boldsymbol{e}_{f}\right)\apprle\frac{1}{2}KL_{var}\left(g_{LAP},f\right)
\end{equation}
to obtain another approximation of $KL\left(g_{LAP},f\right)$ which
we will denote as ``LSI+KLvar''.

However, observe that the KL variance is related to the VarELBO term
$\text{Var}_{\boldsymbol{e}\sim g_{LAP}\left(\boldsymbol{e}\right)}\left[\E_{r\sim g_{LAP}\left(r\right)}\left[\tilde{\phi}_{f}\left(r\boldsymbol{e}\right)-\frac{1}{2}r^{2}\right]\right]$
as:
\begin{align}
KL_{var}\left(g_{LAP},f\right) & =\text{Var}_{\boldsymbol{e}\sim g_{LAP}\left(\boldsymbol{e}\right)}\left[\E_{r\sim g_{LAP}\left(r\right)}\left[\tilde{\phi}_{f}\left(r\boldsymbol{e}\right)-\frac{1}{2}r^{2}\right]\right]\nonumber \\
 & \ \ \ \ +\E_{\boldsymbol{e}\sim g_{LAP}\left(\boldsymbol{e}\right)}\left[\text{Var}_{r\sim g_{LAP}\left(r\right)}\left[\tilde{\phi}_{f}\left(r\boldsymbol{e}\right)-\frac{1}{2}r^{2}\right]\right]
\end{align}
where the second term corresponds to the KL divergence due to the
$r$-variables (Appendix Lemma \ref{lem:Relationship-between-Var(ELBO) and KL-var}).
The ``LSI+KLvar'' approximation thus counts twice the contribution
of the $r$-variables to $KL\left(g_{LAP},f\right)$ which seems silly.
Theorem \ref{thm:A-deterministic-BvM theorem} thus offer heuristic\textbf{
}support for the idea of using directly the KL variance as an approximation
of $KL\left(g_{LAP},f\right)$ as suggested by Proposition \ref{prop:LSI-bound-on KL(g,f)}
. I must emphasize again that neither of those results establish rigorously
that this is correct but this turns out to be a very accurate of $KL\left(g_{LAP},f\right)$
in my experiments.

Computing these bounds requires the evaluation of expected values
of $\phi_{f}\left(\boldsymbol{\theta}_{g}\right)$ which can be performed
through sampling from $g_{LAP}\left(\boldsymbol{\theta}\right)$.
Further simplification can be achieved by replacing $\tilde{\phi}_{f}\left(\tilde{\boldsymbol{\theta}}\right)$
by its Taylor approximation to third or fourth order. For the KL variance
approximation, this yields immediately an explicit expected value.
For the LSI and the VarELBO approximations, we need to further perform
the rough approximation $r_{g}\approx p^{1/2}$ (accurate in the large
$p$ limit; Appendix Lemma \ref{lem: APPENDIX Moments-of-r_g}) in
order to simplify the corresponding expected values.

These various approximations are summarized in the following corollary
of Theorem \ref{thm:A-deterministic-BvM theorem}.
\begin{cor}
Computable approximations of $KL\left(g_{LAP},f\right)$. \label{cor:Computable-approximations-of KL(g,f)}

In an asymptote where $\Delta_{3}p^{3/2}\rightarrow0$, the KL divergence
$KL\left(g_{LAP},f\right)$ can be approximated as (in order from
most interesting to least):
\begin{align}
KL\left(g_{LAP},f\right) & \approx\frac{1}{2}KL_{var}\left(g_{LAP},f\right)\\
 & \approx LSI+\frac{1}{2}KL_{var}\left(g_{LAP},f\right)\\
 & \approx LSI+\frac{1}{2}VarELBO
\end{align}
where:
\begin{align}
KL_{var}\left(g_{LAP},f\right) & =\text{Var}_{\boldsymbol{\theta}\sim g_{LAP}\left(\boldsymbol{\theta}\right)}\left[\phi_{f}\left(\boldsymbol{\theta}\right)-\phi_{g}\left(\boldsymbol{\theta}\right)\right]\\
LSI & =\frac{1}{2p^{2/3}}\E_{r\sim g\left(r\right)}\left[r^{4/3}\left(\boldsymbol{e}^{T}\nabla\tilde{\phi}_{f}\left(r\boldsymbol{e}\right)-r\right)^{2}\right]\\
VarELBO & =\text{Var}_{\boldsymbol{e}\sim g_{LAP}\left(\boldsymbol{e}\right)}\left[\E_{r\sim g_{LAP}\left(r\right)}\left[\tilde{\phi}_{f}\left(r\boldsymbol{e}\right)-\frac{1}{2}r^{2}\right]\right]
\end{align}
All expected values and variances need to be further approximated
by sampling from $g_{LAP}\left(\boldsymbol{\theta}\right)$.

Alternatively, a Taylor approximation of $\tilde{\phi}_{f}\left(\boldsymbol{\theta}\right)$
around $0$ and the rough approximation $r_{g}\approx p^{1/2}$ yields:
\begin{align}
KL_{var}\left(g_{LAP},f\right) & \approx VarELBO\\
 & \approx\frac{1}{6}\sum_{i,j,k}\left\{ \left[\tilde{\phi}_{f}^{\left(3\right)}\left(0\right)\right]_{i,j,k}\right\} ^{2}+\frac{1}{4}\sum_{i,j,k}\left[\tilde{\phi}_{f}^{\left(3\right)}\left(0\right)\right]_{i,j,j}\left[\tilde{\phi}_{f}^{\left(3\right)}\left(0\right)\right]_{i,k,k}\nonumber \\
 & \ \ \ \ +\frac{1}{24}\sum_{i,j,k,l}\left\{ \left[\tilde{\phi}_{f}^{\left(4\right)}\left(0\right)\right]_{i,j,k,l}\right\} ^{2}+\frac{1}{8}\sum_{i,j,k,l}\left\{ \left[\tilde{\phi}_{f}^{\left(4\right)}\left(0\right)\right]_{i,j,k,k}\right\} \left\{ \left[\tilde{\phi}_{f}^{\left(4\right)}\left(0\right)\right]_{i,j,l,l}\right\} \\
LSI & \approx\frac{1}{p}\Bigg[\frac{3}{4}\sum_{i,j,k}\left\{ \left[\tilde{\phi}_{f}^{\left(3\right)}\left(0\right)\right]_{i,j,k}\right\} ^{2}+\frac{9}{8}\sum_{i,j,k}\left[\tilde{\phi}_{f}^{\left(3\right)}\left(0\right)\right]_{i,j,j}\left[\tilde{\phi}_{f}^{\left(3\right)}\left(0\right)\right]_{i,k,k}\nonumber \\
 & \ \ \ \ +\frac{1}{3}\sum_{i,j,k,l}\left\{ \left[\tilde{\phi}_{f}^{\left(4\right)}\left(0\right)\right]_{i,j,k,l}\right\} ^{2}+\sum_{i,j,k,l}\left\{ \left[\tilde{\phi}_{f}^{\left(4\right)}\left(0\right)\right]_{i,j,k,k}\right\} \left\{ \left[\tilde{\phi}_{f}^{\left(4\right)}\left(0\right)\right]_{i,j,l,l}\right\} \nonumber \\
 & \ \ \ \ +\frac{1}{8}\sum_{i,j,k,l}\left\{ \left[\tilde{\phi}_{f}^{\left(4\right)}\left(0\right)\right]_{i,i,j,j}\right\} \left\{ \left[\tilde{\phi}_{f}^{\left(4\right)}\left(0\right)\right]_{k,k,l,l}\right\} \Bigg]
\end{align}
\end{cor}

Critically, note that these approximations are only appropriate asymptotically
as $\Delta_{3}p^{3/2}\rightarrow0$ so that the additional error terms
vanish. This is critical to keep in mind when using these approximations,
as we discuss further in Section \ref{sec:Discussion}.

Further note that this result establishes that the LSI term is negligible
when $p$ is large since it scales as $p^{-1}$. This is unsurprising
since it corresponds to the contribution of the one-dimensional variable
$r$ to the overall KL divergence due to all coordinates of $\boldsymbol{\theta}\in\R^{p}$.

\subsection{Empirical results in the logistic classification model\label{subsec:Empirical-results-in}}

In order to validate the theoretical analysis, I checked that the
approximations of Cor.\ref{cor:Computable-approximations-of KL(g,f)}
would be appropriate in a logistic linear classification model under
a wide range of circumstances.

More precisely, I constructed artificial datasets with values of $p$
ranging from $10$ to $1000$ and $n$ ranging from $10$ to $5600$.
These values were chosen due to the memory limitations of the computer
on which the simulations were ran. In these datasets, all predictors
were IID with a Gaussian distribution (mean $0$, standard deviation
$1.5$) and class labels $Y\in\left\{ -1,1\right\} $ were generated
from the logistic model with:
\begin{equation}
\mathbb{P}\left(Y=1|\boldsymbol{X}\right)=\frac{1}{1+\exp\left(-\frac{1}{\sqrt{p}}\sum_{i=1}^{p}X_{i}\right)}
\end{equation}
This corresponds to a regression coefficient with constant coefficients
$\boldsymbol{\theta}_{0}=\left(1/\sqrt{p}\dots1/\sqrt{p}\right)$.
The scaling with $p$ (and the standard deviation of $X$) was chosen
so that the marginal distribution of $\mathbb{P}\left(Y|X\right)$
would be approximately spread over the range $0.3-0.7$, i.e. most
values of $X$ would be ambiguously classified even with perfect knowledge
of $\boldsymbol{\theta}_{0}$. This ensures that Fisher information
is somewhat high and that the asymptotic regime is reached quickly.

Once the data was generated, it was analyzed with the standard logistic
linear classification conditional model:
\begin{equation}
\mathbb{P}\left(Y=\pm1|\boldsymbol{X},\boldsymbol{\theta}\right)=\frac{1}{1+\exp\left(-Y\boldsymbol{\theta}^{T}\boldsymbol{X}\right)}
\end{equation}
Note that this likelihood function is log-concave.

Under the prior distribution, the coefficients $\theta_{i}$ were
modeled to be IID with standard deviation $1/\sqrt{p}$, encoding
the assumption that the norm of the coefficient vector should be of
order $1$, so that the marginal (in $X$) distribution of $\mathbb{P}\left(Y|X,\boldsymbol{\theta}\right)$
would be spread over the range $0.3-0.7$: 
\begin{align}
\left\Vert \boldsymbol{\theta}\right\Vert _{2} & =\sqrt{\sum_{i=1}^{p}\left(\theta_{i}\right)^{2}}\\
 & \sim\frac{1}{\sqrt{p}}\chi_{p}\\
 & \approx1
\end{align}
Note that this prior is more informative for a given coefficient in
high dimensions since the prior standard deviation tends to $0$ as
$p\rightarrow\infty$.

Given this model, I then computed the Laplace approximation using
a Newton-conjugate gradient (\citet{nocedal2006numerical}) implemented
in the Scipy Python library (\citet{scipy}). The optimization was
initialized at $\boldsymbol{\theta}_{0}$. The approximations of Cor.\ref{cor:Computable-approximations-of KL(g,f)}
were computed in a straightforward fashion by sampling from the Laplace
approximation $g_{LAP}\left(\boldsymbol{\theta}\right)$ and computing
the corresponding empirical means and variances.

Computing the true KL divergence was challenging. I first sampled
from $f\left(\boldsymbol{\theta}\right)$ using the NUTS algorithm
(\citet{hoffman2014no}) using an implementation from \href{https://github.com/mfouesneau/NUTS}{M. Fouesnau on github}.
The normalization constant was then approximated as:
\begin{equation}
\left[\int\exp\left(-\phi_{f}\left(\boldsymbol{\theta}\right)\right)\right]^{-1}=\frac{1}{s}\sum_{i=1}^{s}\frac{g\left(\boldsymbol{\theta}_{i}\right)}{\tilde{f}\left(\boldsymbol{\theta}_{i}\right)}
\end{equation}
 Once the normalization constant was approximated, the KL divergence
was computed by sampling from $g\left(\boldsymbol{\theta}\right)$.

Inspection of the KL divergence reveals the following interesting
features (Fig.\ref{fig: evolution of KL with n}.A). First, we observe
that, as the size of the dataset $n$ grows, the KL divergence initially
rises from a low value then reaches a plateau then decreases at speed
$n^{-1}$. The behavior for large $n$ is thus in accordance with
the behavior predicted by the asymptotic analysis (Cor.\ref{cor:Bernstein-von-Mises:-IID case}).
However, the KL divergence is much lower than predicted for small
$n$. Indeed, Note that Cor.\ref{cor:Bernstein-von-Mises:-IID case}
and conventional BvM results fail to identify that the KL divergence
is small not only in the asymptote $n\rightarrow\infty$ but also
when $n$ is close to $0$. In contrast, approximating the KL divergence
with $KL_{var}$ recovers the full complexity of the evolution of
the KL divergence with $n$ (Fig.\ref{fig: evolution of KL with n}.B).

\begin{figure}
\begin{centering}
\includegraphics[width=13cm]{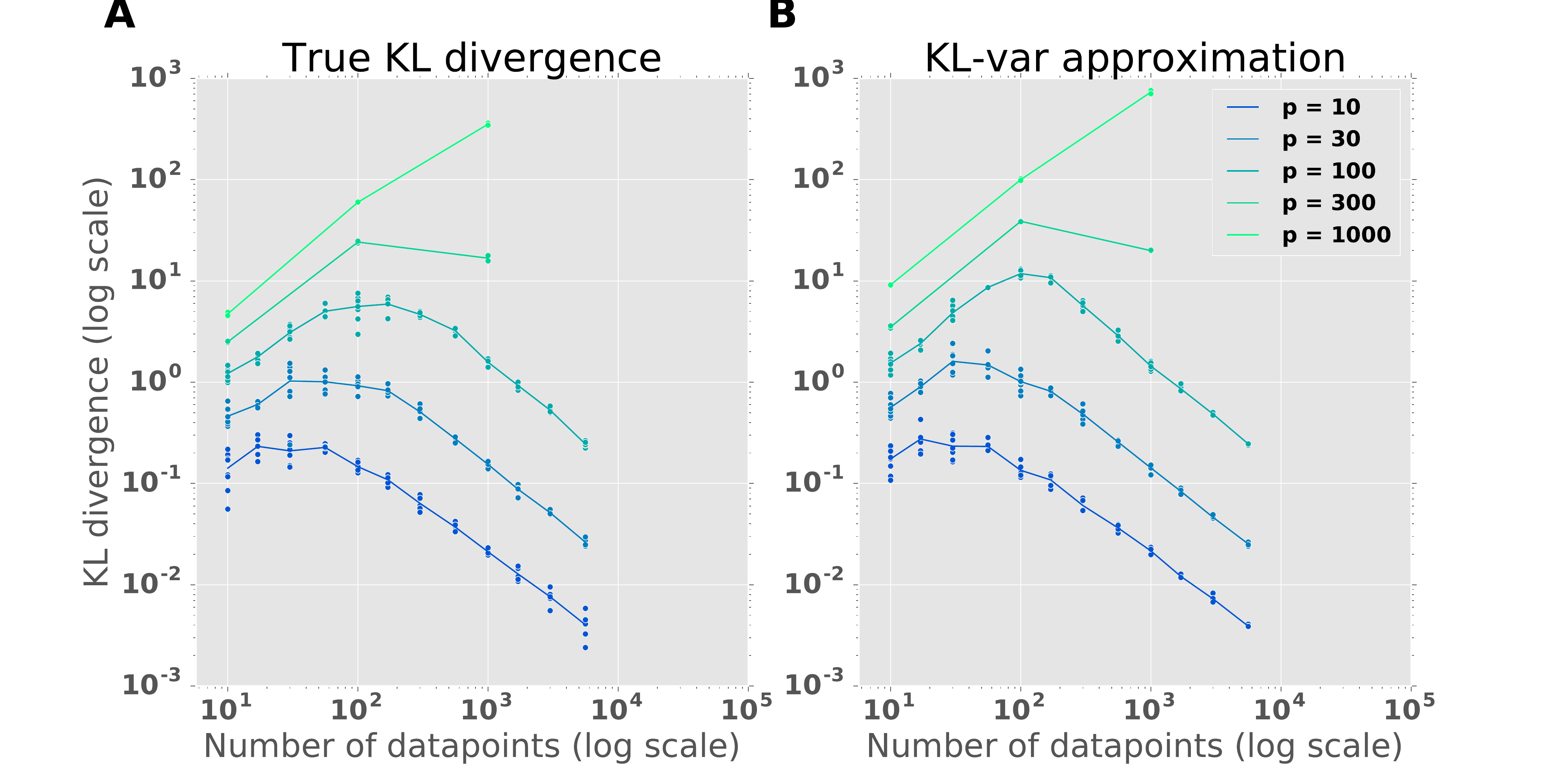}
\par\end{centering}
\caption{\label{fig: evolution of KL with n}Evolution of the KL divergence
with $n$ in a logistic regression model. A: evolution of the true
KL divergence $KL\left(g_{LAP},f\right)$ with the number of datapoints
$n$. Color indicates the dimension $p$ of the model. For every value
of $\left(n,p\right)$, multiple simulations were performed (dots).
The line corresponds to the average over these simulations. The KL
divergence asymptotes as $n^{-1}$ for large values of $n$, but is
also small for small values of $n$ when the posterior is almost equal
to the Gaussian prior. Usual asymptotic arguments only identify the
large $n$ asymptotic behavior. B: evolution of the $KL_{var}$ approximation
of $KL\left(g_{LAP},f\right)$ with $n$. The approximation successfully
recovers the full range of the dynamics of the KL divergence. Note
that the asymptote hasn't been reached yet for the larger values of
$p$.}
\end{figure}
Throughout the range of values I explored, the KL variance proved
to be a great approximation of the KL divergence (Fig.\ref{fig: quality of KL var and KL var + LSI}.A).
Indeed, the ratio of the KL-variance-based approximation of $KL\left(g_{LAP},f\right)$
to the truth is mostly close to $1$ and only has a maximum of $5$.

The ratio of the ``LSI+KLvar'' approximation to the true KL divergence
is similarly close to $1$(Fig.\ref{fig: quality of KL var and KL var + LSI}.B).
However, the ratio is always above $1$, indicating that this approximation
is always strictly bigger than the KL divergence, thus confirming
the approximate upper-bound status of this approximation.

\begin{figure}
\begin{centering}
\includegraphics[width=13cm]{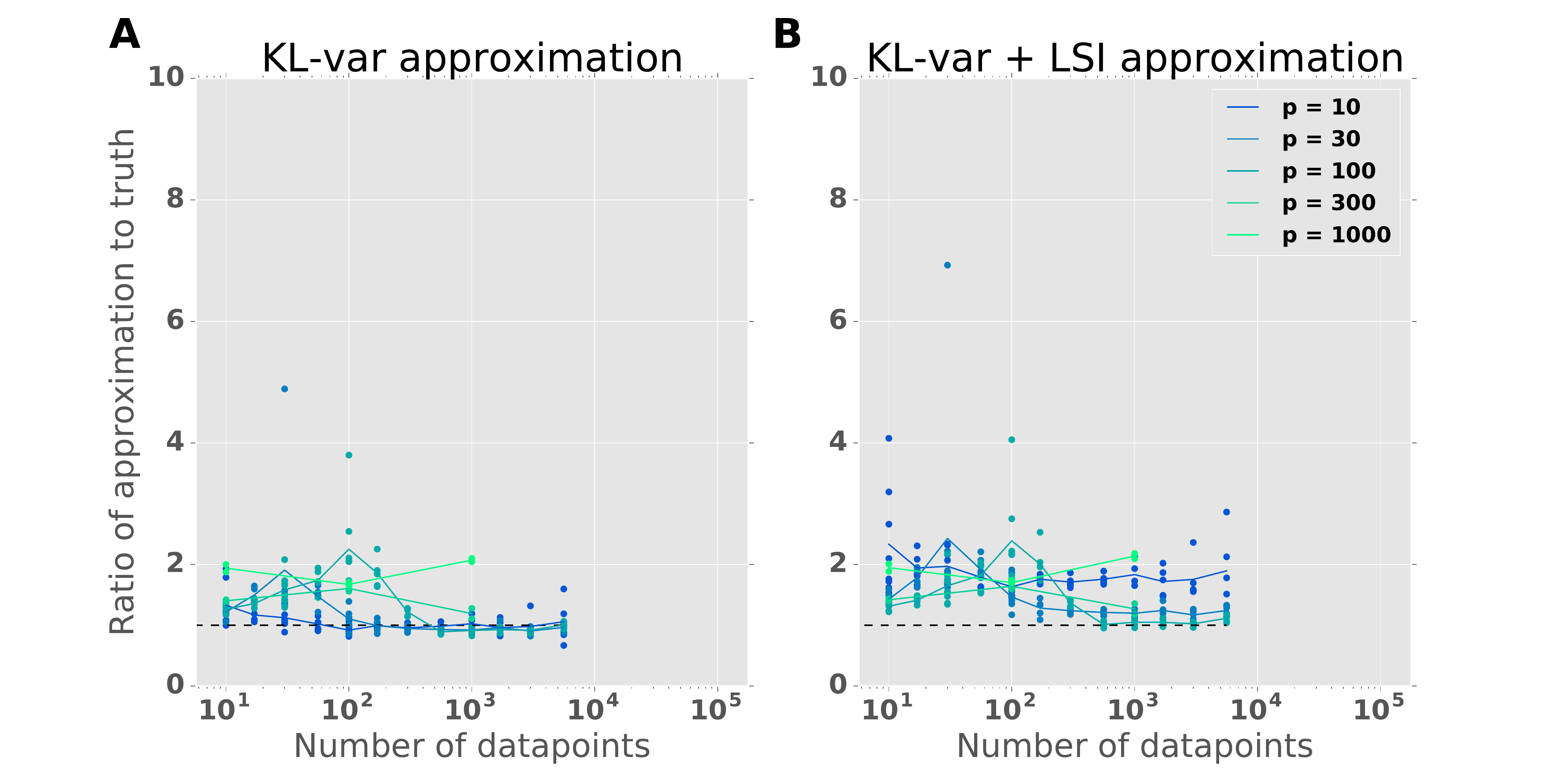}
\par\end{centering}
\caption{\label{fig: quality of KL var and KL var + LSI}Quality of the KL-var
and the \textquotedbl KL-var + LSI\textquotedbl{} approximations.
A: evolution of the ratio $\frac{1}{2}KL_{var}/KL$ as a function
of $n$. A ratio above $1$ indicates overestimation of the true KL
divergence. Throughout the range of parameters explored, $KL_{var}$
yields a tight approximation of the KL divergence. B: same as A for
the \textquotedbl KL-var + LSI\textquotedbl{} approximation of Cor.\ref{cor:Computable-approximations-of KL(g,f)}.
This approximation consistently upper-bounds the KL divergence in
the experiment. Note that the upper-bound behavior is more pronounced
for small values of $p$. This is unsurprising since the LSI term
is $p$-times smaller than the $KL_{var}$ term (Cor.\ref{cor:Computable-approximations-of KL(g,f)}).}
\end{figure}
Finally, the varELBO and Taylor-based approximations proved to be
very disappointing. The varELBO was mostly uncomputable when compared
to the alternatives. This is due to the fact that computing the varELBO
requires computing an empirical mean over the radius $r\sim g\left(r\right)$,
thus increasing the number of samples required by an order of magnitude.
However, for the small values of $n,p$ for which the varELBO remained
computable, it was almost equal to $KL_{var}$ (Appendix Fig.\ref{fig: APPENDIX evolution of all approximations}.C).

Taylor-based approximations are similarly very expensive to compute
since they require the manipulation of the $d^{3}$ coefficients of
$\tilde{\phi}_{f}^{\left(3\right)}\left(0\right)$ and $d^{4}$ coefficients
of $\tilde{\phi}_{f}^{\left(4\right)}\left(0\right)$. This proved
doable for $p\in\left\{ 10,30,100\right\} $. However, in this range,
the Taylor-based approximations were quite disappointing in comparison
with the corresponding approximations based on empirical means. Indeed,
most Taylor-based approximation were five to ten times bigger than
the truth for small values of $n$ and only became good approximations
when the large-data limit behavior dominated. This is probably due
to the fact that Taylor approximations of $\tilde{\phi}_{f}\left(\boldsymbol{\theta}\right)$
are unrepresentative of its shape for small values of $n$ due to
them being too local in this regime (Appendix Fig.\ref{fig: APPENDIX evolution of all approximations}.D-G).

Overall, the KL variance and the ``LSI+KLvar'' approximations appear
to yield great approximations of $KL\left(g_{LAP},f\right)$ in the
examples considered here. Indeed, both are always in the correct order
of magnitude. Furthermore, the KL-variance approximation is close
to being exact in both the small $n$ and large $n$ regimes. In contrast,
the varELBO approximations and the Taylor-based approximations appear
to be dominated. Indeed, they are computationally trickier while yielding
worse approximations.

\section{Discussion\label{sec:Discussion}}

In this article, I have presented a new Bernstein-von Mises theorem
which aims at being more computationally relevant than existing results.
I have shown that, in a limit where a log-concave probability distribution
has higher derivatives that are small compared to the second derivative,
the Laplace approximation becomes exact (Th.\ref{thm:A-deterministic-BvM theorem}).
This result can be used to rederive the classical BvM result that
the posterior is asymptotically Gaussian for IID data (Cor.\ref{cor:Bernstein-von-Mises:-IID case}).
However, the main draw of Th.\ref{thm:A-deterministic-BvM theorem}
is the fact that it yields several computable approximations (Cor.\ref{cor:Computable-approximations-of KL(g,f)})
of the KL divergence $KL\left(g_{LAP},f\right)$ and in particular
an approximation based on the ``KL variance'':
\begin{align}
KL\left(g_{LAP},f\right) & \approx\frac{1}{2}KL_{var}\left(g_{LAP},f\right)\\
 & \approx\frac{1}{2}\text{Var}_{\boldsymbol{\theta}\sim g_{LAP}\left(\boldsymbol{\theta}\right)}\left[\phi_{f}\left(\boldsymbol{\theta}\right)-\phi_{g}\left(\boldsymbol{\theta}\right)\right]
\end{align}
This approximation can thus be used to compute an indicator of the
quality of the Laplace approximation of a given log-concave posterior.

\subsection{Validating the Laplace approximation using $KL\left(g_{LAP},f\right)$}

Knowledge of the KL divergence $KL\left(g_{LAP},f\right)$ can be
used to validate whether, in a given problem, it is a valid approximation
or not. Indeed, knowing or bounding the KL divergence can be used
in two different ways to assess whether $g_{LAP}\left(\boldsymbol{\theta}\right)$
is indeed a good approximation of the target $f\left(\boldsymbol{\theta}\right)$.

First, we can use knowledge of $KL\left(g_{LAP},f\right)$ to derive
whether credible intervals of $g_{LAP}$ have good coverage under
$f$. Indeed, for any region $R$, the probabilities:
\begin{align}
p_{g} & =\mathbb{P}\left(\boldsymbol{\theta}_{g}\in R\right)\\
p_{f} & =\mathbb{P}\left(\boldsymbol{\theta}_{f}\in R\right)\nonumber 
\end{align}
must be such that:
\begin{equation}
p_{g}\log\left(\frac{p_{g}}{p_{f}}\right)+\left(1-p_{g}\right)\log\left(\frac{1-p_{g}}{1-p_{f}}\right)\leq KL\left(g_{LAP},f\right)
\end{equation}
This yields an upper and a lower bound on $p_{f}$ given the KL divergence
and $p_{g}$ and can thus be used to derive regions with guaranteed
coverage under $f$.

Second, knowledge of the KL divergence can be used to approximate
the marginal likelihood of the data under the model in Bayesian inference
which is a quantity required to perform model comparison. Given multiple
Bayesian models of a dataset $\mathcal{D}=d$, a posterior distribution
over the models can be constructed by computing the marginal likelihood
of the data under each model:
\begin{equation}
f_{M_{i}}\left(\mathcal{D}=d\right)=\int f_{M_{i}}\left(\mathcal{D}=d|\boldsymbol{\theta}\right)f_{M_{i}}\left(\boldsymbol{\theta}\right)d\boldsymbol{\theta}
\end{equation}
This quantity is very hard to estimate but it can be approximated
using an approximation of the posterior density $g\left(\boldsymbol{\theta}\right)\approx f_{M_{i}}\left(\boldsymbol{\theta}|\mathcal{D}=d\right)$
through the Evidence Lower Bound (ELBO):
\begin{align}
\log\left(f_{M_{i}}\left(\mathcal{D}=d\right)\right) & \geq ELBO\\
 & \geq\E_{\boldsymbol{\theta}\sim g\left(\boldsymbol{\theta}\right)}\left[\log\frac{f_{M_{i}}\left(\mathcal{D}=d|\boldsymbol{\theta}\right)f_{M_{i}}\left(\boldsymbol{\theta}\right)}{g\left(\boldsymbol{\theta}\right)}\right]
\end{align}
where the error of this approximation is precisely equal to the KL
divergence $KL\left(g\left(\boldsymbol{\theta}\right),f_{M_{i}}\left(\boldsymbol{\theta}|\mathcal{D}=d\right)\right)$.
Control of the KL divergence can thus enable rigorous Bayesian model
comparison based on the ELBO.

Theorem \ref{thm:A-deterministic-BvM theorem} and Cor.\ref{cor:Computable-approximations-of KL(g,f)}
thus provide a path towards rigorous Bayesian inference based on the
Laplace approximation through approximating or bounding $KL\left(g_{LAP},f\right)$
and then assessing precisely the error introduced by extrapolating
inferences drawn from $g_{LAP}$ to the true posterior $f$. My experiments
hint at $KL_{var}$ and the ``KLvar+LSI'' combination being the
best approximations for this purpose since they offered respectively
a tight approximation and tight upper-bound of the true KL divergence,
while being straightforward to compute.

\subsection{Assumptions}

It is instructive to consider the assumptions I propose here in great
detail, and in particular to compare them to the classical assumptions
of the BvM theorem.

In classical approaches to BvM results, global control is achieved
instead by assuming that the posterior concentrates around $\boldsymbol{\theta}_{0}$.
My assumption of the posterior being log-concave is considerably stronger
and implies consistency. This is due to log-concave distributions
being strongly concentrated around their posterior distribution (\citet{pereyra2016approximating}).
Weakening this assumption on the posterior would thus considerably
expand the applicability of Theorem \ref{thm:A-deterministic-BvM theorem}
but this appears quite challenging since the LSI will not be applicable
anymore to bound $KL\left(r_{g},r_{f}|\boldsymbol{e}\right)$.

In contrast, my assumption that the posterior has controlled third
derivatives is much closer to the equivalent assumption that the log-likelihood
is Locally Asymptotically Normal (LAN; \citet{van2000asymptotic}
Ch.7). These conditions state that around $\boldsymbol{\theta}_{0}$,
a quadratic approximation of the log-likelihood is asymptotically
valid which ensures in turn that the Laplace approximation of the
posterior is asymptotically correct. My assumption is stronger in
that I assume that a third derivative exists and that it is globally
bounded while LAN behavior can emerge with less regularity. For example,
the likelihood of a Laplace distribution does not have a second derivative,
but still leads to asymptotically Gaussian posteriors. However, extending
Theorem \ref{thm:A-deterministic-BvM theorem} to these circumstances
appears to be straightforward and I will pursue it in further work.

Finally, note that the assumption that every neighborhood of $\boldsymbol{\theta}_{0}$
has non-zero prior probability is implied by the assumption controlling
the third derivatives of the prior (eq.\eqref{eq: control on the third derivatives of the prior}).
Indeed, if the third derivative is bounded, then the prior must spread
its mass over all of $\R^{p}$ and it is impossible for the prior
not to put mass on all neighborhood of $\boldsymbol{\theta}_{0}$.

It seems to me straightforward to check whether the posterior in a
given model respects the conditions outlined in here. It might be
possible to directly assess whether the posterior is log-concave directly.
Alternatively, it is also possible to guarantee that the posterior
is log-concave by combining a log-concave prior with log-concave likelihoods
(such as those of the linear logistic classification model considered
in this article). Indeed, log-concavity is clearly conserved by products.
Controlling the derivatives of the posterior is similarly straightforward.
The easiest possibility consists in deriving an upper-bound on the
third derivative of the negative log-likelihood. This upper-bound
can be quite pessimistic since the precise value of $\Delta_{3}$
does not come into play in the $KL_{var}$ approximation of $KL\left(g_{LAP},f\right)$.

\subsection{Scaling with dimensionality}

An important aspect of modern asymptotics is the fact the dimensionality
$p$ might be comparable to $n$, thus complexifying the asymptotic
analysis. In this work, I have established in Theorem \ref{thm:A-deterministic-BvM theorem}
that the KL divergence tends to $0$ at least as $\left(\Delta_{3}\right)^{2}p^{3}$
in the limit. However, I am doubtful that the exponent of $p$ is
accurate here since it might instead reflect limits of the proof.
While additional work is required to acquire definite understanding
of the scaling $p$, the following examples are instructive.

First, I present an example that saturates the bound. Consider a target
density $f\left(\boldsymbol{\theta}\right)$ such that:
\begin{equation}
\tilde{\phi}_{f}\left(\tilde{\boldsymbol{\theta}}\right)\approx\frac{1}{2}\left\Vert \tilde{\boldsymbol{\theta}}\right\Vert _{2}^{2}+\frac{1}{6}\Delta_{3}\text{sign}\left(\tilde{\theta}_{1}\right)\left\Vert \tilde{\boldsymbol{\theta}}\right\Vert _{2}^{3}
\end{equation}
Note that this target is not continuous in $\tilde{\boldsymbol{\theta}}$
due to the discontinuity at $\tilde{\theta}_{1}=0$ so it does not
quite fit the assumptions of Theorem \ref{thm:A-deterministic-BvM theorem}
but we can still try to compute the quality of the approximation with
the standard Gaussian $\tilde{g}\left(\tilde{\boldsymbol{\theta}}\right)$.
Careful examination reveals that the proof of Theorem \ref{thm:A-deterministic-BvM theorem}
would still hold for this example since I only actually use continuity
in $r$ of $\tilde{\phi}_{f}\left(r\boldsymbol{e}\right)$. To understand
$KL\left(g,f\right)$, first compute the ELBO approximation of $\log\left[f\left(\boldsymbol{e}\right)\right]$:
\begin{align}
\log\left[f\left(\boldsymbol{e}\right)\right] & \approx\E_{r\sim g_{LAP}\left(r\right)}\left[\tilde{\phi}_{f}\left(r\boldsymbol{e}\right)-\tilde{\phi}_{g}\left(\boldsymbol{e}\right)\right]+C\\
 & \approx\frac{1}{6}\Delta_{3}\text{sign}\left(e_{1}\right)\E\left(r_{g}^{3}\right)+C\\
 & \approx\frac{1}{6}\Delta_{3}\text{sign}\left(e_{1}\right)p^{3/2}+C
\end{align}
The distribution of $\boldsymbol{e}_{f}$ is thus uniform on the two
half spheres $e_{1}>0$ and $e_{1}<0$. The KL divergence then scales
as the variance of $\log\left[f\left(\boldsymbol{e}\right)\right]$
(Prop.\ref{prop: KL approx KL_var}):
\begin{equation}
KL\left(g_{LAP},f\right)=\mathcal{O}\left(\left(\Delta_{3}\right)^{2}p^{3}\right)
\end{equation}
Thus, this example is such that the limit behavior indeed is reached
at the rate $\left(\Delta_{3}\right)^{2}p^{3}$ predicted by Theorem
\ref{thm:A-deterministic-BvM theorem}. However, this example is less
regular than assumed by the Theorem and is not representative of the
typical posterior distribution. More regular targets $\phi_{f}\left(\boldsymbol{\theta}\right)$
might have improved scaling in $p$.

In particular, modifying the assumptions on the local control of $\phi_{f}\left(\boldsymbol{\theta}\right)$
might change the exponent of $p$ in the asymptotic behavior. Indeed,
if the posterior is such that $\tilde{\phi}_{f}\left(\tilde{\boldsymbol{\theta}}\right)-\tilde{\phi}_{g}\left(\tilde{\boldsymbol{\theta}}\right)$
grows at $\left\Vert \tilde{\boldsymbol{\theta}}\right\Vert _{2}^{k}$
instead of $\left\Vert \tilde{\boldsymbol{\theta}}\right\Vert _{2}^{3}$
as assumed here, we should expect at most a growth at rate $p^{k}$
because:
\begin{align}
\text{Var}_{\boldsymbol{\theta}\sim g_{LAP}\left(\boldsymbol{\theta}\right)}\left[\phi_{f}\left(\boldsymbol{\theta}\right)-\phi_{g}\left(\boldsymbol{\theta}\right)\right] & =\text{Var}_{\tilde{\boldsymbol{\theta}}\sim\tilde{g}_{LAP}\left(\tilde{\boldsymbol{\theta}}\right)}\left[\tilde{\phi}_{f}\left(\tilde{\boldsymbol{\theta}}\right)-\tilde{\phi}_{g}\left(\tilde{\boldsymbol{\theta}}\right)\right]\\
 & \leq\E_{\tilde{\boldsymbol{\theta}}\sim\tilde{g}_{LAP}\left(\tilde{\boldsymbol{\theta}}\right)}\left[\tilde{\phi}_{f}\left(\tilde{\boldsymbol{\theta}}\right)-\tilde{\phi}_{g}\left(\tilde{\boldsymbol{\theta}}\right)\right]^{2}\\
 & \leq\E\left(\left\Vert \tilde{\boldsymbol{\theta}}\right\Vert _{2}^{2k}\right)\\
 & \apprle p^{k}
\end{align}

More generally, structural assumptions on the target $\phi_{f}\left(\boldsymbol{\theta}\right)$
can have a heavy impact on the dimensionality scaling. For example,
if the target distribution factorizes over each component: $f\left(\boldsymbol{\theta}\right)=\prod_{i=1}^{p}f\left(\theta_{i}\right)$,
then the Laplace approximation also factorizes, and we have:
\begin{align}
\text{Var}_{\boldsymbol{\theta}\sim g_{LAP}\left(\boldsymbol{\theta}\right)}\left[\phi_{f}\left(\boldsymbol{\theta}\right)-\phi_{g}\left(\boldsymbol{\theta}\right)\right] & =\sum_{i=1}^{p}\text{Var}_{\theta_{i}\sim g_{LAP}\left(\theta_{i}\right)}\left[\phi_{f}\left(\theta_{i}\right)-\phi_{g}\left(\theta_{i}\right)\right]\\
 & \leq p\Delta_{3}\E\left(\theta^{6}\right)
\end{align}
The scaling with dimensionality of the main term of Theorem \ref{thm:A-deterministic-BvM theorem}
scales only as $p$ instead of $p^{3}$. However, note that Theorem
\ref{thm:A-deterministic-BvM theorem} would still require $\Delta_{3}p^{3/2}\rightarrow0$
in order to guarantee that the additional error terms go to $0$.
Once again, I must emphasize that this might reflect a limit of the
proof and that the approximations of Cor.\ref{cor:Computable-approximations-of KL(g,f)}
might be valid even for faster growth of $p$.

\subsection{Limits}

Finally, please beware that quite a bit of care must be taken when
using the $KL_{var}$ approximation of the KL divergence to ensure
mathematical rigor.

Indeed, this approximation requires the asymptote $\Delta_{3}p^{3/2}\rightarrow0$
to be reached. I have established that this asymptote is reached as
$n\rightarrow\infty$ in the classical IID setting (Cor.\ref{cor:Bernstein-von-Mises:-IID case})
and I conjecture that it generally holds whenever the posterior has
a Gaussian limit, but further work will be required to check this
conjecture. Intuitively, the KL variance approximation should furthermore
yield a good approximation of the KL divergence when $n$ is small
and the prior is close to Gaussian but further theoretical work is
needed to establish this rigorously.

It is furthermore critical to keep in mind the key restriction that
$f$ needs to be log-concave in order for Theorem \ref{thm:A-deterministic-BvM theorem}
to apply. Indeed, it is straightforward to design counter-examples,
such as the mixture of a thin and wide Gaussian discussed in Section
\ref{subsec:Assumptions}, for which the KL variance is arbitrarily
close to $0$ while the KL divergence is arbitrarily big. Using the
KL variance approximation might thus be inappropriate in the absence
of an argument to ensure that most of the mass of $f$ is concentrated
around its mode.

Another more trivial limit is the fact that the KL divergence $KL\left(g_{LAP},f\right)$
is infinite if the support of $f$ is not $\R^{p}$. Careful application
of Theorem \ref{thm:A-deterministic-BvM theorem} would reveal this
immediately since the support of $f$ being strictly smaller than
$\R^{p}$ requires $\Delta_{3}=\infty$. The KLvar approximation of
the KL divergence would also be infinite but empirical approximations
of $KL_{var}\left(g_{LAP},f\right)$ might fail to recover this infinite
value. It is thus necessary to check the support of $f$.

Finally, it is critical to notice that the present work only applies
to the Laplace approximation of $f$. Indeed, several important modern
alternatives, such as the Gaussian Variational approximation or mean-field
approximations (i.e. Gaussian approximation with a Diagonal or block-Diagonal
covariance; \citet{blei2017variational}), are not covered by Theorem
\ref{thm:A-deterministic-BvM theorem}. Non Gaussian approximations
are also not covered. While it is possible that the KL variance gives
a possible measure of the quality of such approximations, additional
work is required to determine this rigorously.

\section{Conclusion}

In this article, I have shown that, under assumptions, it is asymptotically
valid to approximate the KL divergence between a posterior distribution
and its Laplace approximation by using one-half of the KL variance
instead. This gives a computable measure of the quality of the Laplace
approximation that can then be used to measure the error induced by
this approximation, and thus assessing whether the it is acceptable,
or needs to be replaced by another more precise approximation of the
posterior. Future work will be required to assess if this approximation
of the KL divergence can be extended to other approximations and weaker
assumptions.

\bibliographystyle{unsrtnat}
\bibliography{ref}

\newpage{}

\appendix

\part*{Appendix}

This appendix provides detailed proofs of every claim I make in the
main text of the article and gives details of how to reproduce the
empirical findings. Proofs are contained in Sections \ref{sec: APPENDIX KL-var approximation}-\ref{sec: APPENDIX Approximations-of-the KL divergence}.
Details of the empirical study of the logistic classification model
are given in Section \ref{sec:APPENDIX Details-of-the simulations}.

Each section deals with the proof of one result in the main text:
\begin{itemize}
\item Section \ref{sec: APPENDIX KL-var approximation} proves Prop.\ref{prop: KL approx KL_var}.
\item Section \ref{sec:APPENDIX LSI based approximations} proves Prop.\ref{prop:LSI-bound-on KL(g,f)}
and \ref{prop:LSI-bound-on KL(f,g)}.
\item Section \ref{sec: APPENDIX General-deterministic-Bound} proves Th.\ref{thm:A-deterministic-BvM theorem}.
\item Section \ref{sec: APPENDIX Recovering-the-classical IID result} proves
Cor.\ref{cor:Bernstein-von-Mises:-IID case}.
\item Section \ref{sec: APPENDIX Approximations-of-the KL divergence} proves
Cor.\ref{cor:Computable-approximations-of KL(g,f)}.
\end{itemize}
Since the sequence of proofs is quite involved, I have tried to maximize
its readability in the following ways. Each section starts with a
detailed description of the structure of the proof of the corresponding
claim. The Lemmas proved in each section are almost only required
inside the corresponding section. Each Lemma and Proposition of this
appendix details its ``parents'' in the dependency structure of
this proof.

\section{The $KL_{var}$ approximation }

\label{sec: APPENDIX KL-var approximation}

This section deals with the proof of Prop.\ref{prop: KL approx KL_var}
which asserts that, under limiting assumptions, the KL variance can
be used as an approximation for the KL divergence.

Recall that throughout this section, $f\left(\boldsymbol{\theta}\right)$
and $g\left(\boldsymbol{\theta}\right)$ can be any probability densities:
\begin{align}
f\left(\boldsymbol{\theta}\right) & =\exp\left(-\phi_{f}\left(\boldsymbol{\theta}\right)-\log\left(Z_{f}\right)\right)\\
g\left(\boldsymbol{\theta}\right) & =\exp\left(-\phi_{g}\left(\boldsymbol{\theta}\right)-\log\left(Z_{g}\right)\right)
\end{align}

\subsection{Proof structure}

This proof is fairly straightforward. The following results are established
in order:
\begin{itemize}
\item First, I establish that, for any densities $f\left(\boldsymbol{\theta}\right)$
and $g\left(\boldsymbol{\theta}\right)$, the KL divergence can be
rewritten in a way that discards the log-normalizing constants (Lemma
\ref{lem:Another-formula-for KL(g,f)}).
\item Then, I prove that the family:
\begin{align}
h\left(\boldsymbol{\theta};\lambda\right) & =g\left(\boldsymbol{\theta}\right)\left[\frac{f\left(\boldsymbol{\theta}\right)}{g\left(\boldsymbol{\theta}\right)}\right]^{\lambda}\exp\left[-C\left(\lambda\right)\right]\ \ \lambda\in\left[0,1\right]\\
C\left(\lambda\right) & =\log\left\{ \int g\left(\boldsymbol{\theta}\right)\left[\frac{f\left(\boldsymbol{\theta}\right)}{g\left(\boldsymbol{\theta}\right)}\right]^{\lambda}d\boldsymbol{\theta}\right\} 
\end{align}
exists and is an exponential family (Lemma \ref{lem:The-h-lambda-family.}).
\item Then, I establish key properties of a function $\tilde{C}\left(\lambda\right)$
that is closely related to the cumulant generating function $C\left(\lambda\right)$
(Lemma \ref{lem:Properties-of- C-tilde}).
\item Then, I give a bound on the cumulants of a bounded random variable
(Lemma \ref{lem: approximation of KL(g,h)}).
\item Finally, I define the function:
\begin{equation}
K\left(\lambda\right)=KL\left(g\left(\boldsymbol{\theta}\right),h\left(\boldsymbol{\theta};\lambda\right)\right)
\end{equation}
and show how to approximate it using the KL variance (Lemma \ref{lem:Bound-on-the cumulants.}).
\end{itemize}
These results then yield Prop.\ref{prop: KL approx KL_var}.

\subsection{Proofs}
\begin{lem}
Another formula for $KL\left(g,f\right)$.\label{lem:Another-formula-for KL(g,f)}

\textbf{Requires: }NA.

The KL divergence:
\begin{align}
KL\left(g,f\right) & =\E_{\boldsymbol{\theta}\sim g\left(\boldsymbol{\theta}\right)}\left[\log\frac{g\left(\boldsymbol{\theta}\right)}{f\left(\boldsymbol{\theta}\right)}\right]\\
 & =\E_{\boldsymbol{\theta}\sim g\left(\boldsymbol{\theta}\right)}\left[\phi_{f}\left(\boldsymbol{\theta}\right)+\log\left(Z_{f}\right)-\phi_{g}\left(\boldsymbol{\theta}\right)-\log\left(Z_{g}\right)\right]
\end{align}
can be rewritten with no reference to the normalizing constants:
\begin{equation}
KL\left(g,f\right)=\E_{\boldsymbol{\theta}\sim g\left(\boldsymbol{\theta}\right)}\left[\phi_{f}\left(\boldsymbol{\theta}\right)-\phi_{g}\left(\boldsymbol{\theta}\right)\right]+\log\left\{ \E_{\boldsymbol{\theta}\sim g\left(\boldsymbol{\theta}\right)}\left[\exp\left(\phi_{g}\left(\boldsymbol{\theta}\right)-\phi_{f}\left(\boldsymbol{\theta}\right)\right)\right]\right\} 
\end{equation}
\end{lem}

\begin{proof}
This proof results from a straightforward manipulation of the expression
of $KL\left(g,f\right)$.

Starting from the usual formulation for $KL\left(g,f\right)$, we
first remove the normalizing constants from the expected value:
\begin{align}
KL\left(g,f\right) & =\E_{\boldsymbol{\theta}\sim g\left(\boldsymbol{\theta}\right)}\left[\phi_{f}\left(\boldsymbol{\theta}\right)+\log\left(Z_{f}\right)-\phi_{g}\left(\boldsymbol{\theta}\right)-\log\left(Z_{g}\right)\right]\\
 & =\E_{\boldsymbol{\theta}\sim g\left(\boldsymbol{\theta}\right)}\left[\phi_{f}\left(\boldsymbol{\theta}\right)-\phi_{g}\left(\boldsymbol{\theta}\right)\right]+\log\left(Z_{f}\right)-\log\left(Z_{g}\right)\\
 & =\E_{\boldsymbol{\theta}\sim g\left(\boldsymbol{\theta}\right)}\left[\phi_{f}\left(\boldsymbol{\theta}\right)-\phi_{g}\left(\boldsymbol{\theta}\right)\right]+\log\left(\frac{Z_{f}}{Z_{g}}\right)\label{eq: intermediate equation for KL rephrasing}
\end{align}

Let us make explicit $Z_{f}$:
\begin{equation}
Z_{f}=\int\exp\left(-\phi_{f}\left(\boldsymbol{\theta}\right)\right)d\boldsymbol{\theta}
\end{equation}
We can rework this expression in order to make $g\left(\boldsymbol{\theta}\right)$
appear:
\begin{align}
Z_{f} & =\int\frac{g\left(\boldsymbol{\theta}\right)}{g\left(\boldsymbol{\theta}\right)}\exp\left(-\phi_{f}\left(\boldsymbol{\theta}\right)\right)d\boldsymbol{\theta}\\
 & =\int g\left(\boldsymbol{\theta}\right)\exp\left(\phi_{g}\left(\boldsymbol{\theta}\right)+\log\left(Z_{g}\right)-\phi_{f}\left(\boldsymbol{\theta}\right)\right)d\boldsymbol{\theta}\\
 & =Z_{g}\int g\left(\boldsymbol{\theta}\right)\exp\left(\phi_{g}\left(\boldsymbol{\theta}\right)-\phi_{f}\left(\boldsymbol{\theta}\right)\right)d\boldsymbol{\theta}\\
\frac{Z_{f}}{Z_{g}} & =\E_{\boldsymbol{\theta}\sim g\left(\boldsymbol{\theta}\right)}\left[\exp\left(\phi_{g}\left(\boldsymbol{\theta}\right)-\phi_{f}\left(\boldsymbol{\theta}\right)\right)\right]
\end{align}

Returning to eq.\eqref{eq: intermediate equation for KL rephrasing}
with this final expression, we obtain the claimed result:
\begin{align}
KL\left(g,f\right) & =\E_{\boldsymbol{\theta}\sim g\left(\boldsymbol{\theta}\right)}\left[\phi_{f}\left(\boldsymbol{\theta}\right)-\phi_{g}\left(\boldsymbol{\theta}\right)\right]+\log\left(\frac{Z_{f}}{Z_{g}}\right)\\
 & =\E_{\boldsymbol{\theta}\sim g\left(\boldsymbol{\theta}\right)}\left[\phi_{f}\left(\boldsymbol{\theta}\right)-\phi_{g}\left(\boldsymbol{\theta}\right)\right]+\log\left\{ \E_{\boldsymbol{\theta}\sim g\left(\boldsymbol{\theta}\right)}\left[\exp\left(\phi_{g}\left(\boldsymbol{\theta}\right)-\phi_{f}\left(\boldsymbol{\theta}\right)\right)\right]\right\} 
\end{align}
which concludes this proof.
\end{proof}
\begin{lem}
The $h\left(\boldsymbol{\theta};\lambda\right)$ family.\label{lem:The-h-lambda-family.}

\textbf{Requires: }NA.

The family of distributions:
\begin{align}
h\left(\boldsymbol{\theta};\lambda\right) & =g\left(\boldsymbol{\theta}\right)\left[\frac{f\left(\boldsymbol{\theta}\right)}{g\left(\boldsymbol{\theta}\right)}\right]^{\lambda}\exp\left[-C\left(\lambda\right)\right]\ \ \lambda\in\left[0,1\right]\\
 & =g\left(\boldsymbol{\theta}\right)\exp\left(\lambda\left[\phi_{g}\left(\boldsymbol{\theta}\right)-\phi_{f}\left(\boldsymbol{\theta}\right)\right]-\tilde{C}\left(\lambda\right)\right)\\
C\left(\lambda\right) & =\log\left\{ \int g\left(\boldsymbol{\theta}\right)\left[\frac{f\left(\boldsymbol{\theta}\right)}{g\left(\boldsymbol{\theta}\right)}\right]^{\lambda}d\boldsymbol{\theta}\right\} \\
\tilde{C}\left(\lambda\right) & =\log\left\{ \E_{\boldsymbol{\theta}\sim g\left(\boldsymbol{\theta}\right)}\left[\exp\left(\lambda\left[\phi_{g}\left(\boldsymbol{\theta}\right)-\phi_{f}\left(\boldsymbol{\theta}\right)\right]\right)\right]\right\} 
\end{align}
exists and is an exponential family.
\end{lem}

\begin{proof}
This Lemma contains multiple statements that we will check in order.

First, let us prove that the family does exist. This requires proving
that the function $C\left(\lambda\right)$ takes finite values over
the range $\lambda\in]0,1[$. This follows from the Holder inequality
with $p=\left(1-\lambda\right)^{-1}$ and $q=\lambda^{-1}$ such that
$p^{-1}+q^{-1}=1$:
\begin{align}
\int g\left(\boldsymbol{\theta}\right)\left[\frac{f\left(\boldsymbol{\theta}\right)}{g\left(\boldsymbol{\theta}\right)}\right]^{\lambda}d\boldsymbol{\theta} & =\int\left[g\left(\boldsymbol{\theta}\right)\right]^{1-\lambda}\left[f\left(\boldsymbol{\theta}\right)\right]^{\lambda}d\boldsymbol{\theta}\\
 & \leq\left\{ \int\left[g\left(\boldsymbol{\theta}\right)\right]^{p\left(1-\lambda\right)}d\boldsymbol{\theta}\right\} ^{p^{-1}}\left\{ \int\left[f\left(\boldsymbol{\theta}\right)\right]^{q\lambda}d\boldsymbol{\theta}\right\} ^{q^{-1}}\\
 & \leq\left\{ \int g\left(\boldsymbol{\theta}\right)d\boldsymbol{\theta}\right\} ^{1-\lambda}\left\{ \int f\left(\boldsymbol{\theta}\right)d\boldsymbol{\theta}\right\} ^{\lambda}\\
 & \leq1
\end{align}
Then, we have that, for all $\lambda\in\left[0,1\right]$, $h\left(\boldsymbol{\theta};\lambda\right)$
is a positive function which integrates to $1$. It is thus indeed
a density. The family thus exists.

In order to see that it is indeed an exponential family, just rewrite
the density:
\begin{align}
h\left(\boldsymbol{\theta};\lambda\right) & =g\left(\boldsymbol{\theta}\right)\left[\frac{f\left(\boldsymbol{\theta}\right)}{g\left(\boldsymbol{\theta}\right)}\right]^{\lambda}\exp\left[-C\left(\lambda\right)\right]\\
 & =g\left(\boldsymbol{\theta}\right)\exp\left(\lambda\left[\phi_{g}\left(\boldsymbol{\theta}\right)-\phi_{f}\left(\boldsymbol{\theta}\right)+\log\left(\frac{Z_{g}}{Z_{f}}\right)\right]-C\left(\lambda\right)\right)
\end{align}
This is indeed an exponential family. It has base density $g\left(\boldsymbol{\theta}\right)$,
a single natural parameter $\lambda$, associated sufficient statistic
$\left[\phi_{g}\left(\boldsymbol{\theta}\right)-\phi_{f}\left(\boldsymbol{\theta}\right)+\log\left(\frac{Z_{g}}{Z_{f}}\right)\right]$
and normalizing constant (or cumulant generating function) $C\left(\lambda\right)$.

We can also rewrite the exponential family as:
\begin{align}
h\left(\boldsymbol{\theta};\lambda\right) & =g\left(\boldsymbol{\theta}\right)\exp\left(\lambda\left[\phi_{g}\left(\boldsymbol{\theta}\right)-\phi_{f}\left(\boldsymbol{\theta}\right)+\log\left(\frac{Z_{g}}{Z_{f}}\right)\right]-C\left(\lambda\right)\right)\\
 & =g\left(\boldsymbol{\theta}\right)\exp\left(\lambda\left[\phi_{g}\left(\boldsymbol{\theta}\right)-\phi_{f}\left(\boldsymbol{\theta}\right)\right]-\lambda\left[\log\left(\frac{Z_{g}}{Z_{f}}\right)\right]-C\left(\lambda\right)\right)\\
 & =g\left(\boldsymbol{\theta}\right)\exp\left(\lambda\left[\phi_{g}\left(\boldsymbol{\theta}\right)-\phi_{f}\left(\boldsymbol{\theta}\right)\right]-\tilde{C}\left(\lambda\right)\right)
\end{align}
where:
\begin{align}
\tilde{C}\left(\lambda\right) & =\log\left\{ \int g\left(\boldsymbol{\theta}\right)\exp\left(\lambda\left[\phi_{g}\left(\boldsymbol{\theta}\right)-\phi_{f}\left(\boldsymbol{\theta}\right)\right]\right)d\boldsymbol{\theta}\right\} \\
 & =\log\left\{ \E_{\boldsymbol{\theta}\sim g\left(\boldsymbol{\theta}\right)}\left[\exp\left(\lambda\left[\phi_{g}\left(\boldsymbol{\theta}\right)-\phi_{f}\left(\boldsymbol{\theta}\right)\right]\right)\right]\right\} 
\end{align}

This concludes the proof.
\end{proof}
\begin{lem}
Properties of $\tilde{C}\left(\lambda\right)$.\label{lem:Properties-of- C-tilde}

\textbf{Requires: }Lemma \ref{lem:The-h-lambda-family.}.

The Cumulant Generating Function:
\begin{equation}
\tilde{C}\left(\lambda\right)=\log\left\{ \E_{\boldsymbol{\theta}\sim g\left(\boldsymbol{\theta}\right)}\left[\exp\left(\lambda\left[\phi_{g}\left(\boldsymbol{\theta}\right)-\phi_{f}\left(\boldsymbol{\theta}\right)\right]\right)\right]\right\} 
\end{equation}
is convex.

Its derivatives are the cumulants of the statistic $\phi_{g}\left(\boldsymbol{\theta}\right)-\phi_{f}\left(\boldsymbol{\theta}\right)$
for $\boldsymbol{\theta}\sim h\left(\boldsymbol{\theta};\lambda\right)$.
In particular:
\begin{align}
\tilde{C}^{'}\left(\lambda\right) & =k_{1}\left(\lambda\right)=\E_{\boldsymbol{\theta}\sim h\left(\boldsymbol{\theta};\lambda\right)}\left[\phi_{g}\left(\boldsymbol{\theta}\right)-\phi_{f}\left(\boldsymbol{\theta}\right)\right]\\
\tilde{C}^{''}\left(\lambda\right) & =k_{2}\left(\lambda\right)=\text{Var}_{\boldsymbol{\theta}\sim h\left(\boldsymbol{\theta};\lambda\right)}\left[\phi_{g}\left(\boldsymbol{\theta}\right)-\phi_{f}\left(\boldsymbol{\theta}\right)\right]\\
\tilde{C}^{\left(3\right)}\left(\lambda\right) & =k_{3}\left(\lambda\right)=\E_{\boldsymbol{\theta}\sim h\left(\boldsymbol{\theta};\lambda\right)}\left[\left\{ \phi_{g}\left(\boldsymbol{\theta}\right)-\phi_{f}\left(\boldsymbol{\theta}\right)-k_{1}\left(\lambda\right)\right\} ^{3}\right]
\end{align}
\end{lem}

\begin{proof}
This Lemma just corresponds to standard properties of exponential
families and their CGF.

$\tilde{C}\left(\lambda\right)$ is defined as:
\begin{equation}
\tilde{C}\left(\lambda\right)=\log\left\{ \E_{\boldsymbol{\theta}\sim g\left(\boldsymbol{\theta}\right)}\left[\exp\left(\lambda\left[\phi_{g}\left(\boldsymbol{\theta}\right)-\phi_{f}\left(\boldsymbol{\theta}\right)\right]\right)\right]\right\} 
\end{equation}
which we can rewrite as:
\begin{align}
\exp\left[\tilde{C}\left(\lambda\right)\right] & =\E_{\boldsymbol{\theta}\sim g\left(\boldsymbol{\theta}\right)}\left[\exp\left(\lambda\left[\phi_{g}\left(\boldsymbol{\theta}\right)-\phi_{f}\left(\boldsymbol{\theta}\right)\right]\right)\right]\\
 & =\int g\left(\boldsymbol{\theta}\right)\exp\left(\lambda\left[\phi_{g}\left(\boldsymbol{\theta}\right)-\phi_{f}\left(\boldsymbol{\theta}\right)\right]\right)d\boldsymbol{\theta}\\
 & =\int\exp\left(-\phi_{g}\left(\boldsymbol{\theta}\right)-\log\left(Z_{g}\right)-\lambda\left[\phi_{g}\left(\boldsymbol{\theta}\right)-\phi_{f}\left(\boldsymbol{\theta}\right)\right]\right)d\boldsymbol{\theta}\\
 & =\frac{1}{Z_{g}}\int\exp\left(-\left(1-\lambda\right)\phi_{g}\left(\boldsymbol{\theta}\right)-\lambda\phi_{f}\left(\boldsymbol{\theta}\right)\right)d\boldsymbol{\theta}
\end{align}

Let us now prove the convexity of $\tilde{C}\left(\lambda\right)$.
Let $\lambda_{1}<\lambda_{2}<\lambda_{3}$. Without loss of generality,
we can assume that $\lambda_{1}=0$ and $\lambda_{3}=1$ and $\lambda=\lambda_{2}$
which simply corresponds to considering new functions $g\left(\boldsymbol{\theta}\right)=h\left(\boldsymbol{\theta};\lambda_{1}\right)$
and $f\left(\boldsymbol{\theta};\lambda_{3}\right)$. Now apply Holder
inequality to $\tilde{C}\left(\lambda\right)$ with $p=\left(1-\lambda\right)$
and $q=\lambda^{-1}$:
\begin{equation}
\int\exp\left(-\left(1-\lambda\right)\phi_{g}\left(\boldsymbol{\theta}\right)-\lambda\phi_{f}\left(\boldsymbol{\theta}\right)\right)d\boldsymbol{\theta}\leq\left[\int\exp\left(-\phi_{g}\left(\boldsymbol{\theta}\right)\right)d\boldsymbol{\theta}\right]^{1-\lambda}\left[\int\exp\left(-\phi_{f}\left(\boldsymbol{\theta}\right)\right)d\boldsymbol{\theta}\right]^{\lambda}
\end{equation}
We thus have:
\begin{align}
\tilde{C}\left(\lambda\right) & =\log\left[\frac{1}{Z_{g}}\int\exp\left(-\left(1-\lambda\right)\phi_{g}\left(\boldsymbol{\theta}\right)-\lambda\phi_{f}\left(\boldsymbol{\theta}\right)\right)d\boldsymbol{\theta}\right]\\
 & \leq-\left(1-\lambda+\lambda\right)\log\left(Z_{g}\right)+\left(1-\lambda\right)\log\left[\int\exp\left(-\phi_{g}\left(\boldsymbol{\theta}\right)\right)d\boldsymbol{\theta}\right]+\lambda\log\left[\int\exp\left(-\phi_{f}\left(\boldsymbol{\theta}\right)\right)d\boldsymbol{\theta}\right]\\
 & \leq\left(1-\lambda\right)\tilde{C}\left(0\right)+\lambda\tilde{C}\left(1\right)
\end{align}
which establishes the convexity of $\tilde{C}\left(\lambda\right)$.

The derivatives of a CGF are (like the name indicates) the cumulants
of the associated sufficient statistic. The first derivative of $\tilde{C}\left(\lambda\right)$
is thus the mean of $\phi_{g}\left(\boldsymbol{\theta}\right)-\phi_{f}\left(\boldsymbol{\theta}\right)$.
Its second derivative is similarly the covariance. The third derivative
is slightly more tricky to state since it is the third cumulant which
is defined, for a variable $X$, as:
\begin{equation}
k_{3}\left(X\right)=\E\left[\left\{ X-\E\left(X\right)\right\} ^{3}\right]
\end{equation}
These cumulants are computed under the density $h\left(\boldsymbol{\theta};\lambda\right)$.

For example, the calculation of the first derivative yields:
\begin{align}
\tilde{C}^{'}\left(\lambda\right) & =\frac{\partial}{\partial\lambda}\log\left\{ \E_{\boldsymbol{\theta}\sim g\left(\boldsymbol{\theta}\right)}\left[\exp\left(\lambda\left[\phi_{g}\left(\boldsymbol{\theta}\right)-\phi_{f}\left(\boldsymbol{\theta}\right)\right]\right)\right]\right\} \\
 & =\frac{\int g\left(\boldsymbol{\theta}\right)\frac{\partial}{\partial\lambda}\exp\left(\lambda\left[\phi_{g}\left(\boldsymbol{\theta}\right)-\phi_{f}\left(\boldsymbol{\theta}\right)\right]\right)d\boldsymbol{\theta}}{\int g\left(\boldsymbol{\theta}\right)\exp\left(\lambda\left[\phi_{g}\left(\boldsymbol{\theta}\right)-\phi_{f}\left(\boldsymbol{\theta}\right)\right]\right)d\boldsymbol{\theta}}\\
 & =\frac{\int g\left(\boldsymbol{\theta}\right)\left[\phi_{g}\left(\boldsymbol{\theta}\right)-\phi_{f}\left(\boldsymbol{\theta}\right)\right]\exp\left(\lambda\left[\phi_{g}\left(\boldsymbol{\theta}\right)-\phi_{f}\left(\boldsymbol{\theta}\right)\right]\right)d\boldsymbol{\theta}}{\int g\left(\boldsymbol{\theta}\right)\exp\left(\lambda\left[\phi_{g}\left(\boldsymbol{\theta}\right)-\phi_{f}\left(\boldsymbol{\theta}\right)\right]\right)d\boldsymbol{\theta}}\\
 & =\int g\left(\boldsymbol{\theta}\right)\left[\phi_{g}\left(\boldsymbol{\theta}\right)-\phi_{f}\left(\boldsymbol{\theta}\right)\right]\exp\left(\lambda\left[\phi_{g}\left(\boldsymbol{\theta}\right)-\phi_{f}\left(\boldsymbol{\theta}\right)\right]-\tilde{C}\left(\lambda\right)\right)d\boldsymbol{\theta}\\
 & =\E_{\boldsymbol{\theta}\sim h\left(\boldsymbol{\theta};\lambda\right)}\left[\phi_{g}\left(\boldsymbol{\theta}\right)-\phi_{f}\left(\boldsymbol{\theta}\right)\right]
\end{align}

This concludes the proof.
\end{proof}
\begin{lem}
Approximation of $KL\left(g\left(\boldsymbol{\theta}\right),h\left(\boldsymbol{\theta};\lambda\right)\right)$.\label{lem: approximation of KL(g,h)}

\textbf{Requires:} Lemmas \ref{lem:Another-formula-for KL(g,f)}-\ref{lem:Properties-of- C-tilde}.

The KL divergence between $g\left(\boldsymbol{\theta}\right)$ and
$h\left(\boldsymbol{\theta};\lambda\right)$:
\begin{align}
KL\left(g\left(\boldsymbol{\theta}\right),h\left(\boldsymbol{\theta};\lambda\right)\right) & =K\left(\lambda\right)\\
 & =\lambda\E_{\boldsymbol{\theta}\sim g\left(\boldsymbol{\theta}\right)}\left[\phi_{f}\left(\boldsymbol{\theta}\right)-\phi_{g}\left(\boldsymbol{\theta}\right)\right]+\log\left\{ \E_{\boldsymbol{\theta}\sim g\left(\boldsymbol{\theta}\right)}\left[\exp\left(\lambda\left[\phi_{g}\left(\boldsymbol{\theta}\right)-\phi_{f}\left(\boldsymbol{\theta}\right)\right]\right)\right]\right\} \\
 & =\lambda\E_{\boldsymbol{\theta}\sim g\left(\boldsymbol{\theta}\right)}\left[\phi_{f}\left(\boldsymbol{\theta}\right)-\phi_{g}\left(\boldsymbol{\theta}\right)\right]+\tilde{C}\left(\lambda\right)
\end{align}
has a Taylor approximation with integral remainder:
\begin{align}
K\left(\lambda\right) & =0+0+\frac{\lambda^{2}}{2}\text{Var}_{\boldsymbol{\theta}\sim g\left(\boldsymbol{\theta}\right)}\left[\phi_{g}\left(\boldsymbol{\theta}\right)-\phi_{f}\left(\boldsymbol{\theta}\right)\right]+\int_{0}^{\lambda}\frac{1}{2}l^{2}k_{3}\left(l\right)dl\\
k_{3}\left(\lambda\right) & =\E_{\boldsymbol{\theta}\sim h\left(\boldsymbol{\theta};\lambda\right)}\left[\left\{ \phi_{g}\left(\boldsymbol{\theta}\right)-\phi_{f}\left(\boldsymbol{\theta}\right)-k_{1}\left(\lambda\right)\right\} ^{3}\right]\\
k_{1}\left(\lambda\right) & =\E_{\boldsymbol{\theta}\sim h\left(\boldsymbol{\theta};\lambda\right)}\left[\phi_{g}\left(\boldsymbol{\theta}\right)-\phi_{f}\left(\boldsymbol{\theta}\right)\right]
\end{align}
\end{lem}

\begin{proof}
This proof follows from Lemmas \ref{lem:Another-formula-for KL(g,f)}
and \ref{lem:The-h-lambda-family.}, which yield that $K\left(\lambda\right)$
is indeed a correct formula for $KL\left(g\left(\boldsymbol{\theta}\right),h\left(\boldsymbol{\theta};\lambda\right)\right)$,
and then a straightforward manipulation of $K\left(\lambda\right)$.

First, let us check that the formula for $K\left(\lambda\right)$
is indeed correct. Lemma \ref{lem:The-h-lambda-family.} gives us
the log-density of $h\left(\boldsymbol{\theta};\lambda\right)$:
\begin{align}
h\left(\boldsymbol{\theta};\lambda\right) & =g\left(\boldsymbol{\theta}\right)\left[\frac{f\left(\boldsymbol{\theta}\right)}{g\left(\boldsymbol{\theta}\right)}\right]^{\lambda}\exp\left(-C\left(\boldsymbol{\theta}\right)\right)\\
 & =\exp\left(-\phi_{g}\left(\boldsymbol{\theta}\right)+\lambda\left[\phi_{g}\left(\boldsymbol{\theta}\right)-\phi_{f}\left(\boldsymbol{\theta}\right)\right]-\log\left(Z_{g}\right)-\tilde{C}\left(\boldsymbol{\theta}\right)\right)
\end{align}

Lemma \ref{lem:Another-formula-for KL(g,f)} asserts that we can ignore
the normalizing constant terms $-\log\left(Z_{g}\right)-\tilde{C}\left(\boldsymbol{\theta}\right)$.
One unnormalized negative log-density for $h$ is:
\begin{equation}
\phi_{h}\left(\boldsymbol{\theta}\right)=\phi_{g}\left(\boldsymbol{\theta}\right)-\lambda\left[\phi_{g}\left(\boldsymbol{\theta}\right)-\phi_{f}\left(\boldsymbol{\theta}\right)\right]
\end{equation}
and Lemma \ref{lem:Another-formula-for KL(g,f)} then gives us the
KL divergence as:
\begin{align}
K\left(\lambda\right) & =\E_{\boldsymbol{\theta}\sim g\left(\boldsymbol{\theta}\right)}\left[\phi_{h}\left(\boldsymbol{\theta}\right)-\phi_{g}\left(\boldsymbol{\theta}\right)\right]+\log\left\{ \E_{\boldsymbol{\theta}\sim g\left(\boldsymbol{\theta}\right)}\left[\exp\left(\phi_{g}\left(\boldsymbol{\theta}\right)-\phi_{h}\left(\boldsymbol{\theta}\right)\right)\right]\right\} \\
 & =\E_{\boldsymbol{\theta}\sim g\left(\boldsymbol{\theta}\right)}\left[-\lambda\left[\phi_{g}\left(\boldsymbol{\theta}\right)-\phi_{f}\left(\boldsymbol{\theta}\right)\right]\right]+\log\left\{ \E_{\boldsymbol{\theta}\sim g\left(\boldsymbol{\theta}\right)}\left[\exp\left(\lambda\left[\phi_{g}\left(\boldsymbol{\theta}\right)-\phi_{f}\left(\boldsymbol{\theta}\right)\right]\right)\right]\right\} \\
 & =\lambda\E_{\boldsymbol{\theta}\sim g\left(\boldsymbol{\theta}\right)}\left[\phi_{f}\left(\boldsymbol{\theta}\right)-\phi_{g}\left(\boldsymbol{\theta}\right)\right]+\log\left\{ \E_{\boldsymbol{\theta}\sim g\left(\boldsymbol{\theta}\right)}\left[\exp\left(\lambda\left[\phi_{g}\left(\boldsymbol{\theta}\right)-\phi_{f}\left(\boldsymbol{\theta}\right)\right]\right)\right]\right\} 
\end{align}

In the formula for $K\left(\lambda\right)$, we recognize $\tilde{C}\left(\lambda\right)$
in the second term:
\begin{align}
K\left(\lambda\right) & =\lambda\E_{\boldsymbol{\theta}\sim g\left(\boldsymbol{\theta}\right)}\left[\phi_{f}\left(\boldsymbol{\theta}\right)-\phi_{g}\left(\boldsymbol{\theta}\right)\right]+\log\left\{ \E_{\boldsymbol{\theta}\sim g\left(\boldsymbol{\theta}\right)}\left[\exp\left(\lambda\left[\phi_{g}\left(\boldsymbol{\theta}\right)-\phi_{f}\left(\boldsymbol{\theta}\right)\right]\right)\right]\right\} \\
 & =\lambda\E_{\boldsymbol{\theta}\sim g\left(\boldsymbol{\theta}\right)}\left[\phi_{f}\left(\boldsymbol{\theta}\right)-\phi_{g}\left(\boldsymbol{\theta}\right)\right]+\tilde{C}\left(\lambda\right)
\end{align}
The derivatives of $K\left(\lambda\right)$ at $\lambda=0$ are then
straightforward to compute:
\begin{align}
K\left(0\right) & =0+\tilde{C}\left(0\right)\\
 & =\log\left\{ \E_{\boldsymbol{\theta}\sim g\left(\boldsymbol{\theta}\right)}\left[\exp\left(0\left[\phi_{g}\left(\boldsymbol{\theta}\right)-\phi_{f}\left(\boldsymbol{\theta}\right)\right]\right)\right]\right\} \\
 & =\log\left\{ \E_{\boldsymbol{\theta}\sim g\left(\boldsymbol{\theta}\right)}\left[\exp\left(0\right)\right]\right\} \\
 & =\log\left\{ \E_{\boldsymbol{\theta}\sim g\left(\boldsymbol{\theta}\right)}\left[1\right]\right\} \\
 & =0\\
K^{'}\left(0\right) & =\E_{\boldsymbol{\theta}\sim g\left(\boldsymbol{\theta}\right)}\left[\phi_{f}\left(\boldsymbol{\theta}\right)-\phi_{g}\left(\boldsymbol{\theta}\right)\right]+\E_{\boldsymbol{\theta}\sim g\left(\boldsymbol{\theta}\right)}\left[\phi_{g}\left(\boldsymbol{\theta}\right)-\phi_{f}\left(\boldsymbol{\theta}\right)\right]\\
 & =0\\
K^{''}\left(0\right) & =0+\text{Var}_{\boldsymbol{\theta}\sim g\left(\boldsymbol{\theta}\right)}\left[\phi_{g}\left(\boldsymbol{\theta}\right)-\phi_{f}\left(\boldsymbol{\theta}\right)\right]
\end{align}
and for $d\geq2$, the derivatives of $K\left(\lambda\right)$ coincide
with that of $\tilde{C}\left(\lambda\right)$:
\begin{align}
K^{\left(d\right)}\left(\lambda\right) & =\tilde{C}^{\left(d\right)}\left(\lambda\right)\\
 & =k_{d}\left(\lambda\right)
\end{align}

A Taylor expansion of $K\left(\lambda\right)$ to third order with
integral remainder then yields:
\begin{align}
K\left(\lambda\right) & =K\left(0\right)+\lambda K^{'}\left(0\right)+\frac{\lambda^{2}}{2}K^{''}\left(0\right)+\int_{0}^{\lambda}\frac{1}{2}l^{2}K^{\left(3\right)}\left(\lambda\right)dl\\
 & =0+0+\frac{\lambda^{2}}{2}\text{Var}_{\boldsymbol{\theta}\sim g\left(\boldsymbol{\theta}\right)}\left[\phi_{g}\left(\boldsymbol{\theta}\right)-\phi_{f}\left(\boldsymbol{\theta}\right)\right]+\int_{0}^{\lambda}\frac{1}{2}l^{2}k_{3}\left(l\right)dl
\end{align}
which concludes the proof.
\end{proof}
\begin{lem}
Bound on the cumulants.\label{lem:Bound-on-the cumulants.}

\textbf{Requires: }NA.

If $X$ is a bounded random variable: 
\begin{equation}
\left|X\right|\leq M
\end{equation}
Then its absolute cumulants for $k\in\left\{ 1,2,3\right\} $ are
bounded:
\begin{equation}
\E\left(\left|X-\E\left(X\right)\right|^{k}\right)\leq M^{k}
\end{equation}

\end{lem}

\begin{proof}
The proof proceeds in two steps. First, we prove that the worst case
is a Bernoulli random variable taking values $\pm M$ with probabilities
$p$ and $\left(1-p\right)$. Then, we derive the worst value of $p$
which gives the claimed bound.

The first step of the proof using an elegant ``coupling'' construction:
for any bounded variable $X$, we will construct a correlated variable
$Y$ which has the same mean but which is more extreme. $Y$ then
has higher cumulants than $X$.

We construct $Y$ conditional on $X$. Given $X=x$, we would like
the conditional mean of $Y$ to be equal to $x$. The following probabilities
achieve this goal:
\begin{align}
p & =\frac{\left(x+M\right)}{2M}\\
\mathbb{P}\left(Y=+M|X=x\right) & =p\\
\mathbb{P}\left(Y=-M|X=x\right) & =1-p
\end{align}
If $x=M$, then we always have $Y=x$. As $x$ decreases, $Y$ becomes
more variable and is maximally variable at $x=0$ where $p=0.5$.
As $x$ decreases further, $Y$ becomes less variable until $x=-M$
where we always have $Y=x=-M$.

Note that $\E\left(Y|X\right)=X$ and we thus have that both variables
have the same mean:
\begin{equation}
\E\left(Y\right)=\E\left(X\right)=e
\end{equation}

Now, for $k\geq1$, consider the function:
\begin{equation}
a_{k}\left(t\right)=\left|t-e\right|^{k}
\end{equation}
This function is convex and its expected value is the $k^{th}$ absolute
cumulant.

It is straightforward to compare the expected value of $a_{k}\left(X\right)$
and $a_{k}\left(Y\right)$. Indeed, compute the expected value of
$a_{k}\left(Y\right)$ by conditioning on $X$ as an intermediate
step:
\begin{equation}
\E\left[a_{k}\left(Y\right)\right]=\E\left[\E\left\{ a_{k}\left(Y\right)|X\right\} \right]
\end{equation}
Applying Jensen's inequality to the convex function $a_{k}\left(t\right)$
for $X=x$:
\begin{align}
\E\left\{ a_{k}\left(Y\right)|X=x\right\}  & \geq a_{k}\left(\E\left\{ Y|X=x\right\} \right)\\
 & \geq a_{k}\left(x\right)
\end{align}
We thus finally have:
\begin{align}
\E\left[a_{k}\left(Y\right)\right] & =\E\left[\E\left\{ a_{k}\left(Y\right)|X\right\} \right]\\
 & \geq\E\left[a_{k}\left(X\right)\right]
\end{align}
and we have proved that $Y$ has a higher cumulant than $X$. Since
$Y$ only takes the values $\pm M$, it is thus marginally a Bernoulli
random variable. Bernoulli random variables thus have the biggest
cumulants among bounded random variables.

It is now straightforward to compute the cumulant of a Bernoulli random
variable with probability $p$:
\begin{align}
\E\left(Y\right) & =Mp-\left(1-p\right)M\\
 & =M\left(2p-1\right)\\
\E\left(\left|Y-\E\left(Y\right)\right|^{k}\right) & =p\left|M-M\left(2p-1\right)\right|^{k}+\left(1-p\right)\left|-M-M\left(2p-1\right)\right|^{k}\\
 & =pM^{k}\left|2-2p\right|^{k}+\left(1-p\right)M^{k}\left|2p\right|^{k}\\
 & =2^{k}M^{k}\left[p\left(1-p\right)^{k}+\left(1-p\right)p^{k}\right]\\
 & =2^{k}M^{k}p\left(1-p\right)\left[\left(1-p\right)^{k-1}+p^{k-1}\right]
\end{align}
and we only need to optimize for $p$.

Consider the polynomial in $p$:
\begin{equation}
A\left(p\right)=p\left(1-p\right)\left[\left(1-p\right)^{k-1}+p^{k-1}\right]
\end{equation}

The gradient of $A\left(p\right)$ is:
\begin{align}
A^{'}\left(p\right) & =\left(1-p\right)\left[\left(1-p\right)^{k-1}+p^{k-1}\right]-p\left[\left(1-p\right)^{k-1}+p^{k-1}\right]\nonumber \\
 & \ \ \ \ +p\left(1-p\right)\left[-\left(k-1\right)\left(1-p\right)^{k-2}+\left(k-1\right)p^{k-2}\right]\\
 & =\left(1-2p\right)\left[\left(1-p\right)^{k-1}+p^{k-1}\right]+\left(k-1\right)p\left(1-p\right)\left[-\left(1-p\right)^{k-2}+p^{k-2}\right]
\end{align}
For $k=1$ or $k=2$, we have that $A^{'}\left(p\right)$ is monotone
with a unique zero at $p=0.5$.
\begin{align}
\text{If }k=1\ \ A^{'}\left(p\right) & =\left(1-2p\right)\left(2\right)+0\\
 & =2\left(1-2p\right)\\
\text{If }k=2\ \ A^{'}\left(p\right) & =\left(1-2p\right)\left(1-p+p\right)+p\left(1-p\right)\left(-1+1\right)\\
 & =\left(1-2p\right)
\end{align}
For $k=3$:
\begin{align}
A^{'}\left(p\right) & =\left(1-2p\right)\left[\left(1-p\right)^{2}+p^{2}\right]+2p\left(1-p\right)\left[-\left(1-p\right)+p\right]\\
 & =\left(1-2p\right)\left[\left(1-p\right)^{2}+p^{2}\right]+2p\left(1-p\right)\left[2p-1\right]\\
 & =\left(1-2p\right)\left[\left(1-p\right)^{2}+p^{2}-2p\left(1-p\right)\right]\\
 & =\left(1-2p\right)\left[\left(1-p-p\right)^{2}\right]\\
 & =\left(1-2p\right)^{3}
\end{align}
Once again $A^{\prime}\left(p\right)$ is monotone and has a unique
zero at $p=0.5$.

The maximum of $\E\left(\left|Y-\E\left(Y\right)\right|^{k}\right)$
for $k\in\left\{ 1,2,3\right\} $ is thus:
\begin{align}
\E\left(\left|Y-\E\left(Y\right)\right|^{k}\right) & \leq2^{k}M^{k}\frac{1}{2}\frac{1}{2}\left(\left(\frac{1}{2}\right)^{k-1}+\left(\frac{1}{2}\right)^{k-1}\right)\\
 & \leq2^{k}M^{k}\frac{2}{2^{k+1}}\\
 & \leq M^{k}
\end{align}
which gives the claimed bound.

\end{proof}
At this point, we are now ready to give the proof of Prop.\ref{prop: KL approx KL_var}.
\begin{proof}
The statement of Prop.\ref{prop: KL approx KL_var} combines elements
from the various Lemmas of this Appendix Section.

Lemma \ref{lem:The-h-lambda-family.} asserts that $h\left(\boldsymbol{\theta};\lambda\right)$
is indeed an exponential family.

Then, Lemma \ref{lem: approximation of KL(g,h)} (in the special case
$\lambda=1$) asserts that the KL divergence $KL\left(g,f\right)$
can be approximated using a Taylor expansion.

Finally, we get to the case for which $\left|\phi_{f}\left(\boldsymbol{\theta}\right)-\phi_{g}\left(\boldsymbol{\theta}\right)\right|\leq M$.
Lemma \ref{lem:Bound-on-the cumulants.} gives a bound on the absolute
cumulants of $\phi_{f}\left(\boldsymbol{\theta}\right)-\phi_{g}\left(\boldsymbol{\theta}\right)$:
they must be lower than $M^{3}$.

The Taylor expansion with integral remainder of $KL\left(g,f\right)$
can then be further simplified by bounding the integral remainder:
\begin{align}
\left|\int_{0}^{1}\frac{1}{2}l^{2}k_{3}\left(l\right)dl\right| & \leq\int_{0}^{1}\frac{1}{2}l^{2}\left|k_{3}\left(l\right)\right|dl\\
 & \leq\int_{0}^{1}\frac{1}{2}l^{2}M^{3}dl\\
 & \leq\frac{M^{3}}{6}
\end{align}
which concludes the proof.
\end{proof}

\section{The LSI-based approximations}

\label{sec:APPENDIX LSI based approximations}

This section deals with the proof of Prop.\ref{prop:LSI-bound-on KL(g,f)}
and \ref{prop:LSI-bound-on KL(f,g)} which give upper-bounds on the
forward and backward KL divergences based on the Log-Sobolev Inequality
(LSI).

Recall that throughout this section, $g\left(\boldsymbol{\theta}\right)$
is not restricted to being the Laplace approximation of $f\left(\boldsymbol{\theta}\right)$
and is not necessarily restricted to being Gaussian.

\subsection{Proof structure}

The proof of both propositions is fairly straightforward:
\begin{itemize}
\item First, I establish a variant of the classical LSI result which is
invariant to affine reparameterizations of the space. This is exactly
Prop.\ref{prop:LSI-bound-on KL(g,f)}
\item Finally, I prove Prop.\ref{prop:LSI-bound-on KL(f,g)}, as a straightforward
corollary from Prop.\ref{prop:LSI-bound-on KL(g,f)}.
\end{itemize}

\subsection{Proofs}

First, recall the classical statement of the LSI. Or, more precisely,
of the fact that strongly log-concave distributions obey the LSI.
\begin{thm}
LSI (\citet{bakry1985diffusions,otto2000generalization}).\label{thm: APPENDIX LSI-(BakryEmeryRef).}

If $f\left(\boldsymbol{\theta}\right)$ is a $\lambda$-strongly log-concave
density, i.e.
\begin{equation}
\forall\boldsymbol{\theta}\ H\phi_{f}\left(\boldsymbol{\theta}\right)\geq\lambda I_{p}
\end{equation}
 then, for any $g\left(\boldsymbol{\theta}\right)$, the reverse KL
divergence is bounded:
\begin{equation}
KL\left(g,f\right)\leq\frac{1}{2\lambda}\E_{\boldsymbol{\theta}\sim g\left(\boldsymbol{\theta}\right)}\left[\left\Vert \nabla\phi_{f}\left(\boldsymbol{\theta}\right)-\nabla\phi_{g}\left(\boldsymbol{\theta}\right)\right\Vert _{2}^{2}\right]
\end{equation}
\end{thm}

A trivial limit of this Theorem is that it is not invariant to affine
reparameterizations of the parameter space. Indeed, these modify the
gradients and $\lambda$ simultaneously in a way that does not lead
to the terms canceling.

However, this limit is trivial since it can be solved by making sure
that the inequality $\forall\boldsymbol{\theta}\ H\phi_{f}\left(\boldsymbol{\theta}\right)\geq\lambda I_{p}$
is tight in all eigenvalues at once. This yields Prop.\ref{prop:LSI-bound-on KL(g,f)}.
\begin{prop}
Main text Prop.\ref{prop:LSI-bound-on KL(g,f)}: an affine equivariant
LSI.\label{prop: APPENDIX: KL(g,f) bound LSI rephrasing}

\textbf{Requires: }Th.\ref{thm: APPENDIX LSI-(BakryEmeryRef).}.

If $f\left(\boldsymbol{\theta}\right)$ is strongly log-concave so
that there exists a strictly positive matrix $H_{min}$ such that:
\begin{equation}
\forall\boldsymbol{\theta}\ H\phi_{f}\left(\boldsymbol{\theta}\right)\geq H_{min}
\end{equation}
then, for any approximation $g\left(\boldsymbol{\theta}\right)$,
the reverse KL divergence is bounded:
\begin{equation}
KL\left(g,f\right)\leq\frac{1}{2}\E_{\boldsymbol{\theta}\sim g\left(\boldsymbol{\theta}\right)}\left[\left\{ \nabla\phi_{f}\left(\boldsymbol{\theta}\right)-\nabla\phi_{g}\left(\boldsymbol{\theta}\right)\right\} ^{T}\left(H_{min}\right)^{-1}\left\{ \nabla\phi_{f}\left(\boldsymbol{\theta}\right)-\nabla\phi_{g}\left(\boldsymbol{\theta}\right)\right\} \right]
\end{equation}
\end{prop}

\begin{proof}
Consider a matrix square-root of $\left[H_{min}\right]^{-1}$: $AA^{T}=\left[H_{min}\right]^{-1}$.
Then, consider the change of variable:
\begin{equation}
\mathring{\boldsymbol{\theta}}=A\boldsymbol{\theta}
\end{equation}
Log-gradients and log-Hessians in $\mathring{\boldsymbol{\theta}}$-space
are:
\begin{align}
\nabla_{\mathring{\boldsymbol{\theta}}}\mathring{\phi}_{f}\left(\mathring{\boldsymbol{\theta}}\right) & =A\nabla_{\boldsymbol{\theta}}\phi_{f}\left(A\boldsymbol{\theta}\right)\\
H_{\mathring{\boldsymbol{\theta}}}\mathring{\phi}_{f}\left(\mathring{\boldsymbol{\theta}}\right) & =A\ H\phi_{f}\left(A\boldsymbol{\theta}\right)\ A^{T}
\end{align}

In $\mathring{\boldsymbol{\theta}}$-space, the log-Hessian density
is still lower-bounded:
\begin{align}
\forall\mathring{\boldsymbol{\theta}}\ H_{\mathring{\boldsymbol{\theta}}}\mathring{\phi}_{f}\left(\mathring{\boldsymbol{\theta}}\right) & \geq AH_{min}A^{T}\\
 & \geq I_{p}
\end{align}
We can thus apply Thm.\ref{thm: APPENDIX LSI-(BakryEmeryRef).} with
$\lambda=1$, yielding:
\begin{align}
KL\left(\mathring{\boldsymbol{\theta}}_{g},\mathring{\boldsymbol{\theta}}_{f}\right) & \leq\frac{1}{2}\E_{\mathring{\boldsymbol{\theta}}\sim g\left(\mathring{\boldsymbol{\theta}}\right)}\left[\left\Vert \nabla_{\mathring{\boldsymbol{\theta}}}\mathring{\phi}_{f}\left(\mathring{\boldsymbol{\theta}}\right)-\nabla_{\mathring{\boldsymbol{\theta}}}\mathring{\phi}_{g}\left(\mathring{\boldsymbol{\theta}}\right)\right\Vert _{2}^{2}\right]\\
 & \leq\frac{1}{2}\E_{\boldsymbol{\theta}\sim g\left(\boldsymbol{\theta}\right)}\left[\left\Vert A\nabla_{\boldsymbol{\theta}}\phi_{f}\left(A\boldsymbol{\theta}\right)-A\nabla_{\boldsymbol{\theta}}\phi_{g}\left(A\boldsymbol{\theta}\right)\right\Vert _{2}^{2}\right]\\
 & \leq\frac{1}{2}\E_{\boldsymbol{\theta}\sim g\left(\boldsymbol{\theta}\right)}\left[\left\Vert A\left\{ \nabla_{\boldsymbol{\theta}}\phi_{f}\left(A\boldsymbol{\theta}\right)-\nabla_{\boldsymbol{\theta}}\phi_{g}\left(A\boldsymbol{\theta}\right)\right\} \right\Vert _{2}^{2}\right]\\
 & \leq\frac{1}{2}\E_{\boldsymbol{\theta}\sim g\left(\boldsymbol{\theta}\right)}\left[\left\{ \nabla\phi_{f}\left(\boldsymbol{\theta}\right)-\nabla\phi_{g}\left(\boldsymbol{\theta}\right)\right\} ^{T}AA^{T}\left\{ \nabla\phi_{f}\left(\boldsymbol{\theta}\right)-\nabla\phi_{g}\left(\boldsymbol{\theta}\right)\right\} \right]\\
 & \leq\frac{1}{2}\E_{\boldsymbol{\theta}\sim g\left(\boldsymbol{\theta}\right)}\left[\left\{ \nabla\phi_{f}\left(\boldsymbol{\theta}\right)-\nabla\phi_{g}\left(\boldsymbol{\theta}\right)\right\} ^{T}\left(H_{min}\right)^{-1}\left\{ \nabla\phi_{f}\left(\boldsymbol{\theta}\right)-\nabla\phi_{g}\left(\boldsymbol{\theta}\right)\right\} \right]
\end{align}
Finally, observe that the KL divergence is invariant to bijective
changes of variable:
\begin{align}
KL\left(g,f\right) & =KL\left(\mathring{\boldsymbol{\theta}}_{g},\mathring{\boldsymbol{\theta}}_{f}\right)\\
 & \leq\frac{1}{2}\E_{\boldsymbol{\theta}\sim g\left(\boldsymbol{\theta}\right)}\left[\left\{ \nabla\phi_{f}\left(\boldsymbol{\theta}\right)-\nabla\phi_{g}\left(\boldsymbol{\theta}\right)\right\} ^{T}\left(H_{min}\right)^{-1}\left\{ \nabla\phi_{f}\left(\boldsymbol{\theta}\right)-\nabla\phi_{g}\left(\boldsymbol{\theta}\right)\right\} \right]
\end{align}
which concludes the proof.
\end{proof}
\begin{prop}
Main text Prop.\ref{prop:LSI-bound-on KL(f,g)}: a straightforward
corollary.

\textbf{Requires: }Appendix Prop.\ref{prop: APPENDIX: KL(g,f) bound LSI rephrasing}
(Main text Prop.\ref{prop:LSI-bound-on KL(g,f)}).

If $g\left(\boldsymbol{\theta}\right)$ is a Gaussian with covariance
$\Sigma$, then for any $f\left(\boldsymbol{\theta}\right)$ the forward
KL divergence is bounded:
\begin{equation}
KL\left(f,g\right)\leq\frac{1}{2}\E_{\boldsymbol{\theta}\sim f\left(\boldsymbol{\theta}\right)}\left[\left\{ \nabla\phi_{f}\left(\boldsymbol{\theta}\right)-\nabla\phi_{g}\left(\boldsymbol{\theta}\right)\right\} ^{T}\Sigma\left\{ \nabla\phi_{f}\left(\boldsymbol{\theta}\right)-\nabla\phi_{g}\left(\boldsymbol{\theta}\right)\right\} \right]
\end{equation}
\end{prop}

\begin{proof}
This proof is a straightforward corollary of Appendix Prop.\ref{prop: APPENDIX: KL(g,f) bound LSI rephrasing}
(Main text Prop.\ref{prop:LSI-bound-on KL(g,f)}). It is derived by
inverting the role of $g$ and $f$.

Indeed, if $g\left(\boldsymbol{\theta}\right)$ is a Gaussian with
covariance $\Sigma$, then:
\begin{align}
\phi_{g} & =\frac{1}{2}\left(\boldsymbol{\theta}-\boldsymbol{\mu}\right)^{T}\Sigma^{-1}\left(\boldsymbol{\theta}-\boldsymbol{\mu}\right)\\
H\phi_{g} & =\Sigma^{-1}
\end{align}
and we can thus apply Appendix Prop.\ref{prop: APPENDIX: KL(g,f) bound LSI rephrasing}
with $H_{min}=\Sigma^{-1}$ yielding, with no assumption on $f\left(\boldsymbol{\theta}\right)$:
\begin{align}
KL\left(g,f\right) & \leq\frac{1}{2}\E_{\boldsymbol{\theta}\sim f\left(\boldsymbol{\theta}\right)}\left[\left\{ \nabla\phi_{f}\left(\boldsymbol{\theta}\right)-\nabla\phi_{g}\left(\boldsymbol{\theta}\right)\right\} ^{T}\left[\Sigma^{-1}\right]^{-1}\left\{ \nabla\phi_{f}\left(\boldsymbol{\theta}\right)-\nabla\phi_{g}\left(\boldsymbol{\theta}\right)\right\} \right]\\
 & \leq\frac{1}{2}\E_{\boldsymbol{\theta}\sim f\left(\boldsymbol{\theta}\right)}\left[\left\{ \nabla\phi_{f}\left(\boldsymbol{\theta}\right)-\nabla\phi_{g}\left(\boldsymbol{\theta}\right)\right\} ^{T}\Sigma\left\{ \nabla\phi_{f}\left(\boldsymbol{\theta}\right)-\nabla\phi_{g}\left(\boldsymbol{\theta}\right)\right\} \right]
\end{align}
which concludes the proof.
\end{proof}

\section{General deterministic Bound}

\label{sec: APPENDIX General-deterministic-Bound}

This section deals with the proof of Th.\ref{thm:A-deterministic-BvM theorem},
establishing the convergence of $KL\left(g,f\right)$ to $0$ in the
limit $\Delta_{3}p^{3/2}\rightarrow0$, as well as several approximations
of this quantity.

Throughout this section, we will simplify notation by using $g\left(\boldsymbol{\theta}\right)$
for the Laplace approximation of $f\left(\boldsymbol{\theta}\right)$
(dropping out the index from $g_{LAP}\left(\boldsymbol{\theta}\right)$).
We also assume that $f\left(\boldsymbol{\theta}\right)$ is strictly
log-concave. All results stated from this point onward are restricted
to this case.

We recall the notation:

\begin{align}
f\left(\boldsymbol{\theta}\right) & =\exp\left(-\phi_{f}\left(\boldsymbol{\theta}\right)-\log\left(Z_{f}\right)\right)\\
H\phi_{f}\left(\boldsymbol{\theta}\right) & >0_{p}\\
g\left(\boldsymbol{\theta}\right) & =\exp\left(-\phi_{g}\left(\boldsymbol{\theta}\right)-\log\left(Z_{g}\right)\right)\\
 & =\exp\left(-\frac{1}{2}\left(\boldsymbol{\theta}-\boldsymbol{\mu}\right)^{T}\Sigma^{-1}\left(\boldsymbol{\theta}-\boldsymbol{\mu}\right)-\log\left(Z_{g}\right)\right)\\
\boldsymbol{\mu} & =\text{argmin}_{\boldsymbol{\theta}}\left[\phi_{f}\left(\boldsymbol{\theta}\right)\right]\\
\Sigma^{-1} & =H\phi_{f}\left(\boldsymbol{\mu}\right)
\end{align}

\subsection{Proof structure}

The proof of this result is very involved.
\begin{enumerate}
\item First, I tackle aspects of the proof relating to the variables $r,\boldsymbol{e}$
and $c,\boldsymbol{e}$. I show the following results:
\begin{enumerate}
\item I establish that $KL\left(g,f\right)$ can be decomposed into the
contributions of the variables $r$ and $\boldsymbol{e}$ (Lemma \ref{lem: APPENDIX Decomposition-of-KL(g,f)}).
\item I give the distribution of $\boldsymbol{e},r$ under $g$ (Lemma \ref{lem: APPENDIX Distribution-of-r_g,e_g}).
\item I give the distribution of $c=\left(r\right)^{1/3}$ under $g$ (Lemma
\ref{lem: APPENDIX Distribution-of-c_g}).
\item I give the moments of $r$ under $g$ (Lemma \ref{lem: APPENDIX Moments-of-r_g}).
\item I give (Lemma \ref{lem: APPENDIX distributions under f}):
\begin{itemize}
\item The conditional distribution of $r$ under $f$.
\item The conditional distribution of $c$ under $f$.
\item The marginal distribution of $\boldsymbol{e}$ under $f$.
\item An approximation to the marginal distribution of $\boldsymbol{e}$
under $f$.
\end{itemize}
\end{enumerate}
\item Then, I tackle the $KL\left(r_{g},r_{f}|\boldsymbol{e}\right)$ term
and derive the following results:
\begin{enumerate}
\item I prove that $f\left(c\right)$ is strongly log-concave (Lemma \ref{lem: f(c|e) is SLC}).
\item I give a lower-bound of the log-curvature of $f\left(c\right)$ (Lemmas
\ref{lem: APPENDIX detailed lower bound for log-curvature c geq c_0}
and \ref{lem: APPENDIX detailed lower bound for log-curvature c LEQ c_0}).
\item I give the asymptotic scaling of the lower-bound (Lemma \ref{lem: APPENDIX approximation of psi_min^sec}).
\item I give an approximation of $KL\left(r_{g},r_{f}|\boldsymbol{e}\right)$
based on Appendix Prop.\ref{prop: APPENDIX: KL(g,f) bound LSI rephrasing}
(Main text Prop.\ref{prop:LSI-bound-on KL(g,f)}) (Prop.\ref{prop: APPENDIX Bound-on- KL r_g r_f | e}).
\end{enumerate}
\item Then, I tackle the $KL\left(\boldsymbol{e}_{g},\boldsymbol{e}_{f}\right)$
term. I further show the following results:
\begin{enumerate}
\item I give an approximation of $\log\left[f\left(\boldsymbol{e}\right)\right]$
(Lemma \ref{lem: APPENDIX approx quality of ELBO}).
\item I give an approximation of $KL\left(\boldsymbol{e}_{g},\boldsymbol{e}_{f}\right)$
based on Appendix Prop.\ref{prop: APPENDIX: KL(g,f) bound LSI rephrasing}
(Main text Prop.\ref{prop:LSI-bound-on KL(g,f)}) (Prop.\ref{Prop: APPENDIX ELBO based approx of KL(e_g,e_f)}).
\end{enumerate}
\item Finally, I relate the bound on $KL\left(\boldsymbol{e}_{g},\boldsymbol{e}_{f}\right)$
to the KL variance (Lemma \ref{lem:Relationship-between-Var(ELBO) and KL-var}).
\end{enumerate}
Theorem. \ref{thm:A-deterministic-BvM theorem} is finally derived
as a combination of all the results proved in this section.

\subsection{Proofs: properties of the change of variable}
\begin{lem}
Decomposition of $KL\left(g,f\right)$.\label{lem: APPENDIX Decomposition-of-KL(g,f)}

\textbf{Requires:} NA.

The KL divergence can be decomposed as:
\begin{equation}
KL\left(g,f\right)=KL\left(\boldsymbol{e}_{g},\boldsymbol{e}_{f}\right)+\E_{\boldsymbol{e}\sim g\left(\boldsymbol{e}\right)}\left[KL\left(r_{g},r_{f}|\boldsymbol{e}\right)\right]
\end{equation}
\end{lem}

\begin{proof}
This proof consists of a simple manipulation of the normal expression
for the KL divergence.

We simply make appear the variables $r,\boldsymbol{e}$:
\begin{align}
KL\left(g,f\right) & =\E_{\boldsymbol{\theta}\sim g\left(\boldsymbol{\theta}\right)}\left[\log\frac{f\left(\boldsymbol{\theta}\right)}{g\left(\boldsymbol{\theta}\right)}\right]\\
 & =\E_{r,\boldsymbol{e}\sim g\left(r,\boldsymbol{e}\right)}\left[\log\frac{f\left(r,\boldsymbol{e}\right)}{g\left(r,\boldsymbol{e}\right)}\right]
\end{align}
and then use the conditional probability formula:
\begin{align}
f\left(r,\boldsymbol{e}\right) & =f\left(\boldsymbol{e}\right)f\left(r|\boldsymbol{e}\right)\\
g\left(r,\boldsymbol{e}\right) & =g\left(\boldsymbol{e}\right)g\left(r|\boldsymbol{e}\right)
\end{align}
yielding:
\begin{align}
KL\left(g,f\right) & =\E_{r,\boldsymbol{e}\sim g\left(r,\boldsymbol{e}\right)}\left[\log\frac{f\left(\boldsymbol{e}\right)f\left(r|\boldsymbol{e}\right)}{g\left(\boldsymbol{e}\right)g\left(r|\boldsymbol{e}\right)}\right]\\
 & =\E_{r,\boldsymbol{e}\sim g\left(r,\boldsymbol{e}\right)}\left[\log\frac{f\left(\boldsymbol{e}\right)}{g\left(\boldsymbol{e}\right)}+\log\frac{f\left(r|\boldsymbol{e}\right)}{g\left(r|\boldsymbol{e}\right)}\right]\\
 & =\E_{\boldsymbol{e}\sim g\left(\boldsymbol{e}\right)}\left[\log\frac{f\left(\boldsymbol{e}\right)}{g\left(\boldsymbol{e}\right)}\right]+\E_{r,\boldsymbol{e}\sim g\left(r,\boldsymbol{e}\right)}\left[\log\frac{f\left(r|\boldsymbol{e}\right)}{g\left(r|\boldsymbol{e}\right)}\right]\\
 & =KL\left(\boldsymbol{e}_{g},\boldsymbol{e}_{f}\right)+\E_{\boldsymbol{e}\sim g\left(\boldsymbol{e}\right)}\left\{ \E_{r\sim g\left(r|\boldsymbol{e}\right)}\left[\log\frac{f\left(r|\boldsymbol{e}\right)}{g\left(r|\boldsymbol{e}\right)}\right]\right\} \\
 & =KL\left(\boldsymbol{e}_{g},\boldsymbol{e}_{f}\right)+\E_{\boldsymbol{e}\sim g\left(\boldsymbol{e}\right)}\left[KL\left(r_{g},r_{f}|\boldsymbol{e}\right)\right]
\end{align}
which concludes the proof.
\end{proof}
\begin{lem}
Distribution of $r_{g},\boldsymbol{e}_{g}$.\label{lem: APPENDIX Distribution-of-r_g,e_g}

\textbf{Requires:} NA.

The random variables $r_{g},\boldsymbol{e}_{g}$ are independent.
$\boldsymbol{e}_{g}$ has a uniform distribution over the unit sphere
$S^{p-1}$ while $r_{g}$ has a $\chi_{p}$ distribution.
\end{lem}

\begin{proof}
This proof is a standard of statistics and probability theory.

Start from $\tilde{\boldsymbol{\theta}}_{g}$ which is a standard
Gaussian random variable:
\begin{equation}
g\left(\tilde{\boldsymbol{\theta}}\right)\propto\exp\left(-\frac{1}{2}\left\Vert \tilde{\boldsymbol{\theta}}\right\Vert _{2}^{2}\right)
\end{equation}

The volume of the element $drd\boldsymbol{e}$ scales as $r^{p-1}$
because it lives on the surface of a $p$-dimensional sphere of radius
$r$. The standard change of variable formula then yields:
\begin{align}
g\left(r,\boldsymbol{e}\right) & \propto r^{p-1}\exp\left(-\frac{1}{2}\left\Vert r\boldsymbol{e}\right\Vert _{2}^{2}\right)\\
 & \propto r^{p-1}\exp\left(-\frac{1}{2}r^{2}\right)
\end{align}
which proves the well-known result:
\begin{itemize}
\item $\boldsymbol{e}_{g}$ is a uniform random variable over its support:
the unit sphere $S^{p-1}$.
\item $r_{g}$ is a $\chi$ (``chi'') random variable with $p$ degrees
of freedom.
\item They are independent.
\end{itemize}
and concludes the proof.
\end{proof}
\begin{lem}
Distribution of $c_{g}$.\label{lem: APPENDIX Distribution-of-c_g}

\textbf{Requires:} Lemma \ref{lem: APPENDIX Distribution-of-r_g,e_g}.

Variable $c_{g}=\left(r_{g}\right)^{1/3}$ follows a $\chi_{p}^{1/3}$
(``chi-one-third'') distribution:
\begin{equation}
g\left(c\right)\propto c^{3p-1}\exp\left(-\frac{c^{6}}{2}\right)
\end{equation}
\end{lem}

\begin{proof}
This follows from Lemma \ref{lem: APPENDIX Distribution-of-r_g,e_g}
through another standard change of variable:
\begin{align}
g\left(c\right) & \propto c^{3\left(p-1\right)}\exp\left(-\frac{c^{6}}{2}\right)\frac{dr}{dc}\\
 & \propto c^{3p-3}\exp\left(-\frac{c^{6}}{2}\right)\left(3c^{2}\right)\\
 & \propto c^{3p-1}\exp\left(-\frac{c^{6}}{2}\right)
\end{align}
which concludes the proof.
\end{proof}
\begin{lem}
Moments of $r_{g}$.\label{lem: APPENDIX Moments-of-r_g}

\textbf{Requires:} Lemma \ref{lem: APPENDIX Distribution-of-r_g,e_g}.

The $0$-centered moments of $r_{g}$ are:
\begin{align}
\E\left(r_{g}^{k}\right) & =2^{k/2}\frac{\Gamma\left(\frac{p+k}{2}\right)}{\Gamma\left(\frac{p}{2}\right)}
\end{align}
For fixed $k$, as $p\rightarrow\infty$, the moments of $r_{g}$
scale asymptotically as:
\begin{equation}
\E\left(r_{g}^{k}\right)=p^{k/2}+\mathcal{O}\left(p^{k/2-1}\right)
\end{equation}
i.e. the ratio $r_{g}/p^{1/2}$ converges in $L_{k}$ to the constant
$1$.
\end{lem}

\begin{proof}
The following derivation of the moments of $r_{g}$ is standard.

The normalization constant of the $\chi_{p}$ distribution is:
\begin{equation}
\int r^{p-1}\exp\left(-\frac{r^{2}}{2}\right)=2^{p/2-1}\Gamma\left(\frac{p}{2}\right)
\end{equation}
Thus, we have the expected value of $r^{k}$ as:
\begin{align}
\E\left[\left(r_{g}\right)^{k}\right] & =\frac{\int r^{k}r^{p-1}\exp\left(-\frac{r^{2}}{2}\right)}{\int r^{p-1}\exp\left(-\frac{r^{2}}{2}\right)}\\
 & =\frac{2^{\left(p+k\right)/2-1}\Gamma\left(\frac{p+k}{2}\right)}{2^{p/2-1}\Gamma\left(\frac{p}{2}\right)}\\
 & =2^{k/2}\frac{\Gamma\left(\frac{p+k}{2}\right)}{\Gamma\left(\frac{p}{2}\right)}
\end{align}
\end{proof}
In order to approximate these moments in the limit $p\rightarrow\infty$
while $k$ remains fixed, we can consider the following approximation
of the ratio of the Gamma function (see \citet{mortici2010new} and
additional references therein): for all $s\in\left[0,1\right]$
\begin{equation}
n^{1-s}\leq\frac{\Gamma\left(n+1\right)}{\Gamma\left(n+s\right)}\leq\left(n+1\right)^{1-s}\label{eq: Gautschi gang formula}
\end{equation}
yielding:
\begin{align}
\frac{\Gamma\left(n+1\right)}{\Gamma\left(n+s\right)} & =n^{1-s}+\mathcal{O}\left(n^{1-s-1}\right)\\
 & =n^{1-s}+\mathcal{O}\left(n^{-s}\right)
\end{align}
in the limit $n\rightarrow\infty$.

Noting $\left\lfloor k/2\right\rfloor $ to be the highest integer
smaller than $k/2$, we then further use the recurrence relationship
of the Gamma function:
\begin{align}
\frac{\Gamma\left[\left(p+k\right)/2\right]}{\Gamma\left[p/2\right]} & =\frac{\Gamma\left[p/2+\left\lfloor k/2\right\rfloor \right]}{\Gamma\left[p/2+\left\lfloor k/2\right\rfloor \right]}\frac{\Gamma\left[p/2+k/2\right]}{\Gamma\left[p/2\right]}\\
 & =\frac{\Gamma\left[p/2+k/2\right]}{\Gamma\left[p/2+\left\lfloor k/2\right\rfloor \right]}\frac{\Gamma\left[p/2+\left\lfloor k/2\right\rfloor \right]}{\Gamma\left[p/2\right]}\\
 & =\frac{\Gamma\left[p/2+k/2\right]}{\Gamma\left[p/2+\left\lfloor k/2\right\rfloor \right]}\prod_{i=1}^{\left\lfloor k/2\right\rfloor }\frac{\left(p+i\right)}{2}
\end{align}
where the product scales asymptotically as:
\begin{equation}
\prod_{i=1}^{\left\lfloor k/2\right\rfloor }\frac{\left(p+i\right)}{2}=\frac{1}{2^{\left\lfloor k/2\right\rfloor }}p^{\left\lfloor k/2\right\rfloor }+\mathcal{O}\left(p^{\left\lfloor k/2\right\rfloor -1}\right)
\end{equation}
while the ratio scales as:
\begin{equation}
\frac{\Gamma\left[p/2+k/2\right]}{\Gamma\left[p/2+\left\lfloor k/2\right\rfloor \right]}=\left(\frac{p}{2}\right)^{k/2-\left\lfloor k/2\right\rfloor }+\mathcal{O}\left(p^{k/2-\left\lfloor k/2\right\rfloor -1}\right)
\end{equation}

Combining both asymptotic scalings sums the exponents, yielding:
\begin{equation}
\frac{\Gamma\left[\left(p+k\right)/2\right]}{\Gamma\left[p/2\right]}=\left(\frac{p}{2}\right)^{k/2}+\mathcal{O}\left(p^{k/2-1}\right)
\end{equation}
which yields the asymptotic scaling of the moments:
\begin{align}
\E\left(r_{g}^{k}\right) & =2^{k/2}\frac{\Gamma\left(\frac{p+k}{2}\right)}{\Gamma\left(\frac{p}{2}\right)}\\
 & =p^{k/2}+\mathcal{O}\left(p^{k/2-1}\right)
\end{align}
which concludes the proof.
\begin{lem}
Distribution of $r_{f}|\boldsymbol{e}$, $c_{f}|\boldsymbol{e}$ and
$\boldsymbol{e}_{f}$.\label{lem: APPENDIX distributions under f}

\textbf{Requires:} NA.

The conditional distribution of $r_{f}|\boldsymbol{e}$ is:
\begin{equation}
f\left(r|\boldsymbol{e}\right)\propto r^{p-1}\exp\left(-\tilde{\phi}_{f}\left(r\boldsymbol{e}\right)\right)
\end{equation}

The conditional distribution of $c_{f}|\boldsymbol{e}$ is:
\begin{equation}
f\left(c|\boldsymbol{e}\right)\propto c^{3p-1}\exp\left(-\tilde{\phi}_{f}\left(c^{3}\boldsymbol{e}\right)\right)
\end{equation}

The marginal distribution of $\boldsymbol{e}_{f}$ is:
\begin{equation}
f\left(\boldsymbol{e}\right)\propto\int r^{p-1}\exp\left(-\tilde{\phi}_{f}\left(r\boldsymbol{e}\right)\right)dr
\end{equation}
where the proportionality sign hides a term that does not depend on
$\boldsymbol{e}$. The marginal distribution of $\boldsymbol{e}_{f}$
can be approximated using the ELBO as:
\begin{align}
\left|\log\left[f\left(\boldsymbol{e}\right)\right]-ELBO\left(\boldsymbol{e}\right)\right| & \leq KL\left(r_{g},r_{f}|\boldsymbol{e}\right)\\
ELBO\left(\boldsymbol{e}\right) & =\E_{r\sim g\left(r\right)}\left\{ \log\left[\frac{f\left(r,\boldsymbol{e}\right)}{g\left(r\right)}\right]\right\} \\
 & =-\E_{r\sim g\left(r\right)}\left[\tilde{\phi}_{f}\left(r\boldsymbol{e}\right)-\frac{1}{2}r^{2}\right]+C
\end{align}
where $C$ is a constant that does not depend on $\boldsymbol{e}$.
\end{lem}

\begin{proof}
Like Lemma \ref{lem: APPENDIX Distribution-of-r_g,e_g} and \ref{lem: APPENDIX Distribution-of-c_g},
the first two claims follow from a standard change of variable:
\begin{align}
f\left(\tilde{\boldsymbol{\theta}}\right) & \propto\exp\left(-\tilde{\phi}_{f}\left(\tilde{\boldsymbol{\theta}}\right)\right)\\
f\left(r,\boldsymbol{e}\right) & \propto r^{p-1}\exp\left(-\tilde{\phi}_{f}\left(r\boldsymbol{e}\right)\right)\label{eq: APPENDIX intermediate joint distribution r,e}
\end{align}
thus yielding:
\begin{align}
f\left(r|\boldsymbol{e}\right) & =\frac{f\left(r,\boldsymbol{e}\right)}{f\left(\boldsymbol{e}\right)}\\
 & \propto r^{p-1}\exp\left(-\tilde{\phi}_{f}\left(r\boldsymbol{e}\right)\right)
\end{align}

Similarly:
\begin{align}
f\left(c|\boldsymbol{e}\right) & \propto c^{3\left(p-1\right)}\exp\left(-\tilde{\phi}_{f}\left(c^{3}\boldsymbol{e}\right)\right)\frac{dr}{dc}\\
 & \propto c^{3\left(p-1\right)}\exp\left(-\tilde{\phi}_{f}\left(c^{3}\boldsymbol{e}\right)\right)\left(3c^{2}\right)\\
 & \propto c^{3p-1}\exp\left(-\tilde{\phi}_{f}\left(c^{3}\boldsymbol{e}\right)\right)
\end{align}

$f\left(\boldsymbol{e}\right)$ is the integral of $f\left(r,\boldsymbol{e}\right)$:
\begin{align}
f\left(\boldsymbol{e}\right) & =\int f\left(r,\boldsymbol{e}\right)dr
\end{align}
Combining with eq.\eqref{eq: APPENDIX intermediate joint distribution r,e},
we obtain:
\begin{equation}
f\left(\boldsymbol{e}\right)=\frac{\int r^{p-1}\exp\left(-\tilde{\phi}_{f}\left(r\boldsymbol{e}\right)\right)dr}{\int d\boldsymbol{e}\left[\int r^{p-1}\exp\left(-\tilde{\phi}_{f}\left(r\boldsymbol{e}\right)\right)dr\right]}
\end{equation}

Finally, the Evidence Lower Bound (ELBO; \citet{blei2017variational})
is a well-known lower-bound on the log-integral of a density:
\begin{align}
\log\left[f\left(\boldsymbol{e}\right)\right] & =\log\left[\int f\left(r,\boldsymbol{e}\right)dr\right]\\
 & =\log\left[\int g\left(r\right)\frac{f\left(r,\boldsymbol{e}\right)}{g\left(r\right)}dr\right]\\
 & =\log\left[\E_{r\sim g\left(r\right)}\left\{ \frac{f\left(r,\boldsymbol{e}\right)}{g\left(r\right)}\right\} \right]\\
 & \geq\E_{r\sim g\left(r\right)}\left\{ \log\left[\frac{f\left(r,\boldsymbol{e}\right)}{g\left(r\right)}\right]\right\} \\
 & \geq ELBO\left(\boldsymbol{e}\right)
\end{align}
where we have used Jensen's inequality on the concave function $t\rightarrow\log\left(t\right)$.
Strictly speaking, the ELBO of $\log\left[f\left(\boldsymbol{e}\right)\right]$
is this exact expression: $\E_{r\sim g\left(r\right)}\left\{ \log\left[\frac{f\left(r,\boldsymbol{e}\right)}{g\left(r\right)}\right]\right\} $
and the error is precisely equal to $KL\left(r_{g},r_{f}|\boldsymbol{e}\right)$.

We can rework the ELBO slightly to remove terms that are constants
in $\boldsymbol{e}$:
\begin{align}
\log\left[\frac{f\left(r,\boldsymbol{e}\right)}{g\left(r\right)}\right] & =\log\left[\frac{r^{p-1}\exp\left(-\tilde{\phi}_{f}\left(r\boldsymbol{e}\right)\right)}{\int r^{p-1}\exp\left(-\tilde{\phi}_{f}\left(r\boldsymbol{e}\right)\right)drd\boldsymbol{e}}\frac{\int r^{p-1}\exp\left(-\frac{1}{2}r^{2}\right)drd\boldsymbol{e}}{r^{p-1}\exp\left(-\frac{1}{2}r^{2}\right)}\right]\\
 & =\log\left[\frac{r^{p-1}\exp\left(-\tilde{\phi}_{f}\left(r\boldsymbol{e}\right)\right)}{r^{p-1}\exp\left(-\frac{1}{2}r^{2}\right)}\frac{\int r^{p-1}\exp\left(-\frac{1}{2}r^{2}\right)drd\boldsymbol{e}}{\int r^{p-1}\exp\left(-\tilde{\phi}_{f}\left(r\boldsymbol{e}\right)\right)drd\boldsymbol{e}}\right]\\
 & =\log\left[\exp\left(\frac{1}{2}r^{2}-\tilde{\phi}_{f}\left(r\boldsymbol{e}\right)\right)\right]+C\\
 & =\frac{1}{2}r^{2}-\tilde{\phi}_{f}\left(r\boldsymbol{e}\right)+C\\
 & =-\left\{ \tilde{\phi}_{f}\left(r\boldsymbol{e}\right)-\frac{1}{2}r^{2}\right\} +C\\
\E_{r\sim g\left(r\right)}\left\{ \log\left[\frac{f\left(r,\boldsymbol{e}\right)}{g\left(r\right)}\right]\right\}  & =-\E_{r\sim g\left(r\right)}\left\{ \tilde{\phi}_{f}\left(r\boldsymbol{e}\right)-\frac{1}{2}r^{2}\right\} +C
\end{align}
which yields the claimed result. Note that the $\frac{1}{2}r^{2}$
term could be dropped out too, but it corresponds to the Taylor expansion
of $\tilde{\phi}_{f}\left(r\boldsymbol{e}\right)$ to second order.
It thus makes sense to keep this term. 

Let us finally check that the gap is indeed the KL divergence:
\begin{align}
\E_{r\sim g\left(r\right)}\left\{ \log\left[\frac{f\left(r,\boldsymbol{e}\right)}{g\left(r\right)}\right]\right\}  & =\E_{r\sim g\left(r\right)}\left\{ \log\left[\frac{f\left(r|\boldsymbol{e}\right)f\left(\boldsymbol{e}\right)}{g\left(r\right)}\right]\right\} \\
 & =\E_{r\sim g\left(r\right)}\left\{ \log\left[\frac{f\left(r|\boldsymbol{e}\right)}{g\left(r\right)}\right]\right\} +\log\left[f\left(\boldsymbol{e}\right)\right]\\
\log\left[f\left(\boldsymbol{e}\right)\right]-\E_{r\sim g\left(r\right)}\left\{ \log\left[\frac{f\left(r,\boldsymbol{e}\right)}{g\left(r\right)}\right]\right\}  & =-\E_{r\sim g\left(r\right)}\left\{ \log\left[\frac{f\left(r|\boldsymbol{e}\right)}{g\left(r\right)}\right]\right\} \\
 & =\E_{r\sim g\left(r\right)}\left\{ \log\left[\frac{g\left(r\right)}{f\left(r|\boldsymbol{e}\right)}\right]\right\} \\
 & =KL\left(r_{g},r_{f}|\boldsymbol{e}\right)
\end{align}
which concludes the proof.
\end{proof}

\subsection{Proofs: $KL\left(r_{g},r_{f}|\boldsymbol{e}\right)$ term.}

For this section, we introduce two additional notations:
\begin{itemize}
\item Let $\varphi_{\boldsymbol{e}}\left(r\right)$ be the value of $\tilde{\phi}_{f}$
along direction $\boldsymbol{e}$:
\begin{align}
\varphi_{\boldsymbol{e}}\left(r\right) & =\tilde{\phi}_{f}\left(r\boldsymbol{e}\right)\\
 & =\phi_{f}\left(\boldsymbol{\mu}+r\Sigma_{g}^{1/2}\boldsymbol{e}\right)
\end{align}
\item Let $\psi_{f}\left(c|\boldsymbol{e}\right)$ denote the (unnormalized)
negative log-density of $f\left(c|\boldsymbol{e}\right)$:
\begin{align}
\psi_{f}\left(c|\boldsymbol{e}\right) & =-\log\left[f\left(c|\boldsymbol{e}\right)\right]+C\\
 & =-\left(3p-1\right)\log\left(c\right)+\tilde{\phi}_{f}\left(c^{3}\boldsymbol{e}\right)\\
 & =-\left(3p-1\right)\log\left(c\right)+\varphi_{\boldsymbol{e}}\left(c^{3}\right)
\end{align}
Similarly, let $\psi_{g}\left(c\right)$ be the negative log-density
of $g\left(c\right)$:
\begin{equation}
\psi_{g}\left(c\right)=-\left(3p-1\right)\log\left(c\right)+\frac{1}{2}c^{6}
\end{equation}
where we have used Lemmas \ref{lem: APPENDIX Distribution-of-c_g}
and \ref{lem: APPENDIX distributions under f}.
\end{itemize}
\begin{lem}
Basic properties of $\varphi_{\boldsymbol{e}}\left(r\right)$.\label{lem:Basic-properties-of varphi_e(r)}

\textbf{Requires:} Lemma \ref{lem: APPENDIX distributions under f}.

$\varphi_{\boldsymbol{e}}\left(r\right)$ is a convex function. Its
higher derivatives are controlled:
\[
\left|\varphi_{\boldsymbol{e}}^{\left(3\right)}\left(r\right)\right|\leq\Delta_{3}
\]

For $r\leq r_{0}=\left(\Delta_{3}\right)^{-1}$, $\varphi_{\boldsymbol{e}}\left(r\right)$
and its derivatives are bounded through a Taylor expansion:
\begin{align}
\varphi_{\boldsymbol{e}}\left(r\right) & \geq\frac{r^{2}}{2}-\frac{\Delta_{3}}{6}r^{3}\\
\varphi_{\boldsymbol{e}}^{'}\left(r\right) & \geq r-\frac{\Delta_{3}}{2}r^{2}\\
\varphi_{\boldsymbol{e}}^{''}\left(r\right) & \geq1-\Delta_{3}r
\end{align}
For $r\geq r_{0}=\left(\Delta_{3}\right)^{-1}$, $\varphi_{\boldsymbol{e}}\left(r\right)$
and its derivatives are bounded instead through the convexity of $\varphi_{\boldsymbol{e}}\left(r\right)$:
\begin{align}
\varphi_{\boldsymbol{e}}\left(r\right) & \geq\frac{1}{3}\frac{1}{\left(\Delta_{3}\right)^{2}}+\frac{1}{2\Delta_{3}}\left(r-\frac{1}{\Delta_{3}}\right)\\
\varphi_{\boldsymbol{e}}^{'}\left(r\right) & \geq\frac{1}{2\Delta_{3}}\\
\varphi_{\boldsymbol{e}}^{''}\left(r\right) & \geq0
\end{align}

Let $\varphi_{\text{min}}\left(r\right)$ be the function corresponding
to the lower-bounds.
\end{lem}

\begin{proof}
This Lemma is proved through several straightforward manipulations
of the expression for $\varphi_{\boldsymbol{e}}\left(r\right)$.

$\varphi_{\boldsymbol{e}}\left(r\right)$ is defined as the restriction
of a convex function, $\tilde{\phi}_{f}$, to an arc:
\begin{equation}
\varphi_{\boldsymbol{e}}\left(r\right)=\phi_{f}\left(r\boldsymbol{e}\right)
\end{equation}
It is thus also convex.

Its derivatives are straightforward to compute:
\begin{align}
\varphi_{\boldsymbol{e}}\left(r\right) & =\tilde{\phi}_{f}\left(r\boldsymbol{e}\right)\\
\varphi_{\boldsymbol{e}}^{'}\left(r\right) & =\boldsymbol{e}^{T}\nabla\tilde{\phi}_{f}\left(r\boldsymbol{e}\right)\\
\varphi_{\boldsymbol{e}}^{''}\left(r\right) & =\boldsymbol{e}^{T}\left[H\tilde{\phi}_{f}\left(r\boldsymbol{e}\right)\right]\boldsymbol{e}\\
\varphi_{\boldsymbol{e}}^{\left(3\right)}\left(r\right) & =\tilde{\phi}_{f}^{\left(3\right)}\left(r\boldsymbol{e}\right)\left[\boldsymbol{e},\boldsymbol{e},\boldsymbol{e}\right]
\end{align}
From which we obtain the value of the derivatives for $r=0$:
\begin{align}
\varphi_{\boldsymbol{e}}^{'}\left(0\right) & =0\\
\varphi_{\boldsymbol{e}}^{''}\left(0\right) & =1
\end{align}
and a bound on the third derivative:
\begin{align}
\left|\varphi_{\boldsymbol{e}}^{\left(3\right)}\left(r\right)\right| & \leq\max_{r\geq0}\left|\tilde{\phi}_{f}^{\left(3\right)}\left(r\boldsymbol{e}\right)\left[\boldsymbol{e},\boldsymbol{e},\boldsymbol{e}\right]\right|\\
 & \leq\max_{\tilde{\boldsymbol{\theta}}}\left\Vert \tilde{\phi}_{f}^{\left(3\right)}\left(\tilde{\boldsymbol{\theta}}\right)\right\Vert _{\text{max}}\\
 & \leq\Delta_{3}
\end{align}

A Taylor expansion centered at $0$ and using the bound on $\left|\varphi_{\boldsymbol{e}}^{\left(3\right)}\left(r\right)\right|$
then gives a lower-bound for $\varphi_{\boldsymbol{e}}^{''}\left(r\right)$:
\begin{equation}
\varphi_{\boldsymbol{e}}^{''}\left(r\right)\geq1-\Delta_{3}r
\end{equation}
This lower-bound is suboptimal for $r\geq r_{0}=\left(\Delta_{3}\right)^{-1}$
for which we already know that $\varphi_{\boldsymbol{e}}^{''}\left(r\right)\geq0$
instead, due to the convexity of $\tilde{\phi}_{f}$ and thus that
of $\varphi_{\boldsymbol{e}}\left(r\right)$. This gives a general
lower-bound for $\varphi_{\boldsymbol{e}}^{''}\left(r\right)$:
\begin{equation}
\varphi_{\boldsymbol{e}}^{''}\left(r\right)\geq\max\left(1-\Delta_{3}r,0\right)
\end{equation}
Let $\varphi_{\text{min}}\left(r\right)$ be the function defined
through the following equations:
\begin{align}
\varphi_{\text{min}}^{''}\left(r\right) & =\max\left(1-\Delta_{3}r,0\right)\\
\varphi_{\text{min}}^{'}\left(0\right) & =0\\
\varphi_{\text{min}}\left(0\right) & =0
\end{align}
$\varphi_{\text{min}}$ coincides at $0$ with $\varphi_{\boldsymbol{e}}$
but has strictly lower second derivative. We can thus find a lower-bound
for $\varphi_{\boldsymbol{e}}^{'}\left(r\right)$ and $\varphi_{\boldsymbol{e}}\left(r\right)$
by integrating $\varphi_{min}\left(r\right)$. For $r\leq r_{0}$:
\begin{align}
\varphi_{\boldsymbol{e}}^{'}\left(r\right) & \geq\varphi_{\text{min}}^{'}\left(r\right)=r-\frac{\Delta_{3}}{2}r^{2}\\
\varphi_{\boldsymbol{e}}\left(r\right) & \geq\varphi_{\text{min}}\left(r\right)=\frac{r^{2}}{2}-\frac{\Delta_{3}}{6}r^{3}
\end{align}
while for $r\geq r_{0}$:
\begin{align}
\varphi_{\boldsymbol{e}}^{'}\left(r\right) & \geq\varphi_{\text{min}}^{'}\left(r\right)=\frac{1}{2\Delta_{3}}\\
\varphi_{\boldsymbol{e}}\left(r\right) & \geq\varphi_{\text{min}}\left(r\right)=\frac{1}{3}\frac{1}{\left(\Delta_{3}\right)^{2}}+\frac{1}{2\Delta_{3}}\left(r-\frac{1}{\Delta_{3}}\right)
\end{align}
which concludes the proof.
\end{proof}
\begin{lem}
$f\left(c|\boldsymbol{e}\right)$ is strongly log-concave.\label{lem: f(c|e) is SLC}

\textbf{Requires:} Lemma \ref{lem:Basic-properties-of varphi_e(r)}.

The derivatives of $\psi_{f}\left(c|\boldsymbol{e}\right)$ are:
\begin{align}
\psi_{f}^{'}\left(c|\boldsymbol{e}\right) & =-\frac{3p-1}{c}+3c^{2}\varphi_{\boldsymbol{e}}^{'}\left(c^{3}\right)\label{eq: psi_f prime(c)}\\
\psi_{f}^{''}\left(c|\boldsymbol{e}\right) & =\frac{3p-1}{c^{2}}+6c\varphi_{\boldsymbol{e}}^{'}\left(c^{3}\right)+9c^{4}\varphi_{\boldsymbol{e}}^{'}\left(c^{3}\right)\label{eq: psi_f sec(c)}
\end{align}

The second-derivative of $\psi_{f}\left(c|\boldsymbol{e}\right)$
is lower-bounded:
\[
\psi_{f}^{''}\left(c|\boldsymbol{e}\right)\geq\begin{cases}
\frac{3p-1}{c^{2}}+\frac{3c}{\Delta_{3}} & \text{ if }c\geq c_{0}=\left(\Delta_{3}\right)^{-1/3}\\
\frac{3p-1}{c^{2}}+15c^{4}-12\Delta_{3}c^{7} & \text{ if }c\leq c_{0}
\end{cases}
\]
Let $\psi_{min}^{''}\left(c\right)$ be the function corresponding
to these bounds.
\end{lem}

\begin{proof}
The derivative of $\psi_{f}\left(c\right)$ is straightforward to
compute:
\begin{align}
\psi_{f}\left(c|\boldsymbol{e}\right) & =-\left(3p-1\right)\log\left(c\right)+\varphi_{\boldsymbol{e}}\left(c^{3}\right)\\
\psi_{f}^{'}\left(c|\boldsymbol{e}\right) & =-\frac{3p-1}{c}+3c^{2}\varphi_{\boldsymbol{e}}^{'}\left(c^{3}\right)\\
\psi_{f}^{''}\left(c|\boldsymbol{e}\right) & =\frac{3p-1}{c^{2}}+6c\varphi_{\boldsymbol{e}}^{'}\left(c^{3}\right)+9c^{4}\varphi_{\boldsymbol{e}}^{''}\left(c^{3}\right)
\end{align}
Combining this with the lower-bounds for $\varphi_{\boldsymbol{e}}^{'}\left(r\right)$
and $\varphi_{\boldsymbol{e}}^{''}\left(r\right)$ from Lemma \ref{lem:Basic-properties-of varphi_e(r)},
we obtain:
\begin{align}
\psi_{f}^{''}\left(c|\boldsymbol{e}\right) & \geq\frac{3p-1}{c^{2}}+6c\left[\varphi_{\text{min}}^{'}\left(c^{3}\right)\right]+9c^{4}\left[\varphi_{\text{min}}^{''}\left(c^{3}\right)\right]\\
 & \geq\begin{cases}
\frac{3p-1}{c^{2}}+6c\left[\frac{1}{2}\frac{1}{\Delta_{3}}\right]+9c^{4}\left[0\right] & \text{ if }c\geq c_{0}=\left(\Delta_{3}\right)^{-1/3}\\
\frac{3p-1}{c^{2}}+6c\left[c^{3}-\frac{\Delta_{3}}{2}c^{6}\right]+9c^{4}\left[1-\Delta_{3}c^{3}\right] & \text{ if }c\leq c_{0}
\end{cases}\\
 & \geq\begin{cases}
\frac{3p-1}{c^{2}}+\frac{3c}{\Delta_{3}} & \text{ if }c\geq c_{0}=\left(\Delta_{3}\right)^{-1/3}\\
\frac{3p-1}{c^{2}}+6c^{4}-3\Delta_{3}c^{7}+9c^{4}-9\Delta_{3}c^{7} & \text{ if }c\leq c_{0}
\end{cases}\\
 & \geq\begin{cases}
\frac{3p-1}{c^{2}}+\frac{3c}{\Delta_{3}} & \text{ if }c\geq c_{0}=\left(\Delta_{3}\right)^{-1/3}\\
\frac{3p-1}{c^{2}}+15c^{4}-12\Delta_{3}c^{7} & \text{ if }c\leq c_{0}
\end{cases}
\end{align}
which concludes the proof.
\end{proof}
\begin{lem}
Detailed lower-bound of $\psi_{f}^{''}\left(c|\boldsymbol{e}\right)$
for $c\geq c_{0}$.\label{lem: APPENDIX detailed lower bound for log-curvature c geq c_0}

\textbf{Requires: }Lemma \ref{lem: f(c|e) is SLC}.

On the $c\geq c_{0}$ region, we have:
\begin{equation}
\min_{c\geq c_{0}}\left[\psi_{f}^{''}\left(c|\boldsymbol{e}\right)\right]\geq\begin{cases}
\left(3p-1\right)\left(\Delta_{3}\right)^{2/3}+3\left(\Delta_{3}\right)^{-4/3} & \text{ if }\Delta_{3}\leq\sqrt{\frac{3}{2}\frac{1}{3p-1}}\\
\frac{9}{2}\frac{\left(\frac{2}{3}3p-1\right)^{1/3}}{\left(\Delta_{3}\right)^{2/3}} & \text{ else }
\end{cases}
\end{equation}
\end{lem}

\begin{proof}
We start from the lower-bound of $\psi_{f}^{''}\left(c|\boldsymbol{e}\right)$
of Lemma \ref{lem: APPENDIX detailed lower bound for log-curvature c geq c_0}:
\[
\psi_{f}^{''}\left(c|\boldsymbol{e}\right)\geq\begin{cases}
\frac{3p-1}{c^{2}}+\frac{3c}{\Delta_{3}} & \text{ if }c\geq c_{0}=\left(\Delta_{3}\right)^{-1/3}\\
\frac{3p-1}{c^{2}}+15c^{4}-12\Delta_{3}c^{7} & \text{ if }c\leq c_{0}
\end{cases}
\]

We consider only the $c\geq c_{0}$ half:
\begin{equation}
\psi_{min}^{''}\left(c\right)=\frac{3p-1}{c^{2}}+\frac{3c}{\Delta_{3}}
\end{equation}
To find the minimum, we compute the derivative of the lower-bound:
\begin{align}
\psi_{min}^{\left(3\right)}\left(c\right) & =-2\frac{3p-1}{c^{3}}+\frac{3}{\Delta_{3}}
\end{align}
which we recognize as a strictly increasing function. Its unique root
is:
\begin{equation}
c^{\star}=\left(2\Delta_{3}\frac{3p-1}{3}\right)^{1/3}
\end{equation}
If $c^{\star}$ is smaller than $c_{0}$, then the minimal curvature
for $c\geq c_{0}$ is reached at $c_{0}$. Otherwise, the minimum
is reached at $c^{\star}$.

The condition for $c^{\star}\leq c_{0}$ corresponds to:
\begin{align}
c^{\star} & \leq c_{0}\\
\left(2\Delta_{3}\frac{3p-1}{3}\right)^{1/3} & \leq\left(\frac{1}{\Delta_{3}}\right)^{1/3}\\
\left(\Delta_{3}\right)^{2} & \leq\frac{3}{2}\frac{1}{3p-1}\\
\Delta_{3} & \leq\sqrt{\frac{3}{2}\frac{1}{3p-1}}
\end{align}

Finally, we compute the curvature at $c_{0}$ and $c^{\star}$:
\begin{align}
\psi_{min}^{''}\left(c_{0}\right) & =\frac{\left(3p-1\right)}{\left(\frac{1}{\Delta_{3}}\right)^{2/3}}+\frac{3}{\Delta_{3}}\left(\frac{1}{\Delta_{3}}\right)^{1/3}\\
 & =\left(3p-1\right)\left(\Delta_{3}\right)^{2/3}+3\left(\Delta_{3}\right)^{-4/3}\label{eq: curvature at c_0}
\end{align}
and:
\begin{align}
\psi_{min}^{''}\left(c^{\star}\right) & =\frac{3p-1}{\left(2\Delta_{3}\frac{3p-1}{3}\right)^{2/3}}+\frac{3\left(2\Delta_{3}\frac{3p-1}{3}\right)^{1/3}}{\Delta_{3}}\\
 & =\frac{\left(3p-1\right)^{1/3}}{\left(\frac{2}{3}\right)^{2/3}\left(\Delta_{3}\right)^{2/3}}+\frac{3\left(\frac{2}{3}\right)^{1/3}\left(3p-1\right)^{1/3}}{\left(\Delta_{3}\right)^{2/3}}\\
 & =\frac{\left(3p-1\right)^{1/3}}{\left(\Delta_{3}\right)^{2/3}}\left(\frac{1}{\left(\frac{2}{3}\right)^{2/3}}+3\left(\frac{2}{3}\right)^{1/3}\right)\\
 & =\frac{\left(3p-1\right)^{1/3}}{\left(\Delta_{3}\right)^{2/3}}\left(\frac{2}{3}\right)^{1/3}\left(\frac{3}{2}+3\right)\\
 & =\frac{9}{2}\frac{\left(\frac{2}{3}3p-1\right)^{1/3}}{\left(\Delta_{3}\right)^{2/3}}\label{eq: minimum curvature in the tail-region}
\end{align}
We thus obtain a lower bound on $\psi_{f}^{''}\left(c|\boldsymbol{e}\right)$
for the $c\geq c_{0}$ region by combining eqs.\eqref{eq: curvature at c_0}
and \eqref{eq: minimum curvature in the tail-region}:
\begin{align}
\min_{c\geq c_{0}}\left[\psi_{f}^{''}\left(c|\boldsymbol{e}\right)\right] & \geq\min_{c\geq c_{0}}\left[\psi_{min}^{''}\left(c\right)\right]\\
 & \geq\begin{cases}
\left(3p-1\right)\left(\Delta_{3}\right)^{2/3}+3\left(\Delta_{3}\right)^{-4/3} & \text{ if }\Delta_{3}\leq\sqrt{\frac{3}{2}\frac{1}{3p-1}}\\
\frac{9}{2}\frac{\left(\frac{2}{3}3p-1\right)^{1/3}}{\left(\Delta_{3}\right)^{2/3}} & \text{ else }
\end{cases}
\end{align}
which concludes the proof of this Lemma.
\end{proof}
\begin{lem}
Detailed lower-bound of $\psi_{f}^{''}\left(c|\boldsymbol{e}\right)$
for $c\leq c_{0}$.\label{lem: APPENDIX detailed lower bound for log-curvature c LEQ c_0}

\textbf{Requires: }Lemma \ref{lem: f(c|e) is SLC}.

On the $c\leq c_{0}$ region, we have:

\begin{equation}
\min_{c\leq c_{0}}\left[\psi_{f}^{''}\left(c|\boldsymbol{e}\right)\right]\geq\text{minimum}\begin{cases}
\left(3p-1\right)\left(\Delta_{3}\right)^{2/3}+3\left(\Delta_{3}\right)^{-4/3} & \text{Always}\\
\left(r_{\left(\infty\right)}\right)^{4/3}\left[15+\frac{3p-1}{\left(r_{\left(\infty\right)}\right)^{2}}\right]-12\Delta_{3}\left(r_{\left(\infty\right)}\right)^{7/3} & \text{If }\Delta_{3}\leq\sqrt{\frac{1000}{441}\frac{1}{3p-1}}
\end{cases}
\end{equation}
where $r_{\left(\infty\right)}$ is the limit of the growing and bounded
sequence:
\begin{align}
r_{\left(0\right)} & =\sqrt{\frac{1}{10}}\sqrt{\frac{3p-1}{3}}\\
r_{\left(n+1\right)} & =\sqrt{\frac{\left(3p-1\right)/3+14\Delta_{3}\left(r_{\left(n\right)}\right)^{3}}{10}}
\end{align}
\end{lem}

\begin{proof}
This proof proceeds similarly as for Lemma \ref{lem: APPENDIX detailed lower bound for log-curvature c geq c_0}.

We start from the lower-bound of $\psi_{f}^{''}\left(c|\boldsymbol{e}\right)$
of Lemma \ref{lem: APPENDIX detailed lower bound for log-curvature c geq c_0}:
\[
\psi_{f}^{''}\left(c|\boldsymbol{e}\right)\geq\begin{cases}
\frac{3p-1}{c^{2}}+\frac{3c}{\Delta_{3}} & \text{ if }c\geq c_{0}=\left(\Delta_{3}\right)^{-1/3}\\
\frac{3p-1}{c^{2}}+15c^{4}-12\Delta_{3}c^{7} & \text{ if }c\leq c_{0}
\end{cases}
\]

The derivative of the lower bound in the $c\leq c_{0}$ region is:
\begin{align}
\psi_{min}^{\left(3\right)}\left(c\right) & =\frac{\partial}{\partial c}\left(\frac{3p-1}{c^{2}}+15c^{4}-12\Delta_{3}c^{7}\right)\\
 & =-2\frac{3p-1}{c^{3}}+60c^{3}-84\Delta_{3}c^{6}
\end{align}
which has the same sign as a third degree polynomial in $r=c^{3}$
over the range $c>0$:
\begin{align}
\text{sign}\left(\psi_{min}^{\left(3\right)}\left(c\right)\right) & =\text{sign}\left(\frac{1}{c^{3}}\left[-84\Delta_{3}c^{9}+60c^{6}-2\left(3p-1\right)\right]\right)\\
 & =\text{sign}\left(-84\Delta_{3}r^{3}+60r^{2}-2\left(3p-1\right)\right)\\
 & =\text{sign}\left(-14\Delta_{3}r^{3}+10r^{2}-\frac{3p-1}{3}\right)
\end{align}
Let $P\left(r\right)$ denote this polynomial:
\[
P\left(r\right)=-14\Delta_{3}r^{3}+10r^{2}-\frac{3p-1}{3}
\]

Depending on the relative values of $\Delta_{3}$ and $p$, the polynomial
$P\left(r\right)$ qualitatively changes behavior. While it is always
negative at $r=0$, we further have that, depending on the value of
$\Delta_{3}$, it might or might not remain negative on the region
$r\geq0$. We can investigate the sign of $P\left(r\right)$ by computing
the sign of the maximum.

We find the position of the maximum by computing the derivative of
the polynomial:
\begin{equation}
P^{'}\left(r\right)=-42\Delta_{3}r^{2}+20r
\end{equation}
which has a unique strictly positive root at:
\begin{align}
r^{\star\star} & =\frac{20}{42}\frac{1}{\Delta_{3}}\\
 & =\frac{10}{21}\frac{1}{\Delta_{3}}\label{eq: the position of the maximum of the derivative of the lower-bound}
\end{align}
The maximum of the polynomial is thus:
\begin{align}
P\left(r^{\star\star}\right) & =-14\Delta_{3}\left(r^{\star\star}\right)^{3}+10\left(r^{\star\star}\right)^{2}-\frac{3p-1}{3}\\
 & =-14\Delta_{3}\left(\frac{10}{21}\frac{1}{\Delta_{3}}\right)^{3}+10\left(\frac{10}{21}\frac{1}{\Delta_{3}}\right)^{2}-\frac{3p-1}{3}\\
 & =\left(\frac{10}{21}\right)^{2}\left(\frac{1}{\Delta_{3}}\right)^{2}\left(-14\frac{10}{21}+10\right)-\frac{3p-1}{3}\\
 & =\left(\frac{1}{\Delta_{3}}\right)^{2}\frac{100}{441}\left(-\frac{2}{3}10+10\right)-\frac{3p-1}{3}\\
 & =\left(\frac{1}{\Delta_{3}}\right)^{2}\frac{100}{441}\frac{10}{3}-\frac{3p-1}{3}\\
 & =\left(\frac{1}{\Delta_{3}}\right)^{2}\frac{1000}{1323}-\frac{3p-1}{3}
\end{align}

Thus, for large values of $\Delta_{3}$, such that:
\begin{align}
\left(\frac{1}{\Delta_{3}}\right)^{2}\frac{1000}{1323}-\frac{3p-1}{3} & \leq0\\
\left(\frac{1}{\Delta_{3}}\right)^{2}\frac{1000}{1323} & \leq\frac{3p-1}{3}\\
\left(\Delta_{3}\right)^{2} & \geq\frac{1000}{1323}\frac{3}{3p-1}\\
 & \geq\frac{1000}{441}\frac{1}{3p-1}\\
\Delta_{3} & \geq\sqrt{\frac{1000}{441}\frac{1}{3p-1}}
\end{align}
the minimum of $\psi_{min}^{''}\left(c\right)$ for $c\leq c_{0}$
is found at the transition point $c_{0}=\left(\Delta_{3}\right)^{-1}$:
\begin{align}
\min_{c\leq c_{0}}\left[\psi_{f}^{''}\left(c|\boldsymbol{e}\right)\right] & \geq\psi_{min}^{''}\left(c_{0}\right)\text{ if }\Delta_{3}\geq\sqrt{\frac{1000}{441}\frac{1}{3p-1}}\\
 & \geq\left(3p-1\right)\left(\Delta_{3}\right)^{2/3}+3\left(\Delta_{3}\right)^{-4/3}
\end{align}

For values of $\Delta_{3}$ smaller than this critical value, the
polynomial instead changes signs twice, since it is negative at $r=0$,
has a strictly positive maximum and a strictly negative asymptote.
The root closest to $r=0$ corresponds to a local minimum while the
furthest one is a local maximum.

Let us compute the root closest to $0$. This could be solved exactly
but is poor in intuition. Instead, consider the following recursive
algorithm which starts from the root when $\Delta_{3}=0$ and which
converges to the first root:
\begin{align}
r_{\left(0\right)} & =\sqrt{\frac{1}{10}}\sqrt{\frac{3p-1}{3}}\\
r_{\left(n+1\right)} & =\sqrt{\frac{\left(3p-1\right)/3+14\Delta_{3}\left(r_{\left(n\right)}\right)^{3}}{10}}
\end{align}

The properties of the sequence $r_{n}$ follow from the fact that
the function:
\begin{equation}
r\rightarrow\sqrt{\frac{\left(3p-1\right)/3+14\Delta_{3}r^{3}}{10}}
\end{equation}
is strictly growing and strictly positive and intersects the identity
at the two roots of the polynomial. A straightforward recursion establishes
that:
\begin{itemize}
\item The sequence $r_{\left(n\right)}$ is strictly growing.
\item The sequence $r_{\left(n\right)}$ is upper-bounded by the first root
of the polynomial, $r^{\star}$.
\end{itemize}
The sequence thus converges to $r^{\star}$.

We can furthermore upper-bound the root $r^{\star}$ using the position
of the maximum of the polynomial (we must encounter the first root
of $P$ before its maximum due to monotonicity):
\begin{equation}
r^{\star}\leq r^{\star\star}=\frac{10}{21}\frac{1}{\Delta_{3}}
\end{equation}
from which we obtain that $r^{\star}<r_{0}=\left(\Delta_{3}\right)^{-1}$.

In order to bound $\psi_{min}^{''}\left(c\right)$ on the region $c\leq c_{0}$,
we need to split this region into two around the second root of $P$:
$r^{\star\star\star}$:
\begin{itemize}
\item For smaller values of $c$: $c\leq\left(r^{\star\star\star}\right)^{1/3}$,
$\psi_{min}^{''}\left(c\right)$ is bounded by the local minimum found
at $\left(r^{\star}\right)^{1/3}=\left(r_{\left(\infty\right)}\right)^{1/3}$.
\item For larger values of $c$: $c\geq\left(r^{\star\star\star}\right)^{1/3}$,
$\psi_{min}^{''}\left(c\right)$ is strictly decreasing until $c_{0}$.
$\psi_{min}^{''}\left(c\right)$ is thus bounded by $\psi_{min}^{''}\left(c_{0}\right)$.
\end{itemize}
To summarize, for values of $\Delta_{3}$ that are smaller than the
critical value:
\begin{equation}
\Delta_{3}\leq\sqrt{\frac{1000}{441}\frac{1}{3p-1}}
\end{equation}
we can lower-bound $\psi_{min}^{''}\left(c\right)$ on the range $c\leq c_{0}$
using the minimum of the two values. We thus have to find the minimum
of the two local minima:
\begin{align}
\min_{c\leq\left(r^{\star\star\star}\right)^{1/3}}\psi_{min}^{''}\left(c\right) & \geq\psi_{min}^{''}\left(\left(r_{\left(\infty\right)}\right)^{1/3}\right)\\
 & \geq\frac{3p-1}{\left(r_{\left(\infty\right)}\right)^{2/3}}+15\left(r_{\left(\infty\right)}\right)^{4/3}-12\Delta_{3}\left(r_{\left(\infty\right)}\right)^{7/3}\\
 & \geq\left(r_{\left(\infty\right)}\right)^{4/3}\left[15+\frac{3p-1}{\left(r_{\left(\infty\right)}\right)^{2}}\right]-12\Delta_{3}\left(r_{\left(\infty\right)}\right)^{7/3}
\end{align}
and:
\begin{align}
\min_{c\geq\left(r^{\star\star\star}\right)^{1/3}}\psi_{min}^{''}\left(c\right) & \geq\psi_{min}^{''}\left(c_{0}\right)\\
 & \geq\left(3p-1\right)\left(\Delta_{3}\right)^{2/3}+3\left(\Delta_{3}\right)^{-4/3}
\end{align}
yielding:
\begin{equation}
\min_{c\leq c_{0}}\left[\psi_{f}^{''}\left(c|\boldsymbol{e}\right)\right]\geq\text{minimum}\begin{cases}
\left(3p-1\right)\left(\Delta_{3}\right)^{2/3}+3\left(\Delta_{3}\right)^{-4/3} & \text{Always}\\
\left(r_{\left(\infty\right)}\right)^{4/3}\left[15+\frac{3p-1}{\left(r_{\left(\infty\right)}\right)^{2}}\right]-12\Delta_{3}\left(r_{\left(\infty\right)}\right)^{7/3} & \text{If }\Delta_{3}\leq\sqrt{\frac{1000}{441}\frac{1}{3p-1}}
\end{cases}
\end{equation}
and concluding the proof.
\end{proof}
\begin{lem}
Approximate behavior of $r_{\left(\infty\right)}$ in the $\Delta_{3}\rightarrow0$
limit.\label{lem: APPENDIX approximation of r_n}

\textbf{Requires:} Lemma \ref{lem: APPENDIX detailed lower bound for log-curvature c LEQ c_0}.

In the limit $\Delta_{3}p\rightarrow0$, $r_{\left(\infty\right)}$
scales as:
\begin{align}
r_{\left(\infty\right)} & =r_{\left(0\right)}+\mathcal{O}\left(\Delta_{3}p\right)\\
 & =\sqrt{\frac{1}{10}}\sqrt{\frac{3p-1}{3}}+\mathcal{O}\left(\Delta_{3}p\right)
\end{align}
\end{lem}

\begin{proof}
To prove this result, consider the normalized sequence:
\begin{align}
t_{\left(0\right)}=\frac{r_{\left(0\right)}}{\sqrt{3p-1}} & =\sqrt{\frac{1}{30}}\\
t_{\left(n+1\right)}=\frac{r_{\left(n+1\right)}}{\sqrt{3p-1}} & =\sqrt{\frac{1/3+14\Delta_{3}\sqrt{3p-1}\left(r_{\left(n\right)}/\sqrt{3p-1}\right)^{3}}{10}}\\
 & =\sqrt{\frac{1/3+14\Delta_{3}\sqrt{3p-1}\left(t_{\left(n\right)}\right)^{3}}{10}}
\end{align}

Let us compute an order $0$ approximation of $t_{\left(1\right)}$:
\begin{align}
t_{\left(1\right)} & =\sqrt{\frac{1/3+14\Delta_{3}\sqrt{3p-1}\left(t_{\left(0\right)}\right)^{3}}{10}}\\
 & =\sqrt{\frac{1/3+14\Delta_{3}\sqrt{3p-1}\left(30\right)^{-3/2}}{10}}\\
 & =\sqrt{\frac{1/3+\mathcal{O}\left(\Delta_{3}p^{1/2}\right)}{10}}\\
 & =\sqrt{\frac{1/3}{10}}+\mathcal{O}\left(\Delta_{3}p^{1/2}\right)\\
 & =\sqrt{\frac{1}{30}}+\mathcal{O}\left(\Delta_{3}p^{1/2}\right)\\
 & =t_{\left(0\right)}+\mathcal{O}\left(\Delta_{3}p^{1/2}\right)
\end{align}

A straightforward recursion shows that the same is true for all members
of the sequence:
\begin{equation}
t_{\left(n\right)}=t_{\left(0\right)}+\mathcal{O}\left(\Delta_{3}p^{1/2}\right)
\end{equation}
and thus also true for the limit:
\begin{equation}
t_{\left(\infty\right)}=t_{\left(0\right)}+\mathcal{O}\left(\Delta_{3}p^{1/2}\right)
\end{equation}

We then return to $r_{\left(\infty\right)}$:
\begin{align}
r_{\left(\infty\right)} & =\sqrt{3p-1}t_{\left(\infty\right)}\\
 & =\sqrt{3p-1}\left[t_{\left(0\right)}+\mathcal{O}\left(\Delta_{3}p^{1/2}\right)\right]\\
 & =r_{\left(0\right)}+\mathcal{O}\left(\Delta_{3}p\right)\\
 & =\sqrt{\frac{1}{10}}\sqrt{\frac{3p-1}{3}}+\mathcal{O}\left(\Delta_{3}p\right)
\end{align}
which concludes the proof.
\end{proof}
\begin{lem}
Approximate behavior of $\psi_{min}^{''}\left(c\right)$ in the $\Delta_{3}\rightarrow0$
limit.\label{lem: APPENDIX approximation of psi_min^sec}

\textbf{Requires:} Lemmas \ref{lem: APPENDIX detailed lower bound for log-curvature c geq c_0},
\ref{lem: APPENDIX detailed lower bound for log-curvature c LEQ c_0}
and \ref{lem: APPENDIX approximation of r_n}.

In the limit $\Delta_{3}p^{7/6}\rightarrow0$, the minimum of $\psi_{min}^{''}\left(c\right)$
scales as:
\begin{equation}
\min_{c\geq0}\left[\psi_{min}^{''}\left(c\right)\right]=45\left(\frac{1}{10}\right)^{2/3}\left(\frac{3p-1}{3}\right)^{2/3}+\mathcal{O}\left(\Delta_{3}p^{7/6}\right)
\end{equation}

NB: in the Gaussian case, we have:
\begin{equation}
\min_{c\geq0}\psi_{g}^{''}\left(c\right)=45\left(\frac{1}{10}\right)^{2/3}\left(\frac{3p-1}{3}\right)^{2/3}
\end{equation}
\end{lem}

\begin{proof}
First, recall the bounds of $\psi_{min}^{''}\left(c\right)$ derived
in Lemmas \ref{lem: APPENDIX detailed lower bound for log-curvature c geq c_0}
and \ref{lem: APPENDIX detailed lower bound for log-curvature c LEQ c_0}:
\begin{equation}
\min_{c\geq c_{0}}\left[\psi_{f}^{''}\left(c|\boldsymbol{e}\right)\right]\geq\begin{cases}
\left(3p-1\right)\left(\Delta_{3}\right)^{2/3}+3\left(\Delta_{3}\right)^{-4/3} & \text{ if }\Delta_{3}\leq\sqrt{\frac{3}{2}\frac{1}{3p-1}}\\
\frac{9}{2}\frac{\left(\frac{2}{3}3p-1\right)^{1/3}}{\left(\Delta_{3}\right)^{2/3}} & \text{ else }
\end{cases}
\end{equation}
and:
\begin{equation}
\min_{c\leq c_{0}}\left[\psi_{f}^{''}\left(c|\boldsymbol{e}\right)\right]\geq\text{minimum}\begin{cases}
\left(3p-1\right)\left(\Delta_{3}\right)^{2/3}+3\left(\Delta_{3}\right)^{-4/3} & \text{Always}\\
\left(r_{\left(\infty\right)}\right)^{4/3}\left[15+\frac{3p-1}{\left(r_{\left(\infty\right)}\right)^{2}}\right]-12\Delta_{3}\left(r_{\left(\infty\right)}\right)^{7/3} & \text{If }\Delta_{3}\leq\sqrt{\frac{1000}{441}\frac{1}{3p-1}}
\end{cases}
\end{equation}

First, let us determine asymptotically which bounds hold. Both conditions
require $\Delta_{3}\sqrt{3p-1}$ small. In the limit $\Delta_{3}p^{7/6}\rightarrow0$,
both conditions thus hold.

Thus, asymptotically, we have the overall bound for all values of
$c$:
\begin{equation}
\min_{c\geq0}\left[\psi_{f}^{''}\left(c|\boldsymbol{e}\right)\right]\geq\text{minimum}\begin{cases}
\left(3p-1\right)\left(\Delta_{3}\right)^{2/3}+3\left(\Delta_{3}\right)^{-4/3} & \text{Always}\\
\left(r_{\left(\infty\right)}\right)^{4/3}\left[15+\frac{3p-1}{\left(r_{\left(\infty\right)}\right)^{2}}\right]-12\Delta_{3}\left(r_{\left(\infty\right)}\right)^{7/3} & \text{If }\Delta_{3}\leq\sqrt{\frac{1000}{441}\frac{1}{3p-1}}
\end{cases}
\end{equation}
$\psi_{f}^{''}\left(c\right)$ will be bounded below by the smallest
of the two bounds.

Let us compute the asymptotic scaling of both bounds.

First, consider the $\psi_{min}^{''}\left(c_{0}\right)$ term:
\begin{align}
\psi_{min}^{''}\left(c_{0}\right) & =\left(3p-1\right)\left(\Delta_{3}\right)^{2/3}+3\left(\Delta_{3}\right)^{-4/3}\\
 & =\left(\Delta_{3}\right)^{-4/3}\left[3+\left(3p-1\right)\Delta_{3}\right]\\
 & =\left(\Delta_{3}\right)^{-4/3}\left[3+\mathcal{O}\left(\Delta_{3}p\right)\right]
\end{align}
In the limit $\Delta_{3}\rightarrow0$, this term diverges due to
the presence of the $\left(\Delta_{3}\right)^{-4/3}$ term. We will
conclude the proof by showing that this term is negligible compared
to the other term in the limit.

Now consider the $\psi_{min}^{''}\left(r_{\left(\infty\right)}\right)$
term:
\begin{equation}
\psi_{min}^{''}\left(r_{\left(\infty\right)}\right)=\left(r_{\left(\infty\right)}\right)^{4/3}\left[15+\frac{3p-1}{\left(r_{\left(\infty\right)}\right)^{2}}\right]-12\Delta_{3}\left(r_{\left(\infty\right)}\right)^{7/3}
\end{equation}
where we have that (Lemma \ref{lem: APPENDIX approximation of r_n}):
\begin{equation}
r_{\left(\infty\right)}=\sqrt{\frac{1}{10}}\sqrt{\frac{3p-1}{3}}+\mathcal{O}\left(\Delta_{3}p\right)
\end{equation}

Let us tackle each term of $\psi_{min}^{''}\left(r_{\left(\infty\right)}\right)$
in order.

First:
\begin{align}
\left(r_{\left(\infty\right)}\right)^{4/3} & =\left[\sqrt{\frac{1}{10}}\sqrt{\frac{3p-1}{3}}+\mathcal{O}\left(\Delta_{3}p\right)\right]^{4/3}\\
 & =\left(\frac{1}{10}\right)^{2/3}\left(\frac{3p-1}{3}\right)^{2/3}\left[1+\mathcal{O}\left(\Delta_{3}p^{1/2}\right)\right]^{4/3}\\
 & =\left(\frac{1}{10}\right)^{2/3}\left(\frac{3p-1}{3}\right)^{2/3}\left[1+\mathcal{O}\left(\Delta_{3}p^{1/2}\right)\right]\\
 & =\left(\frac{1}{10}\right)^{2/3}\left(\frac{3p-1}{3}\right)^{2/3}+\mathcal{O}\left(\Delta_{3}p^{1/2}p^{2/3}\right)\\
 & =\left(\frac{1}{10}\right)^{2/3}\left(\frac{3p-1}{3}\right)^{2/3}+\mathcal{O}\left(\Delta_{3}p^{7/6}\right)
\end{align}

Then (re-using the term $t_{\left(\infty\right)}$ from Lemma \ref{lem: APPENDIX approximation of r_n}):
\begin{align}
\left[15+\frac{3p-1}{\left(r_{\left(\infty\right)}\right)^{2}}\right] & =\left[15+\frac{1}{\left(t_{\left(\infty\right)}\right)^{2}}\right]\\
 & =\left[15+\frac{1}{\left[\sqrt{\frac{1}{30}}+\mathcal{O}\left(\Delta_{3}p^{1/2}\right)\right]^{2}}\right]\\
 & =\left[15+\frac{1}{1/30+\mathcal{O}\left(\Delta_{3}p^{1/2}\right)}\right]\\
 & =\left[15+30\frac{1}{1+\mathcal{O}\left(\Delta_{3}p^{1/2}\right)}\right]\\
 & =45+\mathcal{O}\left(\Delta_{3}p^{1/2}\right)
\end{align}

Finally:
\begin{align}
12\Delta_{3}\left(r_{\left(\infty\right)}\right)^{7/3} & =12\Delta_{3}\left(\sqrt{\frac{1}{10}}\sqrt{\frac{3p-1}{3}}+\mathcal{O}\left(\Delta_{3}p\right)\right)^{7/3}\\
 & =12\Delta_{3}\left(\frac{1}{10}\right)^{7/6}\left(\frac{3p-1}{3}\right)^{7/6}\left[1+\mathcal{O}\left(\Delta_{3}p^{1/2}\right)\right]^{7/3}\\
 & =\mathcal{O}\left(\Delta_{3}p^{7/6}\right)\left[1+\mathcal{O}\left(\left(\Delta_{3}\right)^{7/3}p^{7/6}\right)\right]\\
 & =\mathcal{O}\left(\Delta_{3}p^{7/6}\right)
\end{align}

We thus finally have the scaling of $\psi_{min}^{''}\left(r_{\left(\infty\right)}\right)$
by combining the three parts:
\begin{align}
\psi_{min}^{''}\left(r_{\left(\infty\right)}\right) & =\left[\left(\frac{1}{10}\right)^{2/3}\left(\frac{3p-1}{3}\right)^{2/3}+\mathcal{O}\left(\Delta_{3}p^{7/6}\right)\right]\left[45+\mathcal{O}\left(\Delta_{3}p^{1/2}\right)\right]+\mathcal{O}\left(\Delta_{3}p^{7/6}\right)\\
 & =45\left(\frac{1}{10}\right)^{2/3}\left(\frac{3p-1}{3}\right)^{2/3}+\mathcal{O}\left(\Delta_{3}p^{7/6}\right)+\mathcal{O}\left(\Delta_{3}p^{1/2}p^{2/3}\right)+\mathcal{O}\left(\Delta_{3}p^{7/6}\right)\\
 & =45\left(\frac{1}{10}\right)^{2/3}\left(\frac{3p-1}{3}\right)^{2/3}+\mathcal{O}\left(\Delta_{3}p^{7/6}\right)+\mathcal{O}\left(\Delta_{3}p^{7/6}\right)\\
 & =45\left(\frac{1}{10}\right)^{2/3}\left(\frac{3p-1}{3}\right)^{2/3}+\mathcal{O}\left(\Delta_{3}p^{7/6}\right)
\end{align}
Note that the main term diverges if $p$ is growing.

In order to conclude the proof, we finally compare the two bounds:
\begin{align}
\frac{\psi_{min}^{''}\left(r_{\left(\infty\right)}\right)}{\psi_{min}^{''}\left(c_{0}\right)} & =\frac{45\left(\frac{1}{10}\right)^{2/3}\left(\frac{3p-1}{3}\right)^{2/3}+\mathcal{O}\left(\Delta_{3}p^{7/6}\right)}{\left(\Delta_{3}\right)^{-4/3}\left[3+\mathcal{O}\left(\Delta_{3}p\right)\right]}\\
 & =\frac{45\left(\frac{1}{10}\right)^{2/3}\left(\Delta_{3}\right)^{4/3}\left(\frac{3p-1}{3}\right)^{2/3}+\mathcal{O}\left(\left(\Delta_{3}\right)^{7/3}p^{7/6}\right)}{3+\mathcal{O}\left(\Delta_{3}p\right)}\\
 & =\frac{\mathcal{O}\left(\left(\Delta_{3}\right)^{4/3}p^{2/3}\right)+\mathcal{O}\left(\left(\Delta_{3}\right)^{7/3}p^{7/6}\right)}{3+\mathcal{O}\left(\Delta_{3}p\right)}\\
 & =\frac{\mathcal{O}\left(\left(\Delta_{3}p^{1/2}\right)^{4/3}\right)+\mathcal{O}\left(\left(\Delta_{3}p^{1/2}\right)^{7/3}\right)}{3+\mathcal{O}\left(\Delta_{3}p\right)}\\
 & =\frac{\mathcal{O}\left(\left(\Delta_{3}p\right)^{4/3}\right)+\mathcal{O}\left(\left(\Delta_{3}p\right)^{7/3}\right)}{3+\mathcal{O}\left(\Delta_{3}p\right)}\\
 & =\mathcal{O}\left(\Delta_{3}p\right)
\end{align}
Thus, the $\psi_{min}^{''}\left(r_{\left(\infty\right)}\right)$ term
is asymptotically the smallest.

We thus finally have the asymptotic scaling of the minimum of $\psi_{min}^{''}\left(c\right)$:
\begin{align}
\psi_{min}^{''}\left(c\right) & =\psi_{min}^{''}\left(r_{\left(\infty\right)}\right)\\
 & =45\left(\frac{1}{10}\right)^{2/3}\left(\frac{3p-1}{3}\right)^{2/3}+\mathcal{O}\left(\Delta_{3}p^{7/6}\right)
\end{align}
which concludes our proof.
\end{proof}
\begin{prop}
Bound on $KL\left(r_{g},r_{f}|\boldsymbol{e}\right)$.\label{prop: APPENDIX Bound-on- KL r_g r_f | e}

\textbf{Requires: }Lemmas \ref{lem: f(c|e) is SLC}, \ref{lem: APPENDIX approximation of psi_min^sec}
(and \ref{lem: APPENDIX Moments-of-r_g}).

The KL divergence between $r_{g}$ and $r_{f}|\boldsymbol{e}$ is
bounded:
\begin{align}
KL\left(r_{g},r_{f}|\boldsymbol{e}\right) & =KL\left(c_{g},c_{f}|\boldsymbol{e}\right)\\
 & \leq\frac{1}{2}\left\{ \min_{c}\left[\psi_{f}^{''}\left(c|\boldsymbol{e}\right)\right]\right\} ^{-1}\E_{r\sim g\left(r\right)}\left[9r^{4/3}\left(\varphi_{\boldsymbol{e}}^{'}\left(r\right)-r\right)^{2}\right]
\end{align}

In a limit $\Delta_{3}p^{7/6}\rightarrow0$, then the following approximation
holds:
\begin{align}
KL\left(r_{g},r_{f}|\boldsymbol{e}\right) & \leq\frac{1}{2p^{2/3}}\E_{r\sim g\left(r\right)}\left[r^{4/3}\left(\varphi_{\boldsymbol{e}}^{'}\left(r\right)-r\right)^{2}\right]+\mathcal{O}\left(\left(\Delta_{3}\right)^{3}p^{5/2}\right)\\
 & \leq\mathcal{O}\left(\left(\Delta_{3}\right)^{2}p^{2}\right)
\end{align}
\end{prop}

\begin{proof}
The first part of the claim follows from applying the LSI to the Strongly
Log-Concave density $f\left(c|\boldsymbol{e}\right)$:
\begin{itemize}
\item The minimum curvature is strictly positive (Lemma \ref{lem: f(c|e) is SLC}):
\begin{align}
\min_{c}\left[\psi_{f}^{''}\left(c|\boldsymbol{e}\right)\right] & \geq\begin{cases}
\frac{3p-1}{c^{2}}+\frac{3c}{\Delta_{3}} & \text{ if }c\geq c_{0}=\left(\Delta_{3}\right)^{-1/3}\\
\frac{3p-1}{c^{2}}+15c^{4}-12\Delta_{3}c^{7} & \text{ if }c\leq c_{0}
\end{cases}\\
 & >0
\end{align}
\item The gradient of $\psi_{f}\left(c|\boldsymbol{e}\right)$ is:
\begin{align}
\psi_{f}\left(c|\boldsymbol{e}\right) & =-\left(3p-1\right)\log\left(c\right)+\varphi_{\boldsymbol{e}}\left(c^{3}\right)\\
\psi_{f}^{'}\left(c|\boldsymbol{e}\right) & =-\frac{3p-1}{c}+3c^{2}\varphi_{\boldsymbol{e}}^{'}\left(c^{3}\right)
\end{align}
\item The gradient of $\psi_{g}\left(c\right)$ is:
\begin{equation}
\psi_{g}^{'}\left(c\right)=-\frac{3p-1}{c}+3c^{5}
\end{equation}
\end{itemize}
The LSI then reads:
\begin{align}
KL\left(c_{g},c_{f}|\boldsymbol{e}\right) & \leq\frac{1}{2}\left\{ \min_{c}\left[\psi_{f}^{''}\left(c|\boldsymbol{e}\right)\right]\right\} ^{-1}\E_{c\sim g\left(c\right)}\left[\left(\psi_{f}^{'}\left(c\right)-\psi_{g}^{'}\left(c\right)\right)^{2}\right]\\
 & \leq\frac{1}{2}\left\{ \min_{c}\left[\psi_{f}^{''}\left(c|\boldsymbol{e}\right)\right]\right\} ^{-1}\E_{c\sim g\left(c\right)}\left[\left(3c^{2}\varphi_{\boldsymbol{e}}^{'}\left(c^{3}\right)-3c^{5}\right)^{2}\right]\\
 & \leq\frac{1}{2}\left\{ \min_{c}\left[\psi_{f}^{''}\left(c|\boldsymbol{e}\right)\right]\right\} ^{-1}\E_{c\sim g\left(c\right)}\left[9c^{4}\left(\varphi_{\boldsymbol{e}}^{'}\left(c^{3}\right)-c^{3}\right)^{2}\right]\\
 & \leq\frac{1}{2}\left\{ \min_{c}\left[\psi_{f}^{''}\left(c|\boldsymbol{e}\right)\right]\right\} ^{-1}\E_{r\sim g\left(r\right)}\left[9r^{4/3}\left(\varphi_{\boldsymbol{e}}^{'}\left(r\right)-r\right)^{2}\right]
\end{align}
which gives us the exact form of the lower-bound by recalling that
the KL divergence is invariant to bijective changes of variable:
\begin{align}
KL\left(r_{g},r_{f}|\boldsymbol{e}\right) & =KL\left(c_{g},c_{f}|\boldsymbol{e}\right)\\
 & \leq\frac{1}{2}\left\{ \min_{c}\left[\psi_{f}^{''}\left(c|\boldsymbol{e}\right)\right]\right\} ^{-1}\E_{r\sim g\left(r\right)}\left[9r^{4/3}\left(\varphi_{\boldsymbol{e}}^{'}\left(r\right)-r\right)^{2}\right]
\end{align}

We can give an asymptotic expansion of the expected value under $g\left(r\right)$
by performing a Taylor expansion of $\varphi_{\boldsymbol{e}}^{'}\left(r\right)$
around $0$:
\begin{align}
\left|\varphi_{\boldsymbol{e}}^{'}\left(r\right)-r\right| & \leq\frac{1}{2}\Delta_{3}r^{2}\\
\E_{r\sim g\left(r\right)}\left[r^{4/3}\left(\varphi_{\boldsymbol{e}}^{'}\left(r\right)-r\right)^{2}\right] & \leq\E_{r\sim g\left(r\right)}\left[r^{4/3}\left(\frac{1}{2}\Delta_{3}r^{2}\right)^{2}\right]\\
 & \leq\frac{1}{4}\left(\Delta_{3}\right)^{2}\E_{r\sim g\left(r\right)}\left[r^{4/3}r^{4}\right]\\
 & \leq\frac{1}{4}\left(\Delta_{3}\right)^{2}\E_{r\sim g\left(r\right)}\left[r^{16/3}\right]
\end{align}
where we can obtain the order of $\E_{r\sim g\left(r\right)}\left[r^{16/3}\right]$
from Lemma \ref{lem: APPENDIX Moments-of-r_g}:
\begin{equation}
\E_{r\sim g\left(r\right)}\left[r^{4/3}\left(\varphi_{\boldsymbol{e}}^{'}\left(r\right)-r\right)^{2}\right]\leq\frac{1}{4}\left(\Delta_{3}\right)^{2}p^{8/3}+\mathcal{O}\left(\left(\Delta_{3}\right)^{2}p^{5/3}\right)
\end{equation}

We can then combine this asymptotic expansion with that of the lower-bound
of $\min_{c}\left[\psi_{f}^{''}\left(c|\boldsymbol{e}\right)\right]$
in Lemma \ref{lem: APPENDIX approximation of psi_min^sec}:
\begin{align}
KL\left(r_{g},r_{f}|\boldsymbol{e}\right) & \leq\frac{1}{2}\left\{ \min_{c}\left[\psi_{f}^{''}\left(c|\boldsymbol{e}\right)\right]\right\} ^{-1}\E_{r\sim g\left(r\right)}\left[9r^{4/3}\left(\varphi_{\boldsymbol{e}}^{'}\left(r\right)-r\right)^{2}\right]\\
 & \leq\frac{\frac{9}{8}\left(\Delta_{3}\right)^{2}p^{8/3}+\mathcal{O}\left(\left(\Delta_{3}\right)^{2}p^{5/3}\right)}{45\left(\frac{1}{10}\right)^{2/3}\left(\frac{3p-1}{3}\right)^{2/3}+\mathcal{O}\left(\Delta_{3}p^{7/6}\right)}\\
 & \leq\frac{\mathcal{O}\left(\left(\Delta_{3}\right)^{2}p^{2}\right)}{1+\mathcal{O}\left(\Delta_{3}p^{1/2}\right)}\\
 & \leq\mathcal{O}\left(\left(\Delta_{3}\right)^{2}p^{2}\right)
\end{align}

We can finally give the first order expansion of $KL\left(r_{g},r_{f}|\boldsymbol{e}\right)$:
\begin{align}
KL\left(r_{g},r_{f}|\boldsymbol{e}\right) & \leq\frac{1}{2}\left\{ \min_{c}\left[\psi_{f}^{''}\left(c|\boldsymbol{e}\right)\right]\right\} ^{-1}\E_{r\sim g\left(r\right)}\left[9r^{4/3}\left(\varphi_{\boldsymbol{e}}^{'}\left(r\right)-r\right)^{2}\right]\nonumber \\
 & \leq\frac{\frac{1}{2}9\E_{r\sim g\left(r\right)}\left[r^{4/3}\left(\varphi_{\boldsymbol{e}}^{'}\left(r\right)-r\right)^{2}\right]}{45\left(\frac{1}{10}\right)^{2/3}\left(\frac{3p-1}{3}\right)^{2/3}+\mathcal{O}\left(\Delta_{3}p^{7/6}\right)}\\
 & \leq\frac{9\E_{r\sim g\left(r\right)}\left[r^{4/3}\left(\varphi_{\boldsymbol{e}}^{'}\left(r\right)-r\right)^{2}\right]}{90\left(\frac{1}{10}\right)^{2/3}\left(\frac{3p-1}{3}\right)^{2/3}\left(1+\mathcal{O}\left(\Delta_{3}p^{1/2}\right)\right)}\\
 & \leq\frac{\E_{r\sim g\left(r\right)}\left[r^{4/3}\left(\varphi_{\boldsymbol{e}}^{'}\left(r\right)-r\right)^{2}\right]}{10\left(\frac{1}{10}\right)^{2/3}\left(\frac{3p-1}{3}\right)^{2/3}}\left(1+\mathcal{O}\left(\Delta_{3}p^{1/2}\right)\right)\\
 & \leq\frac{\E_{r\sim g\left(r\right)}\left[r^{4/3}\left(\varphi_{\boldsymbol{e}}^{'}\left(r\right)-r\right)^{2}\right]}{10\left(\frac{1}{10}\right)^{2/3}\left(\frac{3p-1}{3}\right)^{2/3}}+\mathcal{O}\left(\left(\Delta_{3}\right)^{2}p^{2}\right)\mathcal{O}\left(\Delta_{3}p^{1/2}\right)\\
 & \leq\frac{\E_{r\sim g\left(r\right)}\left[r^{4/3}\left(\varphi_{\boldsymbol{e}}^{'}\left(r\right)-r\right)^{2}\right]}{10\left(\frac{1}{10}\right)^{2/3}\left(\frac{3p-1}{3}\right)^{2/3}}+\mathcal{O}\left(\left(\Delta_{3}\right)^{3}p^{5/2}\right)
\end{align}
Since we are aiming for an approximate upper-bound, we can slightly
simplify some of the terms in the denominator:
\begin{align}
\frac{1}{10\left(\frac{1}{10}\right)^{2/3}\left(\frac{3p-1}{3}\right)^{2/3}} & =\frac{1}{2.15\dots\left(\frac{3p-1}{3}\right)^{2/3}}\\
 & \leq\frac{1}{2}\frac{1}{\left(\frac{3p-1}{3}\right)^{2/3}}\\
 & \leq\frac{1}{2p^{2/3}}
\end{align}
Thus yielding finally:
\begin{equation}
KL\left(r_{g},r_{f}|\boldsymbol{e}\right)\leq\frac{1}{2p^{2/3}}\E_{r\sim g\left(r\right)}\left[r^{4/3}\left(\varphi_{\boldsymbol{e}}^{'}\left(r\right)-r\right)^{2}\right]+\mathcal{O}\left(\left(\Delta_{3}\right)^{3}p^{5/2}\right)
\end{equation}
which concludes the proof.
\end{proof}

\subsection{Proofs: $KL\left(\boldsymbol{e}_{g},\boldsymbol{e}_{f}\right)$ term.}
\begin{lem}
Asymptotic quality of the $ELBO$ approximation of $\log\left[f\left(\boldsymbol{e}\right)\right]$.\label{lem: APPENDIX approx quality of ELBO}

\textbf{Requires: }Lemma \ref{lem: APPENDIX distributions under f},
Prop.\ref{prop: APPENDIX Bound-on- KL r_g r_f | e}.

In the limit $\Delta_{3}p^{7/6}\rightarrow0$, the error of the ELBO
approximation of $\log\left[f\left(\boldsymbol{e}\right)\right]$
is upper-bounded:
\begin{equation}
\log\left[f\left(\boldsymbol{e}\right)\right]=ELBO\left(\boldsymbol{e}\right)+\mathcal{O}\left(\left(\Delta_{3}\right)^{2}p^{2}\right)
\end{equation}
and in the limit $\Delta_{3}p^{3/2}\rightarrow0$ the range of values
taken by the ELBO is also upper-bounded:
\begin{equation}
\max_{\boldsymbol{e}}\left[ELBO\left(\boldsymbol{e}\right)\right]-\min_{\boldsymbol{e}}\left[ELBO\left(\boldsymbol{e}\right)\right]=\mathcal{O}\left(\Delta_{3}p^{3/2}\right)
\end{equation}
\end{lem}

\begin{proof}
Lemma \ref{lem: APPENDIX distributions under f} defines the ELBO
approximation of $\log\left[f\left(\boldsymbol{e}\right)\right]$
as:
\begin{equation}
\log\left[f\left(\boldsymbol{e}\right)\right]=-\E_{r\sim g\left(r\right)}\left[\tilde{\phi}_{f}\left(r\boldsymbol{e}\right)-\frac{1}{2}r^{2}\right]+C
\end{equation}
and asserts that error at $\boldsymbol{e}$ is precisely equal to
the KL divergence $KL\left(r_{g},r_{f}|\boldsymbol{e}\right)$. Since
Prop.\ref{prop: APPENDIX Bound-on- KL r_g r_f | e} bounds this KL
divergence, it can then be applied to bound the error:
\begin{align}
\log\left[f\left(\boldsymbol{e}\right)\right] & =ELBO\left(\boldsymbol{e}\right)+KL\left(r_{g},r_{f}|\boldsymbol{e}\right)\\
 & =ELBO\left(\boldsymbol{e}\right)+\mathcal{O}\left(\left(\Delta_{3}\right)^{2}p^{2}\right)
\end{align}

We can furthermore bound the maximum size of the range of values taken
by $ELBO\left(\boldsymbol{e}\right)$ by performing a Taylor expansion
of $\tilde{\phi}_{f}\left(r\boldsymbol{e}\right)-\frac{1}{2}r^{2}$
at $0$:
\begin{equation}
\left|\tilde{\phi}_{f}\left(r\boldsymbol{e}\right)-\frac{1}{2}r^{2}\right|\leq\frac{1}{6}\Delta_{3}r^{3}
\end{equation}
Thus:
\begin{align}
ELBO\left(\boldsymbol{e}\right) & =-\E_{r\sim g\left(r\right)}\left[\tilde{\phi}_{f}\left(r\boldsymbol{e}\right)-\frac{1}{2}r^{2}\right]+C\\
ELBO\left(\boldsymbol{e}\right) & \geq-\E_{r\sim g\left(r\right)}\left[\frac{1}{6}\Delta_{3}r^{3}\right]+C\\
 & \geq\mathcal{O}\left(\Delta_{3}p^{3/2}\right)+C\\
ELBO\left(\boldsymbol{e}\right) & \leq\E_{r\sim g\left(r\right)}\left[\frac{1}{6}\Delta_{3}r^{3}\right]+C\\
 & \leq\mathcal{O}\left(\Delta_{3}p^{3/2}\right)+C
\end{align}
Thus yielding:
\begin{equation}
\max_{\boldsymbol{e}}\left[ELBO\left(\boldsymbol{e}\right)\right]-\min_{\boldsymbol{e}}\left[ELBO\left(\boldsymbol{e}\right)\right]=\mathcal{O}\left(\Delta_{3}p^{3/2}\right)
\end{equation}
and concluding the proof.
\end{proof}
\begin{prop}
$ELBO$-based approximation of $KL\left(\boldsymbol{e}_{g},\boldsymbol{e}_{f}\right)$.\label{Prop: APPENDIX ELBO based approx of KL(e_g,e_f)}

\textbf{Requires: }Lemma \ref{lem: APPENDIX approx quality of ELBO},
Appendix Prop.\ref{prop: KL approx KL_var} (Main text Prop.\ref{prop: KL approx KL_var}).

In the limit $\Delta_{3}p^{3/2}\rightarrow0$, the KL divergence $KL\left(\boldsymbol{e}_{g},\boldsymbol{e}_{f}\right)$
scales as:
\begin{equation}
KL\left(\boldsymbol{e}_{g},\boldsymbol{e}_{f}\right)=\mathcal{O}\left(\left(\Delta_{3}\right)^{2}p^{3}\right)
\end{equation}
The dominating term being:
\begin{equation}
KL\left(\boldsymbol{e}_{g},\boldsymbol{e}_{f}\right)=\frac{1}{2}\text{Var}_{\boldsymbol{e}\sim g\left(\boldsymbol{e}\right)}\left\{ \E_{r\sim g\left(r\right)}\left[\tilde{\phi}_{f}\left(r\boldsymbol{e}\right)-\frac{1}{2}r^{2}\right]\right\} +\mathcal{O}\left(\left(\Delta_{3}\right)^{3}p^{9/2}\right)
\end{equation}
\end{prop}

\begin{proof}
This proof follows from using the KL variance approximation of the
KL divergence (Appendix Prop.\ref{prop: KL approx KL_var} / Main
text Prop.\ref{prop: KL approx KL_var}). Lemma \ref{lem: APPENDIX approx quality of ELBO}
establishes that $\log\left[f\left(\boldsymbol{e}\right)\right]$
has bounded range since:
\begin{equation}
\log\left[f\left(\boldsymbol{e}\right)\right]=ELBO\left(\boldsymbol{e}\right)+\mathcal{O}\left(\left(\Delta_{3}\right)^{2}p^{2}\right)
\end{equation}
and $ELBO\left(\boldsymbol{e}\right)$ has bounded range. Thus:
\begin{align}
\max_{\boldsymbol{e}}\left[\log\left[f\left(\boldsymbol{e}\right)\right]\right]-\min_{\boldsymbol{e}}\left[\log\left[f\left(\boldsymbol{e}\right)\right]\right] & =\max_{\boldsymbol{e}}\left[ELBO\left(\boldsymbol{e}\right)\right]-\min_{\boldsymbol{e}}\left[ELBO\left(\boldsymbol{e}\right)\right]+\mathcal{O}\left(\left(\Delta_{3}\right)^{2}p^{2}\right)\\
 & \leq\mathcal{O}\left(\Delta_{3}p^{3/2}\right)+\mathcal{O}\left(\left(\Delta_{3}\right)^{2}p^{2}\right)\\
 & \leq\mathcal{O}\left(\Delta_{3}p^{3/2}\right)
\end{align}

Furthermore, $\log\left[g\left(\boldsymbol{e}\right)\right]$ is constant.
Thus $\log\left[\frac{g\left(\boldsymbol{e}\right)}{f\left(\boldsymbol{e}\right)}\right]$
also has bounded range:
\begin{align}
\max_{\boldsymbol{e}}\log\left[\frac{g\left(\boldsymbol{e}\right)}{f\left(\boldsymbol{e}\right)}\right]-\min_{\boldsymbol{e}}\log\left[\frac{g\left(\boldsymbol{e}\right)}{f\left(\boldsymbol{e}\right)}\right] & =\max_{\boldsymbol{e}}\left[\log\left[f\left(\boldsymbol{e}\right)\right]\right]-\min_{\boldsymbol{e}}\left[\log\left[f\left(\boldsymbol{e}\right)\right]\right]\\
 & \leq\mathcal{O}\left(\Delta_{3}p^{3/2}\right)
\end{align}
We can thus apply Main text Prop.\ref{prop: KL approx KL_var}, yielding:
\begin{align}
KL\left(\boldsymbol{e}_{g},\boldsymbol{e}_{f}\right) & =\frac{1}{2}KL_{var}\left(\boldsymbol{e}_{g},\boldsymbol{e}_{f}\right)+\left[\frac{1}{6}\mathcal{O}\left(\Delta_{3}p^{3/2}\right)\right]^{3}\\
 & =\frac{1}{2}KL_{var}\left(\boldsymbol{e}_{g},\boldsymbol{e}_{f}\right)+\mathcal{O}\left(\left(\Delta_{3}\right)^{3}p^{9/2}\right)\label{eq: intermediate eq for KL(e_g, e_f)}
\end{align}

Now, we expand the KL variance:
\begin{align}
KL_{var}\left(\boldsymbol{e}_{g},\boldsymbol{e}_{f}\right) & =\text{Var}_{\boldsymbol{e}\sim g\left(\boldsymbol{e}\right)}\left\{ \log\left[\frac{g\left(\boldsymbol{e}\right)}{f\left(\boldsymbol{e}\right)}\right]\right\} \\
 & =\text{Var}_{\boldsymbol{e}\sim g\left(\boldsymbol{e}\right)}\left\{ \log\left[f\left(\boldsymbol{e}\right)\right]\right\} \\
 & =\text{Var}_{\boldsymbol{e}\sim g\left(\boldsymbol{e}\right)}\left\{ ELBO\left(\boldsymbol{e}\right)+\mathcal{O}\left(\left(\Delta_{3}\right)^{2}p^{2}\right)\right\} \\
 & =\text{Var}_{\boldsymbol{e}\sim g\left(\boldsymbol{e}\right)}\left\{ -\E_{r\sim g\left(r\right)}\left[\tilde{\phi}_{f}\left(r\boldsymbol{e}\right)-\frac{1}{2}r^{2}\right]+\mathcal{O}\left(\left(\Delta_{3}\right)^{2}p^{2}\right)\right\} 
\end{align}
which is the variance of the sum of two terms that are both small.
We can decompose this variance into:
\begin{align}
\text{Var}\left(X+Y\right) & =\text{Var}\left(X\right)+\text{Cov}\left(X,Y\right)+\text{Var}\left(Y\right)
\end{align}
where:
\begin{itemize}
\item The first variance is the term we seek to recover:
\begin{align}
\text{Var}_{\boldsymbol{e}\sim g\left(\boldsymbol{e}\right)}\left\{ -\E_{r\sim g\left(r\right)}\left[\tilde{\phi}_{f}\left(r\boldsymbol{e}\right)-\frac{1}{2}r^{2}\right]\right\}  & =\text{Var}_{\boldsymbol{e}\sim g\left(\boldsymbol{e}\right)}\left\{ \E_{r\sim g\left(r\right)}\left[\tilde{\phi}_{f}\left(r\boldsymbol{e}\right)-\frac{1}{2}r^{2}\right]\right\} \\
 & \leq\text{Var}_{\boldsymbol{e}\sim g\left(\boldsymbol{e}\right)}\left\{ \frac{1}{6}\Delta_{3}\E\left(\left(r_{g}\right)^{3}\right)\right\} \\
 & \leq\text{Var}_{\boldsymbol{e}\sim g\left(\boldsymbol{e}\right)}\left\{ \mathcal{O}\left(\Delta_{3}p^{3/2}\right)\right\} \\
 & \leq\mathcal{O}\left(\left(\Delta_{3}\right)^{2}p^{3}\right)
\end{align}
\item The other variance term is small:
\begin{align}
\text{Var}_{\boldsymbol{e}\sim g\left(\boldsymbol{e}\right)}\left\{ \mathcal{O}\left(\left(\Delta_{3}\right)^{2}p^{2}\right)\right\}  & \leq\E_{\boldsymbol{e}\sim g\left(\boldsymbol{e}\right)}\left\{ \mathcal{O}\left(\left(\Delta_{3}\right)^{2}p^{2}\right)\right\} ^{2}\\
 & \leq\mathcal{O}\left(\left(\Delta_{3}\right)^{4}p^{4}\right)
\end{align}
\item And the covariance term is also small because of the Cauchy-Schwarz
inequality:
\begin{align}
\text{Cov}_{\boldsymbol{e}\sim g\left(\boldsymbol{e}\right)}\left\{ -\E_{r\sim g\left(r\right)}\left[\tilde{\phi}_{f}\left(r\boldsymbol{e}\right)-\frac{1}{2}r^{2}\right],\mathcal{O}\left(\left(\Delta_{3}\right)^{2}p^{2}\right)\right\}  & \leq\sqrt{\mathcal{O}\left(\left(\Delta_{3}\right)^{2}p^{3}\right)\mathcal{O}\left(\left(\Delta_{3}\right)^{4}p^{4}\right)}\\
 & \leq\mathcal{O}\left(\left(\Delta_{3}\right)^{3}p^{9/2}\right)
\end{align}
\end{itemize}
We thus have finally:
\begin{align}
KL_{var}\left(\boldsymbol{e}_{g},\boldsymbol{e}_{f}\right) & =\mathcal{O}\left(\left(\Delta_{3}\right)^{2}p^{3}\right)\\
 & =\text{Var}_{\boldsymbol{e}\sim g\left(\boldsymbol{e}\right)}\left\{ \E_{r\sim g\left(r\right)}\left[\tilde{\phi}_{f}\left(r\boldsymbol{e}\right)-\frac{1}{2}r^{2}\right]\right\} +\mathcal{O}\left(\left(\Delta_{3}\right)^{3}p^{9/2}\right)
\end{align}
which we can insert into eq.\eqref{eq: intermediate eq for KL(e_g, e_f)}
yielding:
\begin{align}
KL\left(\boldsymbol{e}_{g},\boldsymbol{e}_{f}\right) & =\frac{1}{2}KL_{var}\left(\boldsymbol{e}_{g},\boldsymbol{e}_{f}\right)+\mathcal{O}\left(\left(\Delta_{3}\right)^{3}p^{9/2}\right)\\
 & =\frac{1}{2}\text{Var}_{\boldsymbol{e}\sim g\left(\boldsymbol{e}\right)}\left\{ \E_{r\sim g\left(r\right)}\left[\tilde{\phi}_{f}\left(r\boldsymbol{e}\right)-\frac{1}{2}r^{2}\right]\right\} +\mathcal{O}\left(\left(\Delta_{3}\right)^{3}p^{9/2}\right)\\
 & =\mathcal{O}\left(\left(\Delta_{3}\right)^{2}p^{3}\right)
\end{align}
and concluding the proof.
\end{proof}

\subsection{Proofs: relationship with the KL variance}
\begin{lem}
Relationship between $\text{Var}_{\boldsymbol{e}\sim g\left(\boldsymbol{e}\right)}\left\{ ELBO\left(\boldsymbol{e}\right)\right\} $
and $KL_{var}\left(g,f\right)$.\label{lem:Relationship-between-Var(ELBO) and KL-var}

\textbf{Requires: }Lemma \ref{lem: APPENDIX Decomposition-of-KL(g,f)},
Appendix Prop.\ref{prop: KL approx KL_var} (Main text Prop.\ref{prop: KL approx KL_var}).

The KL variance upper-bounds the Variance of the ELBO approximation
of $\log\left[f\left(\boldsymbol{e}\right)\right]$:

\begin{equation}
\text{Var}_{\boldsymbol{e}\sim g\left(\boldsymbol{e}\right)}\left\{ \E_{r\sim g\left(r\right)}\left[\tilde{\phi}_{f}\left(r\boldsymbol{e}\right)-\frac{1}{2}r^{2}\right]\right\} \leq KL_{var}\left(g,f\right)
\end{equation}

The difference between the two expressions is equal to a KL-variance-based
approximation of $KL\left(r_{g},r_{f}\right)$.
\end{lem}

\begin{proof}
The KL variance is defined as:
\begin{align}
KL_{var}\left(g,f\right) & =\text{Var}_{\boldsymbol{\theta}\sim g\left(\boldsymbol{\theta}\right)}\left[\log\frac{g\left(\boldsymbol{\theta}\right)}{f\left(\boldsymbol{\theta}\right)}\right]\\
 & =\text{Var}_{\tilde{\boldsymbol{\theta}}\sim g\left(\tilde{\boldsymbol{\theta}}\right)}\left[\log\frac{\tilde{g}\left(\tilde{\boldsymbol{\theta}}\right)}{\tilde{f}\left(\tilde{\boldsymbol{\theta}}\right)}\right]
\end{align}

Now perform a change of variable to work instead with the pair $r,\boldsymbol{e}$.
We can then decompose the variance using the conditional variance
under $\boldsymbol{e}$ and the conditional expected value under $\boldsymbol{e}$:
\begin{equation}
KL_{var}\left(g,f\right)=\text{Var}_{\boldsymbol{e}\sim g\left(\boldsymbol{e}\right)}\left[\E_{r\sim g\left(r\right)}\left[\log\frac{g\left(r,\boldsymbol{e}\right)}{f\left(r,\boldsymbol{e}\right)}\right]\right]+\E_{\boldsymbol{e}\sim g\left(\boldsymbol{e}\right)}\left[\text{Var}_{r\sim g\left(r\right)}\left[\log\frac{g\left(r,\boldsymbol{e}\right)}{f\left(r,\boldsymbol{e}\right)}\right]\right]
\end{equation}
Since we are computing variances, we can drop out the constant terms,
yielding:
\begin{align}
KL_{var}\left(g,f\right) & =\text{Var}_{\boldsymbol{e}\sim g\left(\boldsymbol{e}\right)}\left[\E_{r\sim g\left(r\right)}\left[\tilde{\phi}_{f}\left(r\boldsymbol{e}\right)-\tilde{\phi}_{g}\left(r\boldsymbol{e}\right)\right]\right]+\E_{\boldsymbol{e}\sim g\left(\boldsymbol{e}\right)}\left[\text{Var}_{r\sim g\left(r\right)}\left[\tilde{\phi}_{f}\left(r\boldsymbol{e}\right)-\tilde{\phi}_{g}\left(r\boldsymbol{e}\right)\right]\right]\\
 & =\text{Var}_{\boldsymbol{e}\sim g\left(\boldsymbol{e}\right)}\left[\E_{r\sim g\left(r\right)}\left[\tilde{\phi}_{f}\left(r\boldsymbol{e}\right)-\frac{1}{2}r^{2}\right]\right]+\E_{\boldsymbol{e}\sim g\left(\boldsymbol{e}\right)}\left[\text{Var}_{r\sim g\left(r\right)}\left[\tilde{\phi}_{f}\left(r\boldsymbol{e}\right)-\frac{1}{2}r^{2}\right]\right]
\end{align}
where we recognize the expression for the ELBO approximation of $\log\left[f\left(\boldsymbol{e}\right)\right]$:
$\E_{r\sim g\left(r\right)}\left[\tilde{\phi}_{f}\left(r\boldsymbol{e}\right)-\frac{1}{2}r^{2}\right]$.
Since the second term, an expected value of variances, is necessarily
positive, we thus finally have:
\begin{equation}
\text{Var}_{\boldsymbol{e}\sim g\left(\boldsymbol{e}\right)}\left\{ \E_{r\sim g\left(r\right)}\left[\tilde{\phi}_{f}\left(r\boldsymbol{e}\right)-\frac{1}{2}r^{2}\right]\right\} \leq KL_{var}\left(g,f\right)
\end{equation}
\end{proof}
Now, let us consider the other term in $KL_{var}\left(g,f\right)$.
It corresponds to a KL-variance approximation of the second term in
the decomposition of $KL\left(g,f\right)$ given in Lemma \ref{lem: APPENDIX Decomposition-of-KL(g,f)}:
\begin{align}
\E_{\boldsymbol{e}\sim g\left(\boldsymbol{e}\right)}\left[KL\left(r_{g},r_{f}|\boldsymbol{e}\right)\right] & \approx\E_{\boldsymbol{e}\sim g\left(\boldsymbol{e}\right)}\left[\frac{1}{2}KL_{var}\left(r_{g},r_{f}|\boldsymbol{e}\right)\right]\\
KL_{var}\left(r_{g},r_{f}|\boldsymbol{e}\right) & =\text{Var}_{r\sim g\left(r\right)}\left[\log\frac{g\left(r\right)}{f\left(r|\boldsymbol{e}\right)}\right]\\
 & =\text{Var}_{r\sim g\left(r\right)}\left[\tilde{\phi}_{f}\left(r\boldsymbol{e}\right)-\frac{1}{2}r^{2}\right]\\
\E_{\boldsymbol{e}\sim g\left(\boldsymbol{e}\right)}\left[KL\left(r_{g},r_{f}|\boldsymbol{e}\right)\right] & \approx\frac{1}{2}\E_{\boldsymbol{e}\sim g\left(\boldsymbol{e}\right)}\left[\text{Var}_{r\sim g\left(r\right)}\left[\tilde{\phi}_{f}\left(r\boldsymbol{e}\right)-\frac{1}{2}r^{2}\right]\right]
\end{align}
This observation concludes the proof.

\subsection{Proof of Theorem \ref{thm:A-deterministic-BvM theorem}}

We can finally now prove Theorem \ref{thm:A-deterministic-BvM theorem}.
\begin{proof}
Theorem \ref{thm:A-deterministic-BvM theorem} simply corresponds
to an amalgamation of the results of the various Lemmas and propositions
presented in this section:
\end{proof}
\begin{itemize}
\item The decomposition of $KL\left(g,f\right)$ was derived in Lemma \ref{lem: APPENDIX Decomposition-of-KL(g,f)}.
\item The term $KL\left(r_{g},r_{f}|\boldsymbol{e}\right)$ was bounded
and approximated in Proposition \ref{prop: APPENDIX Bound-on- KL r_g r_f | e}.
\item The term KL$\left(\boldsymbol{e}_{g},\boldsymbol{e}_{f}\right)$ was
bounded and approximated in Proposition \ref{Prop: APPENDIX ELBO based approx of KL(e_g,e_f)}.
\item The term $KL\left(\boldsymbol{e}_{g},\boldsymbol{e}_{f}\right)$ was
related to the KL variance in Lemma \ref{lem:Relationship-between-Var(ELBO) and KL-var}.
\end{itemize}
The final claim of the Theorem: giving the asymptotic order of $KL\left(g,f\right)$,
simply follows from combining the asymptotic order of $KL\left(r_{g},r_{f}|\boldsymbol{e}\right)$
and that of $KL\left(\boldsymbol{e}_{g},\boldsymbol{e}_{f}\right)$:
\begin{align}
KL\left(g,f\right) & =KL\left(\boldsymbol{e}_{g},\boldsymbol{e}_{f}\right)+\E_{\boldsymbol{e}\sim g\left(\boldsymbol{e}\right)}\left[KL\left(r_{g},r_{f}|\boldsymbol{e}\right)\right]\\
 & =\mathcal{O}\left(\left(\Delta_{3}\right)^{2}p^{3}\right)+\E_{\boldsymbol{e}\sim g\left(\boldsymbol{e}\right)}\left[\mathcal{O}\left(\left(\Delta_{3}\right)^{2}p^{2}\right)\right]\\
 & =\mathcal{O}\left(\left(\Delta_{3}\right)^{2}p^{3}\right)+\mathcal{O}\left(\left(\Delta_{3}\right)^{2}p^{2}\right)\\
 & =\mathcal{O}\left(\left(\Delta_{3}\right)^{2}p^{3}\right)
\end{align}

\section{Recovering the classical IID result}

\label{sec: APPENDIX Recovering-the-classical IID result}

This section deals with the proof of Cor.\ref{cor:Bernstein-von-Mises:-IID case},
detailing how Theorem \ref{thm:A-deterministic-BvM theorem} can be
used to recover the classical BvM result in the IID setting.

Throughout this section, we will assume that $f_{n}\left(\boldsymbol{\theta}\right)$
is a sequence of random posteriors, each one constructed from $n$
IID negative log-likelihood functions:
\begin{align}
f_{n}\left(\boldsymbol{\theta}\right) & \propto f_{0}\left(\boldsymbol{\theta}\right)\exp\left(-\sum_{i=1}^{n}NLL_{i}\left(\boldsymbol{\theta}\right)\right)\\
 & \propto\exp\left(-\phi_{0}\left(\boldsymbol{\theta}\right)-\sum_{i=1}^{n}NLL_{i}\left(\boldsymbol{\theta}\right)\right)\\
 & \propto\exp\left(-\phi_{n}\left(\boldsymbol{\theta}\right)\right)
\end{align}
We will further denote by $\hat{\boldsymbol{\theta}}_{n}$ the sequence
of MAP estimators:
\begin{equation}
\hat{\boldsymbol{\theta}}_{n}=\text{argmin}_{\boldsymbol{\theta}}\left\{ \phi_{0}\left(\boldsymbol{\theta}\right)+\sum_{i=1}^{n}NLL_{i}\left(n\right)\right\} 
\end{equation}
and we will denote by $g_{n}\left(\boldsymbol{\theta}\right)$ the
sequence of the Laplace approximations of the $f_{n}$ and $\Sigma_{n}$
their variances.

We further assume that the $NLL_{i}$ are all convex, that 
\begin{equation}
\boldsymbol{\theta}_{0}=\text{argmin}_{\boldsymbol{\theta}}\left\{ \E_{NLL}\left(NLL\left(\boldsymbol{\theta}\right)\right)\right\} 
\end{equation}
exists and is unique, that the Hessian-based Fisher matrix matrix:
\begin{equation}
J=\E\left[H\!NLL\left(\boldsymbol{\theta}_{0}\right)\right]
\end{equation}
is strictly positive, that the gradient-based Fisher matrix:
\begin{equation}
I=\E\left[\nabla\!NLL\left(\boldsymbol{\theta}_{0}\right)\!\nabla^{T}\!NLL\left(\boldsymbol{\theta}_{0}\right)\right]
\end{equation}
is strictly positive and finally that the scalar random variable:
\begin{equation}
\Delta^{\left(i\right)}=\max_{\boldsymbol{\theta},v_{1},v_{2},v_{3}}\frac{\left|NLL_{i}^{\left(3\right)}\left(\boldsymbol{\theta}\right)\left[J^{1/2}v_{1},J^{1/2}v_{2},J^{1/2}v_{3}\right]\right|}{\left\Vert v_{1}\right\Vert _{2}\left\Vert v_{2}\right\Vert _{2}\left\Vert v_{3}\right\Vert _{2}}
\end{equation}
has bounded expectation.

We finally assume that the prior is such that $\phi_{0}$ is convex
and has controlled third derivative:
\begin{equation}
\Delta^{\left(0\right)}=\max_{\boldsymbol{\theta},v_{1},v_{2},v_{3}}\frac{\left|\phi_{0}^{\left(3\right)}\left(\boldsymbol{\theta}\right)\left[J^{1/2}v_{1},J^{1/2}v_{2},J^{1/2}v_{3}\right]\right|}{\left\Vert v_{1}\right\Vert _{2}\left\Vert v_{2}\right\Vert _{2}\left\Vert v_{3}\right\Vert _{2}}
\end{equation}
Throughout this section, we will denote that a sequence of random
variables $X_{n}$ converges in distribution to a limit $X_{\infty}$
as: $X_{n}\rightarrow X_{\infty}$.

\subsection{Proof structure}

The proof of this result is straightforward but a little bit involved.
\begin{enumerate}
\item I first establish the asymptotic properties of the MAP estimator:
consistency and asymptotic normality. These are established after
the following sequence of intermediate results:
\begin{enumerate}
\item Asymptotics of $\nabla\phi_{n}\left(\boldsymbol{\theta}_{0}\right)$
and $H\phi_{n}\left(\boldsymbol{\theta}_{0}\right)$.
\item Asymptotics of the third derivative of $\phi_{n}\left(\boldsymbol{\theta}_{0}\right)$.
\end{enumerate}
The consistency and asymptotic normality then follow from standard
asymptotic arguments.
\item Then, I establish key properties of $\phi_{n}$ and the standardized
$\tilde{\phi}_{n}$:
\begin{enumerate}
\item Rate of growth of $\left[\Sigma_{n}\right]^{-1}$.
\item Finiteness and asymptotic evolution of $\Delta_{3}$.
\end{enumerate}
These establish that Theorem \ref{thm:A-deterministic-BvM theorem}
is applicable.
\item Finally, I apply Theorem \ref{thm:A-deterministic-BvM theorem} to
establish Cor.\ref{cor:Bernstein-von-Mises:-IID case}.
\end{enumerate}

\subsection{Proofs: asymptotic analysis of $\hat{\boldsymbol{\theta}}_{n}$}
\begin{lem}
Asymptotics of $\phi_{n}\left(\boldsymbol{\theta}_{0}\right)$.\label{lem:Asymptotics-of-phi_n (theta_0)}

\textbf{Requires:} NA.

At $\boldsymbol{\theta}_{0}$, we have the following two asymptotic
results:
\begin{align}
\frac{1}{\sqrt{n}}\nabla\phi_{n}\left(\boldsymbol{\theta}_{0}\right) & \rightarrow\mathcal{N}\left(0,I\right)\\
\frac{1}{n}H\phi_{n}\left(\boldsymbol{\theta}_{0}\right) & \rightarrow J
\end{align}
\end{lem}

\begin{proof}
This result is a straightforward application of the Central Limit
Theorem (CLT) and the strong law of large numbers (SLLN).

First, consider the gradient term:
\begin{equation}
\frac{1}{\sqrt{n}}\nabla\phi_{n}\left(\boldsymbol{\theta}_{0}\right)=\frac{1}{\sqrt{n}}\nabla\phi_{0}\left(\boldsymbol{\theta}_{0}\right)+\frac{1}{\sqrt{n}}\sum_{i=1}^{n}\nabla NLL_{i}\left(\boldsymbol{\theta}_{0}\right)
\end{equation}
The term $\frac{1}{\sqrt{n}}\nabla\phi_{0}\left(\boldsymbol{\theta}_{0}\right)$
is a deterministic term that converges to $0$ and plays no further
role. The CLT (\citet{van2000asymptotic}) asserts that, because the
$\nabla NLL_{i}\left(\boldsymbol{\theta}_{0}\right)$ are IID random
variables with mean $0$ and finite variance $I$:
\begin{equation}
\frac{1}{\sqrt{n}}\sum_{i=1}^{n}\nabla NLL_{i}\left(\boldsymbol{\theta}_{0}\right)\rightarrow\mathcal{N}\left(0,I\right)
\end{equation}
Combining the two yields (by Slutsky's theorem):
\begin{equation}
\frac{1}{\sqrt{n}}\nabla\phi_{n}\left(\boldsymbol{\theta}_{0}\right)\rightarrow\mathcal{N}\left(0,I\right)
\end{equation}

Similarly, for the Hessian term:
\begin{equation}
\frac{1}{n}H\phi_{n}\left(\boldsymbol{\theta}_{0}\right)=\frac{1}{n}H\phi_{0}\left(\boldsymbol{\theta}_{0}\right)+\frac{1}{n}\sum_{i=1}^{n}H\!NLL_{i}\left(\boldsymbol{\theta}_{0}\right)
\end{equation}
where $\frac{1}{n}H\phi_{0}\left(\boldsymbol{\theta}_{0}\right)\rightarrow0$
and $\frac{1}{n}\sum_{i=1}^{n}H\!NLL_{i}\left(\boldsymbol{\theta}_{0}\right)\rightarrow J$
by the SLLN. Slutsky's then yields:
\begin{equation}
\frac{1}{n}H\phi_{n}\left(\boldsymbol{\theta}_{0}\right)\rightarrow J
\end{equation}
and concludes the proof.
\end{proof}
\begin{lem}
Asymptotics of the third derivative of $\phi_{n}$.\label{lem: APPENDIX Asymptotics-of-the third derivative}

\textbf{Requires:} NA.

The third derivatives of $\phi_{n}$ are controlled since:
\begin{align}
\frac{1}{n}\Delta_{3}\left(\phi_{n}\right) & =\max_{\boldsymbol{\theta},v_{1},v_{2},v_{3}}\frac{\left|\phi_{n}^{\left(3\right)}\left(\boldsymbol{\theta}\right)\left[J^{1/2}v_{1},J^{1/2}v_{2},J^{1/2}v_{3}\right]\right|}{\left\Vert v_{1}\right\Vert _{2}\left\Vert v_{2}\right\Vert _{2}\left\Vert v_{3}\right\Vert _{2}}\\
 & \leq\E\left(\Delta^{\left(i\right)}\right)+o_{P}\left(1\right)
\end{align}
\end{lem}

\begin{proof}
Once more, this is a straightforward consequence of the SLLN.

Due to the sub-additivity of the maximum:
\begin{align}
\max_{\boldsymbol{\theta},v_{1},v_{2},v_{3}}\frac{\left|\phi_{n}^{\left(3\right)}\left(\boldsymbol{\theta}\right)\left[J^{1/2}v_{1},J^{1/2}v_{2},J^{1/2}v_{3}\right]\right|}{\left\Vert v_{1}\right\Vert _{2}\left\Vert v_{2}\right\Vert _{2}\left\Vert v_{3}\right\Vert _{2}} & \leq\frac{1}{n}\max_{\boldsymbol{\theta},v_{1},v_{2},v_{3}}\frac{\left|\phi_{0}^{\left(3\right)}\left(\boldsymbol{\theta}\right)\left[J^{1/2}v_{1},J^{1/2}v_{2},J^{1/2}v_{3}\right]\right|}{\left\Vert v_{1}\right\Vert _{2}\left\Vert v_{2}\right\Vert _{2}\left\Vert v_{3}\right\Vert _{2}}\nonumber \\
 & \ \ \ +\frac{1}{n}\sum_{i=1}^{n}\max_{\boldsymbol{\theta},v_{1},v_{2},v_{3}}\frac{\left|NLL_{i}^{\left(3\right)}\left(\boldsymbol{\theta}\right)\left[J^{1/2}v_{1},J^{1/2}v_{2},J^{1/2}v_{3}\right]\right|}{\left\Vert v_{1}\right\Vert _{2}\left\Vert v_{2}\right\Vert _{2}\left\Vert v_{3}\right\Vert _{2}}
\end{align}
where, once again, the $\phi_{0}$ term vanishes and the other term
is asymptotically $\E\left(\Delta^{\left(i\right)}\right)+o_{P}\left(1\right)$
due to the SLLN. Slutsky's then yields the claimed result, concluding
the proof.
\end{proof}
\begin{lem}
\textbf{$\hat{\boldsymbol{\theta}}_{n}$} is asymptotically consistent.\label{lem: APPENDIX MAP is consistent}

\textbf{Requires:} Lemmas \ref{lem:Asymptotics-of-phi_n (theta_0)}
and \ref{lem: APPENDIX Asymptotics-of-the third derivative}.

In the limit $n\rightarrow\infty$, $\hat{\boldsymbol{\theta}}_{n}$
converges to $\boldsymbol{\theta}_{0}$.
\end{lem}

\begin{proof}
This claim would be completely standard if not for the presence of
the prior term $\phi_{0}$. The following (involved) proof is almost
identical to that of Th. 5.42 of \citet{van2000asymptotic}.

We will prove that a shrinking neighborhood of $\boldsymbol{\theta}_{0}$
must necessarily contain a local minimum of $\phi_{n}$. Since $\phi_{n}$
is convex, this local minimum is the unique global minimum of $\phi_{n}$:
$\hat{\boldsymbol{\theta}}_{n}$. Thus, $\hat{\boldsymbol{\theta}}_{n}\rightarrow\boldsymbol{\theta}_{0}$.

First, consider the expected value of $NLL_{i}$:
\begin{equation}
\mathcal{E}\left(\boldsymbol{\theta}\right)=\E\left(NLL_{i}\left(\boldsymbol{\theta}\right)\right)
\end{equation}
Note that $\mathcal{E}\left(\boldsymbol{\theta}\right)$ inherits
regularity from the assumptions we made on $NLL_{i}$. $\mathcal{E}$
is strictly convex and its third derivative is controlled:
\begin{equation}
\max_{\boldsymbol{\theta},v_{1},v_{2},v_{3}}\frac{\left|\mathcal{E}_{n}^{\left(3\right)}\left(\boldsymbol{\theta}\right)\left[J^{1/2}v_{1},J^{1/2}v_{2},J^{1/2}v_{3}\right]\right|}{\left\Vert v_{1}\right\Vert _{2}\left\Vert v_{2}\right\Vert _{2}\left\Vert v_{3}\right\Vert _{2}}\leq\E\left(\Delta^{\left(i\right)}\right)
\end{equation}

We can thus perform a Taylor expansion of $\mathcal{E}\left(\boldsymbol{\theta}\right)$
around $\boldsymbol{\theta}_{0}$, yielding:
\begin{equation}
\mathcal{E}\left(\boldsymbol{\theta}\right)-\mathcal{E}\left(\boldsymbol{\theta}_{0}\right)=\frac{1}{2}\left(\boldsymbol{\theta}-\boldsymbol{\theta}_{0}\right)^{T}J\left(\boldsymbol{\theta}-\boldsymbol{\theta}_{0}\right)+\mathcal{O}\left(\left\Vert \boldsymbol{\theta}-\boldsymbol{\theta}_{0}\right\Vert ^{3}\right)
\end{equation}
Thus, if we consider the points $\boldsymbol{\theta}$ on the surface
of a sphere of radius $\delta$, we have that the difference $\mathcal{E}\left(\boldsymbol{\theta}\right)-\mathcal{E}\left(\boldsymbol{\theta}_{0}\right)$
is exactly of order $\delta^{2}$ (where we have used the strict positivity
of $J$):
\begin{equation}
\left\Vert \boldsymbol{\theta}-\boldsymbol{\theta}_{0}\right\Vert _{2}=\delta\ \ \implies\ \ \mathcal{E}\left(\boldsymbol{\theta}\right)-\mathcal{E}\left(\boldsymbol{\theta}_{0}\right)=\Theta\left(\delta^{2}\right)
\end{equation}
The minimum is thus also of order $\delta^{2}$:
\[
\min_{\left\Vert \boldsymbol{\theta}-\boldsymbol{\theta}_{0}\right\Vert _{2}=\delta}\mathcal{E}\left(\boldsymbol{\theta}\right)-\mathcal{E}\left(\boldsymbol{\theta}_{0}\right)=\Theta\left(\delta^{2}\right)
\]

Now, consider the difference between $\frac{1}{n}\phi_{n}\left(\boldsymbol{\theta}\right)$
and its limit $\mathcal{E}\left(\boldsymbol{\theta}\right)$. Performing
a Taylor expansion of the difference leads to:
\begin{align}
\frac{1}{n}\phi_{n}\left(\boldsymbol{\theta}\right)-\mathcal{E}\left(\boldsymbol{\theta}\right) & =\frac{1}{n}\nabla^{T}\phi_{n}\left(\boldsymbol{\theta}_{0}\right)\left(\boldsymbol{\theta}-\boldsymbol{\theta}_{0}\right)+\frac{1}{2}\left(\boldsymbol{\theta}-\boldsymbol{\theta}_{0}\right)^{T}\left[\frac{1}{n}H\phi_{n}\left(\boldsymbol{\theta}_{0}\right)-J\right]\left(\boldsymbol{\theta}-\boldsymbol{\theta}_{0}\right)\nonumber \\
 & \ \ \ \ +\mathcal{O}_{P}\left(\left\Vert \boldsymbol{\theta}-\boldsymbol{\theta}_{0}\right\Vert ^{3}\right)
\end{align}
where the $\mathcal{O}_{P}$ term contains both the term derived in
Lemma \ref{lem: APPENDIX Asymptotics-of-the third derivative} and
the deterministic term due to $\mathcal{E}\left(\boldsymbol{\theta}\right)$.

The SLLN yields that:
\begin{align}
\frac{1}{n}\nabla^{T}\phi_{n}\left(\boldsymbol{\theta}_{0}\right) & =o_{P}\left(1\right)\\
\frac{1}{n}H\phi_{n}\left(\boldsymbol{\theta}_{0}\right)-J & =o_{P}\left(1\right)
\end{align}
Thus, for all $\boldsymbol{\theta}$ in a ball of radius $\delta$
centered on $\boldsymbol{\theta}_{0}$, we have:
\begin{equation}
\max_{\left\Vert \boldsymbol{\theta}-\boldsymbol{\theta}_{0}\right\Vert _{2}\leq\delta}\frac{1}{n}\phi_{n}\left(\boldsymbol{\theta}\right)-\mathcal{E}\left(\boldsymbol{\theta}\right)=\delta o_{P}\left(1\right)+\delta^{2}o_{P}\left(1\right)+\delta^{3}\mathcal{O}_{p}\left(1\right)
\end{equation}
where the error term is uniform in $\delta$.

Consider finally the ratio of the distance between $\frac{1}{n}\phi_{n}\left(\boldsymbol{\theta}\right)$
and $\mathcal{E}\left(\boldsymbol{\theta}\right)$ inside the ball
and the minimum distance between $\mathcal{E}\left(\boldsymbol{\theta}\right)$
on the sphere and $\mathcal{E}\left(\boldsymbol{\theta}_{0}\right)$:
\begin{align}
R_{\delta} & =\frac{\max_{\left\Vert \boldsymbol{\theta}-\boldsymbol{\theta}_{0}\right\Vert _{2}\leq\delta}\left[\frac{1}{n}\phi_{n}\left(\boldsymbol{\theta}\right)-\mathcal{E}\left(\boldsymbol{\theta}\right)\right]}{\min_{\left\Vert \boldsymbol{\theta}-\boldsymbol{\theta}_{0}\right\Vert _{2}=\delta}\left[\mathcal{E}\left(\boldsymbol{\theta}\right)-\mathcal{E}\left(\boldsymbol{\theta}_{0}\right)\right]}\\
 & =\frac{\delta o_{P}\left(1\right)+\delta^{2}o_{P}\left(1\right)+\delta^{3}\mathcal{O}_{P}\left(1\right)}{\Theta\left(\delta^{2}\right)}\\
 & =\delta^{-1}o_{P}\left(1\right)+o_{P}\left(1\right)+\delta\mathcal{O}_{P}\left(1\right)
\end{align}

Now, for every $\delta$ and every $t>0$, we have:
\[
\mathbb{P}\left(\delta^{-1}o_{P}\left(1\right)+o_{P}\left(1\right)\geq t\delta\right)\rightarrow0
\]
We can thus find a sequence a sequence of $\delta_{n}$ tending to
$0$ sufficiently slowly so that:
\[
\delta_{n}^{-1}o_{P}\left(1\right)+o_{P}\left(1\right)=\delta_{n}o_{P}\left(1\right)
\]
This implies convergence to $0$ of the ratio $R_{\delta_{n}}$:
\begin{align}
R_{\delta_{n}} & =\delta_{n}^{-1}o_{P}\left(1\right)+o_{P}\left(1\right)+\delta_{n}\mathcal{O}_{P}\left(1\right)\\
 & =\delta_{n}o_{P}\left(1\right)+\delta_{n}\mathcal{O}_{P}\left(1\right)\\
 & =o_{P}\left(1\right)
\end{align}

Finally, consider the difference between the minimum value of $\frac{1}{n}\phi_{n}\left(\boldsymbol{\theta}\right)$
on the sphere of radius $\delta_{n}$ and $\frac{1}{n}\phi_{n}\left(\boldsymbol{\theta}_{0}\right)$:
\begin{align}
\min_{\left\Vert \boldsymbol{\theta}-\boldsymbol{\theta}_{0}\right\Vert _{2}=\delta_{n}}\frac{1}{n}\phi_{n}\left(\boldsymbol{\theta}\right)-\frac{1}{n}\phi_{n}\left(\boldsymbol{\theta}_{0}\right) & =\min_{\left\Vert \boldsymbol{\theta}-\boldsymbol{\theta}_{0}\right\Vert _{2}=\delta_{n}}\frac{1}{n}\phi_{n}\left(\boldsymbol{\theta}\right)-\mathcal{E}\left(\boldsymbol{\theta}\right)+\mathcal{E}\left(\boldsymbol{\theta}\right)-\mathcal{E}\left(\boldsymbol{\theta}_{0}\right)+\mathcal{E}\left(\boldsymbol{\theta}_{0}\right)-\frac{1}{n}\phi_{n}\left(\boldsymbol{\theta}_{0}\right)\\
 & =\min_{\left\Vert \boldsymbol{\theta}-\boldsymbol{\theta}_{0}\right\Vert _{2}=\delta_{n}}\frac{1}{n}\phi_{n}\left(\boldsymbol{\theta}\right)-\mathcal{E}\left(\boldsymbol{\theta}\right)+\mathcal{E}\left(\boldsymbol{\theta}\right)-\mathcal{E}\left(\boldsymbol{\theta}_{0}\right)+o_{P}\left(1\right)\\
 & =\min_{\left\Vert \boldsymbol{\theta}-\boldsymbol{\theta}_{0}\right\Vert _{2}=\delta_{n}}\mathcal{E}\left(\boldsymbol{\theta}\right)-\mathcal{E}\left(\boldsymbol{\theta}_{0}\right)\left[1+\frac{\frac{1}{n}\phi_{n}\left(\boldsymbol{\theta}\right)-\mathcal{E}\left(\boldsymbol{\theta}\right)}{\mathcal{E}\left(\boldsymbol{\theta}\right)-\mathcal{E}\left(\boldsymbol{\theta}_{0}\right)}\right]+o_{P}\left(1\right)
\end{align}
where the ratio is dominated by $R_{\delta_{n}}$ and thus converges
to $0$:
\begin{equation}
\left|\frac{\frac{1}{n}\phi_{n}\left(\boldsymbol{\theta}\right)-\mathcal{E}\left(\boldsymbol{\theta}\right)}{\mathcal{E}\left(\boldsymbol{\theta}\right)-\mathcal{E}\left(\boldsymbol{\theta}_{0}\right)}\right|\leq R_{\delta_{n}}=o_{P}\left(1\right)
\end{equation}

We thus have:
\begin{equation}
\min_{\left\Vert \boldsymbol{\theta}-\boldsymbol{\theta}_{0}\right\Vert _{2}=\delta_{n}}\frac{1}{n}\phi_{n}\left(\boldsymbol{\theta}\right)-\frac{1}{n}\phi_{n}\left(\boldsymbol{\theta}_{0}\right)=\min_{\left\Vert \boldsymbol{\theta}-\boldsymbol{\theta}_{0}\right\Vert _{2}=\delta_{n}}\mathcal{E}\left(\boldsymbol{\theta}\right)-\mathcal{E}\left(\boldsymbol{\theta}_{0}\right)+o_{P}\left(1\right)
\end{equation}
where the right hand side is strictly positive.

Thus, in the limit $n\rightarrow\infty$, with probability tending
to $1$, the following event occurs:
\begin{equation}
\frac{1}{n}\phi_{n}\left(\boldsymbol{\theta}_{0}\right)<\min_{\left\Vert \boldsymbol{\theta}-\boldsymbol{\theta}_{0}\right\Vert _{2}=\delta_{n}}\frac{1}{n}\phi_{n}\left(\boldsymbol{\theta}\right)
\end{equation}
The function $\phi_{n}\left(\boldsymbol{\theta}\right)$ thus has
a local minimum in the sphere of radius $\delta_{n}$ with probability
tending to $1$. Since $\phi_{n}$ is furthermore convex, this local
minimum is the global minimum $\hat{\boldsymbol{\theta}}_{n}$.

We have thus established:
\begin{align}
\left\Vert \hat{\boldsymbol{\theta}}_{n}-\boldsymbol{\theta}_{0}\right\Vert  & \leq\delta_{n}\ \text{with probability tending to 1}\\
 & =o_{P}\left(1\right)
\end{align}
which concludes the proof.
\end{proof}
\begin{prop}
\textbf{$\hat{\boldsymbol{\theta}}_{n}$} is asymptotically Gaussian.\label{prop: APPENDIX MAP is gaussian}

\textbf{Requires:} Lemmas \ref{lem:Asymptotics-of-phi_n (theta_0)},
\ref{lem: APPENDIX Asymptotics-of-the third derivative} and \ref{lem: APPENDIX MAP is consistent}.

In the limit $n\rightarrow\infty$, if $\hat{\boldsymbol{\theta}}_{n}$
is consistent then:
\begin{equation}
\sqrt{n}\left(\hat{\boldsymbol{\theta}}_{n}-\boldsymbol{\theta}_{0}\right)\rightarrow\mathcal{N}\left(0,J^{-1}IJ^{-1}\right)
\end{equation}
\end{prop}

\begin{proof}
Establishing the asymptotic Gaussianity is straightforward once consistency
is established. The following proof structure is almost identical
to Th. 5.41 of \citet{van2000asymptotic}.

We start from a Taylor expansion of $\nabla\phi_{n}\left(\hat{\boldsymbol{\theta}}_{n}\right)=0$:
\begin{equation}
0=\nabla\phi_{n}\left(\boldsymbol{\theta}_{0}\right)+H\phi_{n}\left(\boldsymbol{\theta}_{0}\right)\left(\hat{\boldsymbol{\theta}}_{n}-\boldsymbol{\theta}_{0}\right)+\frac{1}{2}\sum_{i=1}^{p}\phi_{n}^{\left(3\right)}\left(\tilde{\boldsymbol{\theta}}_{n}\right)\left[\boldsymbol{e}_{i},\hat{\boldsymbol{\theta}}_{n}-\boldsymbol{\theta}_{0},\hat{\boldsymbol{\theta}}_{n}-\boldsymbol{\theta}_{0}\right]\boldsymbol{e}_{i}
\end{equation}
where $\boldsymbol{e}_{i}$ are the canonical basis of $\R^{p}$ and
$\tilde{\boldsymbol{\theta}}_{n}$ stands somewhere on the line from
$\hat{\boldsymbol{\theta}}_{n}$ to $\boldsymbol{\theta}_{0}$. By
dividing the expression by $n$, we obtain:
\begin{equation}
\frac{1}{n}\nabla\phi_{n}\left(\boldsymbol{\theta}_{0}\right)+\frac{1}{n}H\phi_{n}\left(\boldsymbol{\theta}_{0}\right)\left(\hat{\boldsymbol{\theta}}_{n}-\boldsymbol{\theta}_{0}\right)+\frac{1}{2}\sum_{i=1}^{p}\frac{1}{n}\phi_{n}^{\left(3\right)}\left(\tilde{\boldsymbol{\theta}}_{n}\right)\left[\boldsymbol{e}_{i},\hat{\boldsymbol{\theta}}_{n}-\boldsymbol{\theta}_{0},\hat{\boldsymbol{\theta}}_{n}-\boldsymbol{\theta}_{0}\right]\boldsymbol{e}_{i}=0
\end{equation}
where we recognize that: $\frac{1}{n}\nabla\phi_{n}\left(\boldsymbol{\theta}_{0}\right)$
is of order $n^{-1/2}$ (Lemma \ref{lem:Asymptotics-of-phi_n (theta_0)}),
$\frac{1}{n}H\phi_{n}\left(\boldsymbol{\theta}_{0}\right)=J+o_{P}\left(1\right)$
(Lemma \ref{lem:Asymptotics-of-phi_n (theta_0)}) and $\frac{1}{n}\phi_{n}^{\left(3\right)}\left(\tilde{\boldsymbol{\theta}}_{n}\right)$
is of order $\mathcal{O}_{P}\left(1\right)$ (Lemma \ref{lem: APPENDIX Asymptotics-of-the third derivative}).
We can thus rewrite the equation as:
\begin{equation}
-\frac{1}{n}\nabla\phi_{n}\left(\boldsymbol{\theta}_{0}\right)=\left(J+o_{P}\left(1\right)+\mathcal{O}_{P}\left(1\right)\left(\hat{\boldsymbol{\theta}}_{n}-\boldsymbol{\theta}_{0}\right)\right)\left(\hat{\boldsymbol{\theta}}_{n}-\boldsymbol{\theta}_{0}\right)
\end{equation}
which we can further simplify using Slutsky's Theorem: because of
the consistency of $\hat{\boldsymbol{\theta}}_{n}$, we have:
\begin{equation}
\mathcal{O}_{P}\left(1\right)\left(\hat{\boldsymbol{\theta}}_{n}-\boldsymbol{\theta}_{0}\right)=o_{P}\left(1\right)
\end{equation}

Thus, we have the further simplification:
\begin{equation}
-\frac{1}{n}\nabla\phi_{n}\left(\boldsymbol{\theta}_{0}\right)=\left(J+o_{P}\left(1\right)\right)\left(\hat{\boldsymbol{\theta}}_{n}-\boldsymbol{\theta}_{0}\right)
\end{equation}
\end{proof}
In the limit, the matrix $J+o_{P}\left(1\right)$ will be invertible
with probability $1$. We then finally have:
\begin{align}
\hat{\boldsymbol{\theta}}_{n}-\boldsymbol{\theta}_{0} & =-\frac{1}{n}\left[J+o_{P}\left(1\right)\right]^{-1}\nabla\phi_{n}\left(\boldsymbol{\theta}_{0}\right)\\
 & =-\frac{1}{n}J^{-1}\nabla\phi_{n}\left(\boldsymbol{\theta}_{0}\right)+\frac{1}{n}o_{P}\left(1\right)\nabla\phi_{n}\left(\boldsymbol{\theta}_{0}\right)\\
 & =-\frac{1}{\sqrt{n}}J^{-1}\left[\frac{1}{\sqrt{n}}\nabla\phi_{n}\left(\boldsymbol{\theta}_{0}\right)\right]+o_{P}\left(n^{-1/2}\right)
\end{align}
which is indeed asymptotically Gaussian with mean 0 and covariance
$\frac{1}{n}J^{-1}IJ^{-1}$ (Lemma \ref{lem:Asymptotics-of-phi_n (theta_0)}).

\subsection{Proofs: properties of $\phi_{n}$ and $\tilde{\phi}_{n}$.}
\begin{lem}
Properties of $\phi_{n}$.\label{lem:Properties-of-phi_n}

\textbf{Requires:} Prop.\ref{prop: APPENDIX MAP is gaussian}.

The log-curvature at the MAP estimate grows linearly:
\begin{equation}
H\phi_{n}\left(\hat{\boldsymbol{\theta}}_{n}\right)=nJ+o_{P}\left(n\right)
\end{equation}
Thus, the covariance of the Laplace approximation goes to 0 as:
\begin{equation}
\Sigma_{n}=\left[nJ\right]^{-1}+o_{P}\left(n^{-1}\right)
\end{equation}
\end{lem}

\begin{proof}
This result is straightforward.

A Taylor expansion yields:
\begin{equation}
\frac{1}{n}H\phi_{n}\left(\hat{\boldsymbol{\theta}}_{n}\right)=\frac{1}{n}H\phi_{n}\left(\boldsymbol{\theta}_{0}\right)+\frac{1}{n}\sum_{i,j}^{p}\phi_{n}^{\left(3\right)}\left(\tilde{\boldsymbol{\theta}}_{n}\right)\left[\boldsymbol{e}_{i},\boldsymbol{e}_{j},\hat{\boldsymbol{\theta}}_{n}-\boldsymbol{\theta}_{0}\right]\boldsymbol{e}_{i}\boldsymbol{e}_{j}
\end{equation}

The second term is a product of a $\mathcal{O}_{P}\left(1\right)$
term due to the third derivative and a $\mathcal{O}_{P}\left(n^{-1/2}\right)$
term due to $\hat{\boldsymbol{\theta}}_{n}-\boldsymbol{\theta}_{0}$.
Thus:
\begin{align}
\frac{1}{n}H\phi_{n}\left(\hat{\boldsymbol{\theta}}_{n}\right) & =\frac{1}{n}H\phi_{n}\left(\boldsymbol{\theta}_{0}\right)+\mathcal{O}_{P}\left(n^{-1/2}\right)\\
 & =J+o_{P}\left(1\right)\\
H\phi_{n}\left(\hat{\boldsymbol{\theta}}_{n}\right) & =nJ+o_{P}\left(n\right)
\end{align}

The covariance of the Laplace approximation is the inverse of the
log-curvature:
\begin{align}
\Sigma_{n} & =\left[H\phi_{n}\left(\hat{\boldsymbol{\theta}}_{n}\right)\right]^{-1}\\
 & =\left[nJ+o_{P}\left(1\right)\right]^{-1}\\
 & =\left[nJ\right]^{-1}+o_{P}\left(n^{-1}\right)
\end{align}
which concludes the proof.
\end{proof}
\begin{lem}
Properties of $\tilde{\phi}_{n}$.\label{lem:Properties-of- standardized phi_n}

\textbf{Requires:} Prop.\ref{prop: APPENDIX MAP is gaussian} and
Lemmas \ref{lem: APPENDIX Asymptotics-of-the third derivative} and
\ref{lem:Properties-of-phi_n}.

The third-derivative of the standardized log-density $\tilde{\phi}_{n}$
decays at rate $n^{-1/2}$:
\begin{align}
\max_{\boldsymbol{\theta},v_{1},v_{2},v_{3}}\frac{\left|\tilde{\phi}_{n}^{\left(3\right)}\left(\boldsymbol{\theta}\right)\left[v_{1},v_{2},v_{3}\right]\right|}{\left\Vert v_{1}\right\Vert _{2}\left\Vert v_{2}\right\Vert _{2}\left\Vert v_{3}\right\Vert _{2}} & =\max_{\boldsymbol{\theta},v_{1},v_{2},v_{3}}\frac{\left|\phi_{n}^{\left(3\right)}\left(\boldsymbol{\theta}\right)\left[\Sigma_{n}^{-1/2}v_{1},\Sigma_{n}^{-1/2}v_{2},\Sigma_{n}^{-1/2}v_{3}\right]\right|}{\left\Vert v_{1}\right\Vert _{2}\left\Vert v_{2}\right\Vert _{2}\left\Vert v_{3}\right\Vert _{2}}\\
 & =n^{-1/2}\E\left(\Delta^{\left(i\right)}\right)+o_{P}\left(n^{-1/2}\right)
\end{align}
\end{lem}

\begin{proof}
Lemma \ref{lem: APPENDIX Asymptotics-of-the third derivative} yields
that:
\begin{equation}
\max_{\boldsymbol{\theta},v_{1},v_{2},v_{3}}\frac{\left|\phi_{n}^{\left(3\right)}\left(\boldsymbol{\theta}\right)\left[J^{1/2}v_{1},J^{1/2}v_{2},J^{1/2}v_{3}\right]\right|}{\left\Vert v_{1}\right\Vert _{2}\left\Vert v_{2}\right\Vert _{2}\left\Vert v_{3}\right\Vert _{2}}\leq n\E\left(\Delta^{\left(i\right)}\right)+o_{P}\left(1\right)
\end{equation}

We thus just need to replace the $\Sigma_{n}^{-1/2}$ terms by $J^{1/2}$:
\begin{align}
 & \max_{\boldsymbol{\theta},v_{1},v_{2},v_{3}}\frac{\left|\phi_{n}^{\left(3\right)}\left(\boldsymbol{\theta}\right)\left[\Sigma_{n}^{-1/2}v_{1},\Sigma_{n}^{-1/2}v_{2},\Sigma_{n}^{-1/2}v_{3}\right]\right|}{\left\Vert v_{1}\right\Vert _{2}\left\Vert v_{2}\right\Vert _{2}\left\Vert v_{3}\right\Vert _{2}}\nonumber \\
 & =\max_{\boldsymbol{\theta},w_{1},w_{2},w_{3}}\frac{\left|\phi_{n}^{\left(3\right)}\left(\boldsymbol{\theta}\right)\left[J^{1/2}w_{1},J^{1/2}w_{2},J^{1/2}w_{3}\right]\right|}{\left\Vert J^{1/2}\Sigma_{n}^{1/2}w_{1}\right\Vert _{2}\left\Vert J^{1/2}\Sigma_{n}^{1/2}w_{2}\right\Vert _{2}\left\Vert J^{1/2}\Sigma_{n}^{1/2}w_{3}\right\Vert _{2}}
\end{align}
where $w_{i}=J^{-1/2}\Sigma_{n}^{-1/2}v_{1}$.

In order to restore the $\left\Vert w_{i}\right\Vert $ terms in the
numerator, we need to bound the ratio of the norms:
\begin{equation}
\frac{\left\Vert w_{1}\right\Vert }{\left\Vert J^{1/2}\Sigma_{n}^{1/2}w_{1}\right\Vert _{2}}\leq\text{eigmin}\left(J^{1/2}\Sigma_{n}^{1/2}\right)
\end{equation}
which, since $\Sigma_{n}=\left[nJ\right]^{-1}+o_{P}\left(n^{-1}\right)$,
we have:
\begin{align}
\Sigma_{n}^{1/2} & =\left[nJ\right]^{-1/2}+o_{P}\left(n^{-1/2}\right)\\
\text{eigmin}\left(J^{1/2}\Sigma_{n}^{1/2}\right) & =n^{-1/2}+o_{P}\left(n^{-1/2}\right)
\end{align}
and thus:
\begin{align}
 & \max_{\boldsymbol{\theta},v_{1},v_{2},v_{3}}\frac{\left|\phi_{n}^{\left(3\right)}\left(\boldsymbol{\theta}\right)\left[\Sigma_{n}^{-1/2}v_{1},\Sigma_{n}^{-1/2}v_{2},\Sigma_{n}^{-1/2}v_{3}\right]\right|}{\left\Vert v_{1}\right\Vert _{2}\left\Vert v_{2}\right\Vert _{2}\left\Vert v_{3}\right\Vert _{2}}\nonumber \\
 & \leq\left[n^{-3/2}+o_{P}\left(n^{-3/2}\right)\right]\max_{\boldsymbol{\theta},w_{1},w_{2},w_{3}}\frac{\left|\phi_{n}^{\left(3\right)}\left(\boldsymbol{\theta}\right)\left[J^{1/2}w_{1},J^{1/2}w_{2},J^{1/2}w_{3}\right]\right|}{\left\Vert w_{1}\right\Vert _{2}\left\Vert w_{2}\right\Vert _{2}\left\Vert w_{3}\right\Vert _{2}}\\
 & \leq\left[n^{-1/2}+o_{P}\left(n^{-1/2}\right)\right]\E\left(\Delta^{\left(i\right)}\right)\\
 & \leq n^{-1/2}\E\left(\Delta^{\left(i\right)}\right)+o_{P}\left(n^{-1/2}\right)
\end{align}
which concludes the proof.
\end{proof}

\subsection{Proof of Cor.\ref{cor:Bernstein-von-Mises:-IID case}}
\begin{proof}
We can now finally prove the Main text corollary \ref{cor:Bernstein-von-Mises:-IID case}.

First, let us verify that we can indeed apply Theorem \ref{thm:A-deterministic-BvM theorem}.
The dimensionality is fixed and Lemma \ref{lem:Properties-of- standardized phi_n}
shows that $\Delta_{3}=\mathcal{O}_{P}\left(n^{-1/2}\right)$. We
thus indeed have that $\Delta_{3}p^{3/2}\rightarrow0$.

Thus, Theorem \ref{thm:A-deterministic-BvM theorem} yields that:
\begin{align}
KL\left(g_{n},f_{n}\right) & =\mathcal{O}\left(\left(\Delta_{3}\right)^{2}p^{3}\right)\\
 & =\mathcal{O}_{P}\left(n^{-1}\right)
\end{align}
concluding the proof.
\end{proof}

\section{Approximations $KL\left(g_{LAP},f\right)$}

\label{sec: APPENDIX Approximations-of-the KL divergence}

This section deals with the proof of Cor.\ref{cor:Computable-approximations-of KL(g,f)}
which derives computable approximations of $KL\left(g_{LAP},f\right)$.

\subsection{Proof structure}

We establish the relevance of each approximation detailed in Cor.\ref{cor:Computable-approximations-of KL(g,f)}
in its own separate Lemma. In order to improve the exposition, we
change the order compared to the order in which the approximations
are presented in Cor.\ref{cor:Computable-approximations-of KL(g,f)}.
\begin{enumerate}
\item First, the three approximations that require sampling from $g$:
\begin{enumerate}
\item LSI approximation of $\E_{\boldsymbol{e}\sim g\left(\boldsymbol{e}\right)}\left[KL\left(r_{g},r_{f}|\boldsymbol{e}\right)\right]$.
\item varELBO approximation of $KL\left(\boldsymbol{e}_{g},\boldsymbol{e}_{f}\right)$.
\item KL variance approximation of $KL\left(g_{LAP},f\right)$.
\end{enumerate}
\item Then, the approximations of those three approximations removing the
expected values through a Taylor expansion of $\tilde{\phi}_{f}\left(\boldsymbol{\theta}\right)$
at $\boldsymbol{\mu}$.
\begin{enumerate}
\item LSI approximation of $\E_{\boldsymbol{e}\sim g\left(\boldsymbol{e}\right)}\left[KL\left(r_{g},r_{f}|\boldsymbol{e}\right)\right]$.
\item varELBO approximation of $KL\left(\boldsymbol{e}_{g},\boldsymbol{e}_{f}\right)$.
\item KL variance approximation of $KL\left(g_{LAP},f\right)$.
\end{enumerate}
\end{enumerate}

\subsection{Proof: sampling-based approximations}
\begin{lem}
Computable approximation of $\E_{\boldsymbol{e}\sim g\left(\boldsymbol{e}\right)}\left[KL\left(r_{g},r_{f}|\boldsymbol{e}\right)\right]$.\label{lem: APPENDIX approx of KL due to r}

\textbf{Requires:} Main text Thm.\ref{thm:A-deterministic-BvM theorem}.

From Theorem \ref{thm:A-deterministic-BvM theorem}, we can approximate
the KL divergence due to $r$ as:
\begin{equation}
\E_{\boldsymbol{e}\sim g\left(\boldsymbol{e}\right)}\left[KL\left(r_{g},r_{f}|\boldsymbol{e}\right)\right]\approx LSI=\frac{1}{2}\frac{1}{p^{2/3}}\E_{r,\boldsymbol{e}\sim g\left(r,\boldsymbol{e}\right)}\left[r^{4/3}\left(\boldsymbol{e}^{T}\nabla\tilde{\phi}_{f}\left(r\boldsymbol{e}\right)-r\right)^{2}\right]
\end{equation}

This approximation can be computed by sampling from $r_{g},\boldsymbol{e}_{g}$.
\end{lem}

\begin{proof}
This simply requires integrating over $\boldsymbol{e}_{g}$ in Theorem
\ref{thm:A-deterministic-BvM theorem} and replacing $\varphi_{\boldsymbol{e}}^{'}\left(r\right)$:
\begin{equation}
\varphi_{\boldsymbol{e}}^{'}\left(r\right)=\boldsymbol{e}^{T}\nabla\tilde{\phi}_{f}\left(r\boldsymbol{e}\right)
\end{equation}
\end{proof}
\begin{lem}
Computable approximation of $KL\left(\boldsymbol{e}_{g},\boldsymbol{e}_{f}\right)$.\label{lem: APPENDIX approx of KL due to e VARELBO}

\textbf{Requires:} Main text Thm.\ref{thm:A-deterministic-BvM theorem}.

From Theorem \ref{thm:A-deterministic-BvM theorem}, we can approximate
the KL divergence due to $\boldsymbol{e}$ as:
\begin{equation}
KL\left(\boldsymbol{e}_{g},\boldsymbol{e}_{f}\right)\approx\frac{1}{2}\text{Var}_{\boldsymbol{e}\sim g_{LAP}\left(\boldsymbol{e}\right)}\left[\E_{r\sim g_{LAP}\left(r\right)}\left[\tilde{\phi}_{f}\left(r\boldsymbol{e}\right)-\frac{1}{2}r^{2}\right]\right]
\end{equation}

This approximation can be computed by sampling from $r_{g},\boldsymbol{e}_{g}$.
\end{lem}

\begin{proof}
No work is required in deriving this since this approximation is given
as is in Theorem \ref{thm:A-deterministic-BvM theorem}.
\end{proof}
\begin{lem}
Second computable approximation of $KL\left(\boldsymbol{e}_{g},\boldsymbol{e}_{f}\right)$.\label{lem: APPENDIX approx of KL due to e KLVAR}

\textbf{Requires:} Main text Thm.\ref{thm:A-deterministic-BvM theorem}.

From Theorem \ref{thm:A-deterministic-BvM theorem}, we can approximate
the KL divergence due to $\boldsymbol{e}$ as:
\begin{equation}
KL\left(\boldsymbol{e}_{g},\boldsymbol{e}_{f}\right)\approx\frac{1}{2}KL_{var}\left(g_{LAP},f\right)
\end{equation}
This approximation is an upper-bound.
\end{lem}

\begin{proof}
Once again, this approximation is given as is in Theorem \ref{thm:A-deterministic-BvM theorem}.

We recall that this approximation is an upper-bound due to the fact
that $KL_{var}\left(g_{LAP},f\right)$ also includes a term that corresponds
to the KL divergence due to $r$.
\end{proof}
\begin{lem}
Heuristic computable approximation of $KL\left(g_{LAP},f\right)$.\label{lem: APPENDIX approx of KL(g,f) by KLvar}

\textbf{Requires:} Main text Thm.\ref{thm:A-deterministic-BvM theorem},
Lemma \ref{lem:Relationship-between-Var(ELBO) and KL-var}.

Theorem \ref{thm:A-deterministic-BvM theorem} further gives \textbf{heuristic
support} for using the KL variance as an approximation of $KL\left(g_{LAP},f\right)$.
\begin{equation}
KL\left(g_{LAP},f\right)\approx\frac{1}{2}KL_{var}\left(g_{LAP},f\right)
\end{equation}
\end{lem}

\begin{proof}
Since the KL variance $KL_{var}\left(g_{LAP},f\right)$ already counts
the contribution of $r$ to the KL divergence (see the proof of Lemma
\ref{lem:Relationship-between-Var(ELBO) and KL-var}), it seems unwise
to double-count this contribution by further adding the LSI approximation
of the contribution of $r$.
\end{proof}

\subsection{Proof: Taylor-based approximations}
\begin{lem}
Further approximation of the LSI approximation.\label{lem: APPENDIX Taylor approx of LSI}

\textbf{Requires: }Lemma \ref{lem: APPENDIX Moments-of-r_g}.

Performing a Taylor approximation, and an approximation of the expected
value under $r$ yields:
\begin{align}
LSI & \approx\frac{1}{p}\Bigg[\frac{3}{4}\sum_{i,j,k}\left\{ \left[\tilde{\phi}_{f}^{\left(3\right)}\left(0\right)\right]_{i,j,k}\right\} ^{2}+\frac{9}{8}\sum_{i,j,k}\left[\tilde{\phi}_{f}^{\left(3\right)}\left(0\right)\right]_{i,j,j}\left[\tilde{\phi}_{f}^{\left(3\right)}\left(0\right)\right]_{i,k,k}\nonumber \\
 & \ \ \ \ +\frac{1}{3}\sum_{i,j,k,l}\left\{ \left[\tilde{\phi}_{f}^{\left(4\right)}\left(0\right)\right]_{i,j,k,l}\right\} ^{2}+\sum_{i,j,k,l}\left\{ \left[\tilde{\phi}_{f}^{\left(4\right)}\left(0\right)\right]_{i,j,k,k}\right\} \left\{ \left[\tilde{\phi}_{f}^{\left(4\right)}\left(0\right)\right]_{i,j,l,l}\right\} \nonumber \\
 & \ \ \ \ +\frac{1}{8}\sum_{i,j,k,l}\left\{ \left[\tilde{\phi}_{f}^{\left(4\right)}\left(0\right)\right]_{i,i,j,j}\right\} \left\{ \left[\tilde{\phi}_{f}^{\left(4\right)}\left(0\right)\right]_{k,k,l,l}\right\} \Bigg]
\end{align}
\end{lem}

\begin{proof}
The approximation of Lemma \ref{lem: APPENDIX Moments-of-r_g} implies
that $r_{g}$ is equivalent to $p^{1/2}$ asymptotically. We can then
simplify the expected value of interest by introducing an additional
$r^{2/3}$:
\begin{align}
\E_{r,\boldsymbol{e}\sim g\left(r,\boldsymbol{e}\right)}\left[r^{4/3}\left(\boldsymbol{e}^{T}\nabla\tilde{\phi}_{f}\left(r\boldsymbol{e}\right)-r\right)^{2}\right] & \approx\frac{1}{p^{1/3}}\E_{r,\boldsymbol{e}\sim g\left(r,\boldsymbol{e}\right)}\left[r^{2}\left(\boldsymbol{e}^{T}\nabla\tilde{\phi}_{f}\left(r\boldsymbol{e}\right)-r\right)^{2}\right]\\
 & \approx\frac{1}{p^{1/3}}\E_{r,\boldsymbol{e}\sim g\left(r,\boldsymbol{e}\right)}\left[\left(r\boldsymbol{e}^{T}\nabla\tilde{\phi}_{f}\left(r\boldsymbol{e}\right)-r^{2}\right)^{2}\right]\\
 & \approx\frac{1}{p^{1/3}}\E_{\tilde{\boldsymbol{\theta}}\sim g\left(\tilde{\boldsymbol{\theta}}\right)}\left[\left(\tilde{\boldsymbol{\theta}}^{T}\nabla\tilde{\phi}_{f}\left(\tilde{\boldsymbol{\theta}}\right)-\left\Vert \tilde{\boldsymbol{\theta}}\right\Vert _{2}^{2}\right)^{2}\right]\\
 & \approx\frac{1}{p^{1/3}}\E_{\tilde{\boldsymbol{\theta}}\sim g\left(\tilde{\boldsymbol{\theta}}\right)}\left[\left(\tilde{\boldsymbol{\theta}}^{T}\nabla\tilde{\phi}_{f}\left(\tilde{\boldsymbol{\theta}}\right)-\tilde{\boldsymbol{\theta}}^{T}\nabla\tilde{\phi}_{g}\left(\tilde{\boldsymbol{\theta}}\right)\right)^{2}\right]
\end{align}
A Taylor expansion then yields:
\begin{align}
\tilde{\boldsymbol{\theta}}^{T}\nabla\tilde{\phi}_{f}\left(\tilde{\boldsymbol{\theta}}\right)-\tilde{\boldsymbol{\theta}}^{T}\nabla\tilde{\phi}_{g}\left(\tilde{\boldsymbol{\theta}}\right) & =\frac{1}{2}\tilde{\phi}_{f}^{\left(3\right)}\left(0\right)\left[\tilde{\boldsymbol{\theta}},\tilde{\boldsymbol{\theta}},\tilde{\boldsymbol{\theta}}\right]+\frac{1}{6}\tilde{\phi}_{f}^{\left(4\right)}\left(0\right)\left[\tilde{\boldsymbol{\theta}},\tilde{\boldsymbol{\theta}},\tilde{\boldsymbol{\theta}},\tilde{\boldsymbol{\theta}}\right]\\
\left(\tilde{\boldsymbol{\theta}}^{T}\nabla\tilde{\phi}_{f}\left(\tilde{\boldsymbol{\theta}}\right)-\tilde{\boldsymbol{\theta}}^{T}\nabla\tilde{\phi}_{g}\left(\tilde{\boldsymbol{\theta}}\right)\right)^{2} & =\frac{1}{4}\left\{ \tilde{\phi}_{f}^{\left(3\right)}\left(0\right)\left[\tilde{\boldsymbol{\theta}},\tilde{\boldsymbol{\theta}},\tilde{\boldsymbol{\theta}}\right]\right\} ^{2}\nonumber \\
 & \ \ \ \ +2\frac{1}{12}\left\{ \tilde{\phi}_{f}^{\left(3\right)}\left(0\right)\left[\tilde{\boldsymbol{\theta}},\tilde{\boldsymbol{\theta}},\tilde{\boldsymbol{\theta}}\right]\right\} \left\{ \tilde{\phi}_{f}^{\left(4\right)}\left(0\right)\left[\tilde{\boldsymbol{\theta}},\tilde{\boldsymbol{\theta}},\tilde{\boldsymbol{\theta}},\tilde{\boldsymbol{\theta}}\right]\right\} \nonumber \\
 & \ \ \ \ +\frac{1}{36}\left\{ \tilde{\phi}_{f}^{\left(4\right)}\left(0\right)\left[\tilde{\boldsymbol{\theta}},\tilde{\boldsymbol{\theta}},\tilde{\boldsymbol{\theta}},\tilde{\boldsymbol{\theta}}\right]\right\} ^{2}
\end{align}

We can now easily compute the expected value under $\tilde{\boldsymbol{\theta}}$.
The cross-term vanishes because odd moments of the Gaussian are 0:
\begin{align}
\E_{\tilde{\boldsymbol{\theta}}\sim g\left(\tilde{\boldsymbol{\theta}}\right)}\left[\left(\tilde{\boldsymbol{\theta}}^{T}\nabla\tilde{\phi}_{f}\left(\tilde{\boldsymbol{\theta}}\right)-\tilde{\boldsymbol{\theta}}^{T}\nabla\tilde{\phi}_{g}\left(\tilde{\boldsymbol{\theta}}\right)\right)^{2}\right] & \approx\frac{1}{4}\E_{\tilde{\boldsymbol{\theta}}\sim g\left(\tilde{\boldsymbol{\theta}}\right)}\left[\left\{ \tilde{\phi}_{f}^{\left(3\right)}\left(0\right)\left[\tilde{\boldsymbol{\theta}},\tilde{\boldsymbol{\theta}},\tilde{\boldsymbol{\theta}}\right]\right\} ^{2}\right]\nonumber \\
 & \ \ \ \ +\frac{1}{36}\E_{\tilde{\boldsymbol{\theta}}\sim g\left(\tilde{\boldsymbol{\theta}}\right)}\left[\left\{ \tilde{\phi}_{f}^{\left(4\right)}\left(0\right)\left[\tilde{\boldsymbol{\theta}},\tilde{\boldsymbol{\theta}},\tilde{\boldsymbol{\theta}},\tilde{\boldsymbol{\theta}}\right]\right\} ^{2}\right]
\end{align}

Now, let us tackle each term in order:
\begin{equation}
\E_{\tilde{\boldsymbol{\theta}}\sim g\left(\tilde{\boldsymbol{\theta}}\right)}\left[\left\{ \tilde{\phi}_{f}^{\left(3\right)}\left(0\right)\left[\tilde{\boldsymbol{\theta}},\tilde{\boldsymbol{\theta}},\tilde{\boldsymbol{\theta}}\right]\right\} ^{2}\right]=\E_{\tilde{\boldsymbol{\theta}}\sim g\left(\tilde{\boldsymbol{\theta}}\right)}\left[\sum_{\stackrel{a,b,c}{i,j,k}}\left[\tilde{\phi}_{f}^{\left(3\right)}\left(0\right)\right]_{a,b,c}\left[\tilde{\phi}_{f}^{\left(3\right)}\left(0\right)\right]_{i,j,k}\tilde{\theta}_{a}\tilde{\theta}_{b}\tilde{\theta}_{c}\tilde{\theta}_{i}\tilde{\theta}_{j}\tilde{\theta}_{k}\right]
\end{equation}
The sixth moment of the Gaussian is a sum of all the 15 partitions
into pairs of $\left\{ a,b,c,i,j,k\right\} $ (Isserlis' Theorem):
\begin{align}
\E\left[\tilde{\theta}_{a}\tilde{\theta}_{b}\tilde{\theta}_{c}\tilde{\theta}_{i}\tilde{\theta}_{j}\tilde{\theta}_{k}\right] & =\E\left[\tilde{\theta}_{a}\tilde{\theta}_{b}\right]\E\left[\tilde{\theta}_{c}\tilde{\theta}_{i}\right]\E\left[\tilde{\theta}_{j}\tilde{\theta}_{k}\right]+\E\left[\tilde{\theta}_{a}\tilde{\theta}_{c}\right]\E\left[\tilde{\theta}_{b}\tilde{\theta}_{i}\right]\E\left[\tilde{\theta}_{j}\tilde{\theta}_{k}\right]+\dots\\
 & =I_{a,b}I_{c,i}I_{j,k}+I_{a,c}I_{b,i}I_{j,k}+\dots
\end{align}

We can simplify the expected value by using the symmetry of the coordinates
$\left[\tilde{\phi}_{f}^{\left(3\right)}\right]_{a,b,c}$ to permutations.
Among the 15 partitions into pairs, 6 corresponds to pairs that all
cross from $abc$ into $ijk$ (e.g the pairs $ai$, $bj$, $ck$ or
$aj$, $bk$, $ci$) while there are 9 with at least one pair that
is internal to one group. The 6 pairs that cross can all be brought
back by symmetry to $ai$, $bj$, $ck$. The crossing pairs can be
similarly be brought back to $ai$, $bc$, $jk$.

We then have:
\begin{equation}
\E_{\tilde{\boldsymbol{\theta}}\sim g\left(\tilde{\boldsymbol{\theta}}\right)}\left[\left\{ \tilde{\phi}_{f}^{\left(3\right)}\left(0\right)\left[\tilde{\boldsymbol{\theta}},\tilde{\boldsymbol{\theta}},\tilde{\boldsymbol{\theta}}\right]\right\} ^{2}\right]=6\sum_{i,j,k}\left\{ \left[\tilde{\phi}_{f}^{\left(3\right)}\left(0\right)\right]_{i,j,k}\right\} ^{2}+9\sum_{i,j,k}\left[\tilde{\phi}_{f}^{\left(3\right)}\left(0\right)\right]_{i,j,j}\left[\tilde{\phi}_{f}^{\left(3\right)}\left(0\right)\right]_{i,k,k}
\end{equation}

Similarly, for the fourth derivative we have 105 total pairs. There
are 24 pairs that are fully crossing ($24=4!$, 4 choices for $a$,
then $3$ for $b$ etc), 9 pairs that are fully internal ($9=3*3$,
$3$ choices for $a$, the next pair in $abcd$ is forced, then $3$
choices for $i$ and the final pair in $ijkl$ is forced) and finally
72 pairs that are composed of two internal pairs and two crossing
pairs ($72=\binom{4}{2}\binom{4}{2}2$, selecting each internal pair
on each side, then 2 possibilities for the crossing pair). Symmetrizing
then leads to:
\begin{align}
\E_{\tilde{\boldsymbol{\theta}}\sim g\left(\tilde{\boldsymbol{\theta}}\right)}\left[\left\{ \tilde{\phi}_{f}^{\left(4\right)}\left[\tilde{\boldsymbol{\theta}},\tilde{\boldsymbol{\theta}},\tilde{\boldsymbol{\theta}},\tilde{\boldsymbol{\theta}}\right]\right\} ^{2}\right] & =24\sum_{i,j,k,l}\left\{ \left[\tilde{\phi}_{f}^{\left(4\right)}\left(0\right)\right]_{i,j,k,l}\right\} ^{2}\nonumber \\
 & \ \ \ \ +72\sum_{i,j,k,l}\left\{ \left[\tilde{\phi}_{f}^{\left(4\right)}\left(0\right)\right]_{i,j,k,k}\right\} \left\{ \left[\tilde{\phi}_{f}^{\left(4\right)}\left(0\right)\right]_{i,j,l,l}\right\} \nonumber \\
 & \ \ \ \ +9\sum_{i,j,k,l}\left\{ \left[\tilde{\phi}_{f}^{\left(4\right)}\left(0\right)\right]_{i,i,j,j}\right\} \left\{ \left[\tilde{\phi}_{f}^{\left(4\right)}\left(0\right)\right]_{k,k,l,l}\right\} 
\end{align}

Thus, we finally have:
\begin{align}
\E_{\tilde{\boldsymbol{\theta}}\sim g\left(\tilde{\boldsymbol{\theta}}\right)}\left[\left(\tilde{\boldsymbol{\theta}}^{T}\nabla\tilde{\phi}_{f}\left(\tilde{\boldsymbol{\theta}}\right)-\tilde{\boldsymbol{\theta}}^{T}\nabla\tilde{\phi}_{g}\left(\tilde{\boldsymbol{\theta}}\right)\right)^{2}\right] & \approx\frac{3}{2}\sum_{i,j,k}\left\{ \left[\tilde{\phi}_{f}^{\left(3\right)}\left(0\right)\right]_{i,j,k}\right\} ^{2}\nonumber \\
 & \ \ \ \ +\frac{9}{4}\sum_{i,j,k}\left[\tilde{\phi}_{f}^{\left(3\right)}\left(0\right)\right]_{i,j,j}\left[\tilde{\phi}_{f}^{\left(3\right)}\left(0\right)\right]_{i,k,k}\nonumber \\
 & \ \ \ \ +\frac{2}{3}\sum_{i,j,k,l}\left\{ \left[\tilde{\phi}_{f}^{\left(4\right)}\left(0\right)\right]_{i,j,k,l}\right\} ^{2}\nonumber \\
 & \ \ \ \ +2\sum_{i,j,k,l}\left\{ \left[\tilde{\phi}_{f}^{\left(4\right)}\left(0\right)\right]_{i,j,k,k}\right\} \left\{ \left[\tilde{\phi}_{f}^{\left(4\right)}\left(0\right)\right]_{i,j,l,l}\right\} \nonumber \\
 & \ \ \ \ +\frac{1}{4}\sum_{i,j,k,l}\left\{ \left[\tilde{\phi}_{f}^{\left(4\right)}\left(0\right)\right]_{i,i,j,j}\right\} \left\{ \left[\tilde{\phi}_{f}^{\left(4\right)}\left(0\right)\right]_{k,k,l,l}\right\} 
\end{align}

The approximation of the LSI term follows from:
\begin{align}
\frac{1}{2}\frac{1}{p^{2/3}}\E_{r,\boldsymbol{e}\sim g\left(r,\boldsymbol{e}\right)}\left[r^{4/3}\left(\boldsymbol{e}^{T}\nabla\tilde{\phi}_{f}\left(r\boldsymbol{e}\right)-r\right)^{2}\right] & \approx\frac{1}{2}\frac{1}{p}\E_{\tilde{\boldsymbol{\theta}}\sim g\left(\tilde{\boldsymbol{\theta}}\right)}\left[\left(\tilde{\boldsymbol{\theta}}^{T}\nabla\tilde{\phi}_{f}\left(\tilde{\boldsymbol{\theta}}\right)-\tilde{\boldsymbol{\theta}}^{T}\nabla\tilde{\phi}_{g}\left(\tilde{\boldsymbol{\theta}}\right)\right)^{2}\right]\\
 & \approx\frac{3}{4}\sum_{i,j,k}\left\{ \left[\tilde{\phi}_{f}^{\left(3\right)}\left(0\right)\right]_{i,j,k}\right\} ^{2}\nonumber \\
 & \ \ \ \ +\frac{9}{8}\sum_{i,j,k}\left[\tilde{\phi}_{f}^{\left(3\right)}\left(0\right)\right]_{i,j,j}\left[\tilde{\phi}_{f}^{\left(3\right)}\left(0\right)\right]_{i,k,k}\nonumber \\
 & \ \ \ \ +\frac{1}{3}\sum_{i,j,k,l}\left\{ \left[\tilde{\phi}_{f}^{\left(4\right)}\left(0\right)\right]_{i,j,k,l}\right\} ^{2}\nonumber \\
 & \ \ \ \ +\sum_{i,j,k,l}\left\{ \left[\tilde{\phi}_{f}^{\left(4\right)}\left(0\right)\right]_{i,j,k,k}\right\} \left\{ \left[\tilde{\phi}_{f}^{\left(4\right)}\left(0\right)\right]_{i,j,l,l}\right\} \nonumber \\
 & \ \ \ \ +\frac{1}{8}\sum_{i,j,k,l}\left\{ \left[\tilde{\phi}_{f}^{\left(4\right)}\left(0\right)\right]_{i,i,j,j}\right\} \left\{ \left[\tilde{\phi}_{f}^{\left(4\right)}\left(0\right)\right]_{k,k,l,l}\right\} 
\end{align}
\end{proof}
\begin{lem}
Further approximation of the varELBO approximation.\label{lem: APPENDIX Taylor approx of varELBO}

\textbf{Requires: }Lemma \ref{lem: APPENDIX Moments-of-r_g}.

Performing a Taylor approximation, and an approximation of the expected
value under $r$ yields:
\begin{align}
varELBO & \approx\frac{1}{6}\sum_{i,j,k}\left\{ \left[\tilde{\phi}_{f}^{\left(3\right)}\left(0\right)\right]_{i,j,k}\right\} ^{2}+\frac{1}{4}\sum_{i,j,k}\left[\tilde{\phi}_{f}^{\left(3\right)}\left(0\right)\right]_{i,j,j}\left[\tilde{\phi}_{f}^{\left(3\right)}\left(0\right)\right]_{i,k,k}\nonumber \\
 & \ \ \ \ +\frac{1}{24}\sum_{i,j,k,l}\left\{ \left[\tilde{\phi}_{f}^{\left(4\right)}\left(0\right)\right]_{i,j,k,l}\right\} ^{2}+\frac{1}{8}\sum_{i,j,k,l}\left\{ \left[\tilde{\phi}_{f}^{\left(4\right)}\left(0\right)\right]_{i,j,k,k}\right\} \left\{ \left[\tilde{\phi}_{f}^{\left(4\right)}\left(0\right)\right]_{i,j,l,l}\right\} 
\end{align}
\end{lem}

\begin{proof}
Recall the expression of the varELBO term:
\begin{equation}
varELBO=\text{Var}_{\boldsymbol{e}\sim g_{LAP}\left(\boldsymbol{e}\right)}\left[\E_{r\sim g_{LAP}\left(r\right)}\left[\tilde{\phi}_{f}\left(r\boldsymbol{e}\right)-\frac{1}{2}r^{2}\right]\right]
\end{equation}
Start from a Taylor expansion of $\tilde{\phi}_{f}\left(r\boldsymbol{e}\right)-\frac{1}{2}r^{2}$:
\begin{equation}
\tilde{\phi}_{f}\left(r\boldsymbol{e}\right)-\frac{1}{2}r^{2}\approx\frac{1}{6}\tilde{\phi}_{f}^{\left(3\right)}\left(0\right)\left[r\boldsymbol{e},r\boldsymbol{e},r\boldsymbol{e}\right]+\frac{1}{24}\tilde{\phi}_{f}^{\left(4\right)}\left(0\right)\left[r\boldsymbol{e},r\boldsymbol{e},r\boldsymbol{e},r\boldsymbol{e}\right]
\end{equation}

We can now approximate the expected value of this expression under
density $g\left(r\right)$ using Lemma \ref{lem: APPENDIX Moments-of-r_g}:
\begin{equation}
\E_{r\sim g_{LAP}\left(r\right)}\left[\tilde{\phi}_{f}\left(r\boldsymbol{e}\right)-\frac{1}{2}r^{2}\right]\approx\frac{p^{3/2}}{6}\tilde{\phi}_{f}^{\left(3\right)}\left(0\right)\left[\boldsymbol{e},\boldsymbol{e},\boldsymbol{e}\right]+\frac{p^{2}}{24}\tilde{\phi}_{f}^{\left(4\right)}\left(0\right)\left[\boldsymbol{e},\boldsymbol{e},\boldsymbol{e},\boldsymbol{e}\right]
\end{equation}
which we finally need to compute the variance of, under density $g\left(\boldsymbol{e}\right)$.

First, observe that, due to the symmetry of $\boldsymbol{e}$, the
covariance between the two terms is $0$:
\begin{align}
\E_{\boldsymbol{e}\sim g_{LAP}\left(\boldsymbol{e}\right)}\left[\tilde{\phi}_{f}^{\left(3\right)}\left(0\right)\left[\boldsymbol{e},\boldsymbol{e},\boldsymbol{e}\right]\right] & =0\\
\E_{\boldsymbol{e}\sim g_{LAP}\left(\boldsymbol{e}\right)}\left[\tilde{\phi}_{f}^{\left(3\right)}\left(0\right)\left[\boldsymbol{e},\boldsymbol{e},\boldsymbol{e}\right]\tilde{\phi}_{f}^{\left(4\right)}\left(0\right)\left[\boldsymbol{e},\boldsymbol{e},\boldsymbol{e},\boldsymbol{e}\right]\right] & =0\\
\text{Cov}_{\boldsymbol{e}\sim g_{LAP}\left(\boldsymbol{e}\right)}\left[\tilde{\phi}_{f}^{\left(3\right)}\left(0\right)\left[\boldsymbol{e},\boldsymbol{e},\boldsymbol{e}\right],\tilde{\phi}_{f}^{\left(4\right)}\left(0\right)\left[\boldsymbol{e},\boldsymbol{e},\boldsymbol{e},\boldsymbol{e}\right]\right] & =0
\end{align}
Thus we only need to compute the variance of both terms.

At this point, use once again Lemma \ref{lem: APPENDIX Moments-of-r_g}
in the reverse direction, i.e. as a heuristic to replace $p^{1/2}$
with $r_{g}$.
\begin{align}
\text{Var}_{\boldsymbol{e}\sim g_{LAP}\left(\boldsymbol{e}\right)}\left[p^{3/2}\tilde{\phi}_{f}^{\left(3\right)}\left(0\right)\left[\boldsymbol{e},\boldsymbol{e},\boldsymbol{e}\right]\right] & \approx\text{Var}_{r,\boldsymbol{e}\sim g_{LAP}\left(r,\boldsymbol{e}\right)}\left[\tilde{\phi}_{f}^{\left(3\right)}\left(0\right)\left[r\boldsymbol{e},r\boldsymbol{e},r\boldsymbol{e}\right]\right]\\
 & \approx\text{Var}_{\tilde{\boldsymbol{\theta}}\sim g_{LAP}\left(\tilde{\boldsymbol{\theta}}\right)}\left[\tilde{\phi}_{f}^{\left(3\right)}\left(0\right)\left[\tilde{\boldsymbol{\theta}},\tilde{\boldsymbol{\theta}},\tilde{\boldsymbol{\theta}}\right]\right]\\
\text{Var}_{\boldsymbol{e}\sim g_{LAP}\left(\boldsymbol{e}\right)}\left[p^{2}\tilde{\phi}_{f}^{\left(4\right)}\left(0\right)\left[\boldsymbol{e},\boldsymbol{e},\boldsymbol{e},\boldsymbol{e}\right]\right] & \approx\text{Var}_{\tilde{\boldsymbol{\theta}}\sim g_{LAP}\left(\tilde{\boldsymbol{\theta}}\right)}\left[\tilde{\phi}_{f}^{\left(4\right)}\left(0\right)\left[\tilde{\boldsymbol{\theta}},\tilde{\boldsymbol{\theta}},\tilde{\boldsymbol{\theta}},\tilde{\boldsymbol{\theta}}\right]\right]
\end{align}

Once again, we are dealing with moments of the form:
\[
\E_{\tilde{\boldsymbol{\theta}}\sim g\left(\tilde{\boldsymbol{\theta}}\right)}\left[\tilde{\phi}_{f}^{\left(4\right)}\left(0\right)\left[\tilde{\boldsymbol{\theta}},\tilde{\boldsymbol{\theta}},\tilde{\boldsymbol{\theta}},\tilde{\boldsymbol{\theta}}\right]\right]=\E_{\tilde{\boldsymbol{\theta}}\sim g\left(\tilde{\boldsymbol{\theta}}\right)}\left[\sum_{i,j,k,l}\left[\tilde{\phi}_{f}^{\left(4\right)}\left(0\right)\right]_{i,j,k,l}\tilde{\theta}_{i}\tilde{\theta}_{j}\tilde{\theta}_{k}\tilde{\theta}_{l}\right]
\]
which we can express using Isserlis' Theorem and then simplify by
using the symmetry of the derivatives.

Starting with the third derivative, the first expected value simplifies,
thus yielding the variance immediately (see the proof of Lemma \ref{lem: APPENDIX Taylor approx of LSI}
for details of the computation of the second moment of $\tilde{\phi}_{f}^{\left(3\right)}\left(0\right)\left[\tilde{\boldsymbol{\theta}},\tilde{\boldsymbol{\theta}},\tilde{\boldsymbol{\theta}}\right]$):
\begin{align}
\E_{\tilde{\boldsymbol{\theta}}\sim g\left(\tilde{\boldsymbol{\theta}}\right)}\left[\tilde{\phi}_{f}^{\left(3\right)}\left(0\right)\left[\tilde{\boldsymbol{\theta}},\tilde{\boldsymbol{\theta}},\tilde{\boldsymbol{\theta}}\right]\right] & =0\\
\E_{\tilde{\boldsymbol{\theta}}\sim g\left(\tilde{\boldsymbol{\theta}}\right)}\left[\left\{ \tilde{\phi}_{f}^{\left(3\right)}\left(0\right)\left[\tilde{\boldsymbol{\theta}},\tilde{\boldsymbol{\theta}},\tilde{\boldsymbol{\theta}}\right]\right\} ^{2}\right] & =6\sum_{i,j,k}\left\{ \left[\tilde{\phi}_{f}^{\left(3\right)}\left(0\right)\right]_{i,j,k}\right\} ^{2}+9\sum_{i,j,k}\left[\tilde{\phi}_{f}^{\left(3\right)}\left(0\right)\right]_{i,j,j}\left[\tilde{\phi}_{f}^{\left(3\right)}\left(0\right)\right]_{i,k,k}\\
\text{Var}_{\tilde{\boldsymbol{\theta}}\sim g_{LAP}\left(\tilde{\boldsymbol{\theta}}\right)}\left[\tilde{\phi}_{f}^{\left(3\right)}\left(0\right)\left[\tilde{\boldsymbol{\theta}},\tilde{\boldsymbol{\theta}},\tilde{\boldsymbol{\theta}}\right]\right] & =6\sum_{i,j,k}\left\{ \left[\tilde{\phi}_{f}^{\left(3\right)}\left(0\right)\right]_{i,j,k}\right\} ^{2}+9\sum_{i,j,k}\left[\tilde{\phi}_{f}^{\left(3\right)}\left(0\right)\right]_{i,j,j}\left[\tilde{\phi}_{f}^{\left(3\right)}\left(0\right)\right]_{i,k,k}
\end{align}

For the fourth derivative, we need to compute the expected value of
\begin{equation}
\tilde{\phi}_{f}^{\left(4\right)}\left(0\right)\left[\tilde{\boldsymbol{\theta}},\tilde{\boldsymbol{\theta}},\tilde{\boldsymbol{\theta}},\tilde{\boldsymbol{\theta}}\right]=\sum_{i,j,k,l}\left[\tilde{\phi}_{f}^{\left(4\right)}\left(0\right)\right]_{i,j,k,l}\tilde{\theta}_{i}\tilde{\theta}_{j}\tilde{\theta}_{k}\tilde{\theta}_{l}
\end{equation}
The fourth moment is composed of only 3 pairs which are of course
all ``internal'' (since there is only one group of terms):
\begin{equation}
\E_{\tilde{\boldsymbol{\theta}}\sim g\left(\tilde{\boldsymbol{\theta}}\right)}\left[\tilde{\phi}_{f}^{\left(4\right)}\left(0\right)\left[\tilde{\boldsymbol{\theta}},\tilde{\boldsymbol{\theta}},\tilde{\boldsymbol{\theta}},\tilde{\boldsymbol{\theta}}\right]\right]=3\sum_{i,j}\left[\tilde{\phi}_{f}^{\left(4\right)}\left(0\right)\right]_{i,i,j,j}
\end{equation}
We recall the second moment from the proof of Lemma \ref{lem: APPENDIX Taylor approx of LSI}:
\begin{align}
\E_{\tilde{\boldsymbol{\theta}}\sim g\left(\tilde{\boldsymbol{\theta}}\right)}\left[\left\{ \tilde{\phi}_{f}^{\left(4\right)}\left[\tilde{\boldsymbol{\theta}},\tilde{\boldsymbol{\theta}},\tilde{\boldsymbol{\theta}},\tilde{\boldsymbol{\theta}}\right]\right\} ^{2}\right] & =24\sum_{i,j,k,l}\left\{ \left[\tilde{\phi}_{f}^{\left(4\right)}\left(0\right)\right]_{i,j,k,l}\right\} ^{2}\nonumber \\
 & \ \ \ \ +72\sum_{i,j,k,l}\left\{ \left[\tilde{\phi}_{f}^{\left(4\right)}\left(0\right)\right]_{i,j,k,k}\right\} \left\{ \left[\tilde{\phi}_{f}^{\left(4\right)}\left(0\right)\right]_{i,j,l,l}\right\} \nonumber \\
 & \ \ \ \ +9\sum_{i,j,k,l}\left\{ \left[\tilde{\phi}_{f}^{\left(4\right)}\left(0\right)\right]_{i,i,j,j}\right\} \left\{ \left[\tilde{\phi}_{f}^{\left(4\right)}\left(0\right)\right]_{k,k,l,l}\right\} 
\end{align}

We then obtain the variance:
\begin{align}
\text{Var}_{\tilde{\boldsymbol{\theta}}\sim g\left(\tilde{\boldsymbol{\theta}}\right)}\left[\tilde{\phi}_{f}^{\left(4\right)}\left(0\right)\left[\tilde{\boldsymbol{\theta}},\tilde{\boldsymbol{\theta}},\tilde{\boldsymbol{\theta}},\tilde{\boldsymbol{\theta}}\right]\right] & =24\sum_{i,j,k,l}\left\{ \left[\tilde{\phi}_{f}^{\left(4\right)}\left(0\right)\right]_{i,j,k,l}\right\} ^{2}\nonumber \\
 & \ \ \ \ +72\sum_{i,j,k,l}\left\{ \left[\tilde{\phi}_{f}^{\left(4\right)}\left(0\right)\right]_{i,j,k,k}\right\} \left\{ \left[\tilde{\phi}_{f}^{\left(4\right)}\left(0\right)\right]_{i,j,l,l}\right\} 
\end{align}

We then finally get the rough approximation of the varELBO term:
\begin{align}
varELBO & =\text{Var}_{\boldsymbol{e}\sim g_{LAP}\left(\boldsymbol{e}\right)}\left[\E_{r\sim g_{LAP}\left(r\right)}\left[\tilde{\phi}_{f}\left(r\boldsymbol{e}\right)-\frac{1}{2}r^{2}\right]\right]\\
 & \approx\text{Var}_{\tilde{\boldsymbol{\theta}}\sim g_{LAP}\left(\tilde{\boldsymbol{\theta}}\right)}\left[\frac{1}{6}\tilde{\phi}_{f}^{\left(3\right)}\left(0\right)\left[\tilde{\boldsymbol{\theta}},\tilde{\boldsymbol{\theta}},\tilde{\boldsymbol{\theta}}\right]+\frac{1}{24}\tilde{\phi}_{f}^{\left(4\right)}\left(0\right)\left[\tilde{\boldsymbol{\theta}},\tilde{\boldsymbol{\theta}},\tilde{\boldsymbol{\theta}},\tilde{\boldsymbol{\theta}}\right]\right]\label{eq: APPENDIX yet another important intermediate}\\
 & \approx\frac{1}{36}\text{Var}_{\tilde{\boldsymbol{\theta}}\sim g_{LAP}\left(\tilde{\boldsymbol{\theta}}\right)}\left[\tilde{\phi}_{f}^{\left(3\right)}\left(0\right)\left[\tilde{\boldsymbol{\theta}},\tilde{\boldsymbol{\theta}},\tilde{\boldsymbol{\theta}}\right]\right]+\frac{1}{24^{2}}\text{Var}_{\tilde{\boldsymbol{\theta}}\sim g_{LAP}\left(\tilde{\boldsymbol{\theta}}\right)}\left[\tilde{\phi}_{f}^{\left(4\right)}\left(0\right)\left[\tilde{\boldsymbol{\theta}},\tilde{\boldsymbol{\theta}},\tilde{\boldsymbol{\theta}},\tilde{\boldsymbol{\theta}}\right]\right]\\
 & \approx\frac{6}{36}\sum_{i,j,k}\left\{ \left[\tilde{\phi}_{f}^{\left(3\right)}\left(0\right)\right]_{i,j,k}\right\} ^{2}+\frac{9}{36}\sum_{i,j,k}\left[\tilde{\phi}_{f}^{\left(3\right)}\left(0\right)\right]_{i,j,j}\left[\tilde{\phi}_{f}^{\left(3\right)}\left(0\right)\right]_{i,k,k}\nonumber \\
 & \ \ \ \ +\frac{24}{24^{2}}\sum_{i,j,k,l}\left\{ \left[\tilde{\phi}_{f}^{\left(4\right)}\left(0\right)\right]_{i,j,k,l}\right\} ^{2}+\frac{72}{24^{2}}\sum_{i,j,k,l}\left\{ \left[\tilde{\phi}_{f}^{\left(4\right)}\left(0\right)\right]_{i,j,k,k}\right\} \left\{ \left[\tilde{\phi}_{f}^{\left(4\right)}\left(0\right)\right]_{i,j,l,l}\right\} \\
 & \approx\frac{1}{6}\sum_{i,j,k}\left\{ \left[\tilde{\phi}_{f}^{\left(3\right)}\left(0\right)\right]_{i,j,k}\right\} ^{2}+\frac{1}{4}\sum_{i,j,k}\left[\tilde{\phi}_{f}^{\left(3\right)}\left(0\right)\right]_{i,j,j}\left[\tilde{\phi}_{f}^{\left(3\right)}\left(0\right)\right]_{i,k,k}\nonumber \\
 & \ \ \ \ +\frac{1}{24}\sum_{i,j,k,l}\left\{ \left[\tilde{\phi}_{f}^{\left(4\right)}\left(0\right)\right]_{i,j,k,l}\right\} ^{2}+\frac{1}{8}\sum_{i,j,k,l}\left\{ \left[\tilde{\phi}_{f}^{\left(4\right)}\left(0\right)\right]_{i,j,k,k}\right\} \left\{ \left[\tilde{\phi}_{f}^{\left(4\right)}\left(0\right)\right]_{i,j,l,l}\right\} 
\end{align}
\end{proof}
\begin{lem}
Further approximation of the KL variance approximation.\label{lem: APPENDIX Taylor approx of KLvar}

\textbf{Requires: }NA.

Performing a Taylor approximation yields:
\begin{align}
KL_{var}\left(g_{LAP},f\right) & \approx\frac{1}{6}\sum_{i,j,k}\left\{ \left[\tilde{\phi}_{f}^{\left(3\right)}\left(0\right)\right]_{i,j,k}\right\} ^{2}+\frac{1}{4}\sum_{i,j,k}\left[\tilde{\phi}_{f}^{\left(3\right)}\left(0\right)\right]_{i,j,j}\left[\tilde{\phi}_{f}^{\left(3\right)}\left(0\right)\right]_{i,k,k}\nonumber \\
 & \ \ \ \ +\frac{1}{24}\sum_{i,j,k,l}\left\{ \left[\tilde{\phi}_{f}^{\left(4\right)}\left(0\right)\right]_{i,j,k,l}\right\} ^{2}+\frac{1}{8}\sum_{i,j,k,l}\left\{ \left[\tilde{\phi}_{f}^{\left(4\right)}\left(0\right)\right]_{i,j,k,k}\right\} \left\{ \left[\tilde{\phi}_{f}^{\left(4\right)}\left(0\right)\right]_{i,j,l,l}\right\} 
\end{align}
\end{lem}

\begin{proof}
Once again, perform a Taylor expansion of $\tilde{\phi}_{f}\left(\tilde{\boldsymbol{\theta}}\right)-\phi_{g}\left(\tilde{\boldsymbol{\theta}}\right)$:
\begin{equation}
\tilde{\phi}_{f}\left(\tilde{\boldsymbol{\theta}}\right)-\phi_{g}\left(\tilde{\boldsymbol{\theta}}\right)\approx\frac{1}{6}\tilde{\phi}_{f}^{\left(3\right)}\left(0\right)\left[\tilde{\boldsymbol{\theta}},\tilde{\boldsymbol{\theta}},\tilde{\boldsymbol{\theta}}\right]+\frac{1}{24}\tilde{\phi}_{f}^{\left(4\right)}\left(0\right)\left[\tilde{\boldsymbol{\theta}},\tilde{\boldsymbol{\theta}},\tilde{\boldsymbol{\theta}},\tilde{\boldsymbol{\theta}}\right]
\end{equation}
thus yielding:
\begin{equation}
KL_{var}\left(g_{LAP},f\right)\approx\text{Var}_{\tilde{\boldsymbol{\theta}}\sim g_{LAP}\left(\tilde{\boldsymbol{\theta}}\right)}\left[\frac{1}{6}\tilde{\phi}_{f}^{\left(3\right)}\left(0\right)\left[\tilde{\boldsymbol{\theta}},\tilde{\boldsymbol{\theta}},\tilde{\boldsymbol{\theta}}\right]+\frac{1}{24}\tilde{\phi}_{f}^{\left(4\right)}\left(0\right)\left[\tilde{\boldsymbol{\theta}},\tilde{\boldsymbol{\theta}},\tilde{\boldsymbol{\theta}},\tilde{\boldsymbol{\theta}}\right]\right]
\end{equation}
where we recognize eq.\eqref{eq: APPENDIX yet another important intermediate}.
Thus, we read from the end of the proof of Lemma \eqref{lem: APPENDIX Taylor approx of varELBO}
that:
\begin{align}
KL_{var}\left(g_{LAP},f\right) & \approx\frac{1}{6}\sum_{i,j,k}\left\{ \left[\tilde{\phi}_{f}^{\left(3\right)}\left(0\right)\right]_{i,j,k}\right\} ^{2}+\frac{1}{4}\sum_{i,j,k}\left[\tilde{\phi}_{f}^{\left(3\right)}\left(0\right)\right]_{i,j,j}\left[\tilde{\phi}_{f}^{\left(3\right)}\left(0\right)\right]_{i,k,k}\nonumber \\
 & \ \ \ \ +\frac{1}{24}\sum_{i,j,k,l}\left\{ \left[\tilde{\phi}_{f}^{\left(4\right)}\left(0\right)\right]_{i,j,k,l}\right\} ^{2}+\frac{1}{8}\sum_{i,j,k,l}\left\{ \left[\tilde{\phi}_{f}^{\left(4\right)}\left(0\right)\right]_{i,j,k,k}\right\} \left\{ \left[\tilde{\phi}_{f}^{\left(4\right)}\left(0\right)\right]_{i,j,l,l}\right\} 
\end{align}
\end{proof}

\section{Details of the simulations}

\label{sec:APPENDIX Details-of-the simulations}

This section provides additional details and figures on the experimental
results that I obtained on the linear logistic classification model.

\subsection{Satisfying the assumptions of Cor.\ref{cor:Bernstein-von-Mises:-IID case}}

First, let us show that the model respects the assumptions of Cor.\ref{cor:Bernstein-von-Mises:-IID case}.

To do so, consider the log-posterior $\phi_{f}$ which is:
\begin{equation}
\phi_{f}\left(\boldsymbol{\theta}\right)=\frac{1}{2}\frac{1}{\left(\sigma_{prior}\right)^{2}}\left\Vert \boldsymbol{\theta}\right\Vert _{2}^{2}+\sum_{i=1}^{n}\log\left[1+\exp\left(-Y_{i}\boldsymbol{\theta}^{T}\boldsymbol{X}_{i}\right)\right]
\end{equation}
where $\left(\sigma_{prior}\right)^{2}$ is the prior variance.

The Hessian matrix of $\phi_{f}$ is:
\begin{equation}
H\phi_{f}\left(\boldsymbol{\theta}\right)=\frac{1}{\left(\sigma_{prior}\right)^{2}}I_{p}+\sum_{i=1}^{n}\frac{1}{\left[1+\exp\left(-\boldsymbol{\theta}^{T}\boldsymbol{X}_{i}\right)\right]\left[1+\exp\left(\boldsymbol{\theta}^{T}\boldsymbol{X}_{i}\right)\right]}\boldsymbol{X}_{i}\boldsymbol{X}_{i}^{T}
\end{equation}
which is clearly strictly positive, since it is a sum of positive
matrices.

Furthermore, consider the third derivative of one negative log-likelihood:
\begin{align}
NLL_{i}\left(\boldsymbol{\theta}\right) & =\log\left[1+\exp\left(-Y_{i}\boldsymbol{\theta}^{T}\boldsymbol{X}_{i}\right)\right]\\
\left[NLL_{i}^{\left(3\right)}\left(\boldsymbol{\theta}\right)\right]_{abc} & =\left[\boldsymbol{X}_{i}\right]_{a}\left[\boldsymbol{X}_{i}\right]_{b}\left[\boldsymbol{X}_{i}\right]_{c}l^{\left(3\right)}\left(Y_{i}\boldsymbol{\theta}^{T}\boldsymbol{X}_{i}\right)
\end{align}
where $l^{\left(3\right)}\left(a\right)$ is the third derivative
of negative logistic function:
\begin{align}
l^{\left(3\right)}\left(a\right) & =\frac{\left(-a\exp\left(-a\right)\left[1+\exp\left(a\right)\right]+a\exp\left(a\right)\left[1+\exp\left(-a\right)\right]\right)}{\left[1+\exp\left(-a\right)\right]^{2}\left[1+\exp\left(a\right)\right]^{2}}\\
 & =\frac{-a\exp\left(-a\right)}{\left[1+\exp\left(-a\right)\right]^{2}\left[1+\exp\left(a\right)\right]}+\frac{a\exp\left(a\right)}{\left[1+\exp\left(-a\right)\right]\left[1+\exp\left(a\right)\right]^{2}}\\
\left|l^{\left(3\right)}\left(a\right)\right| & \leq2\frac{\left|a\right|}{\left[1+\exp\left(-a\right)\right]\left[1+\exp\left(a\right)\right]}\\
 & <\infty
\end{align}
The third derivative of $NLL_{i}$ is thus bounded by the random quantity:
\begin{align}
\left[NLL_{i}^{\left(3\right)}\left(\boldsymbol{\theta}\right)\right]_{abc} & =\left[\boldsymbol{X}_{i}\right]_{a}\left[\boldsymbol{X}_{i}\right]_{b}\left[\boldsymbol{X}_{i}\right]_{c}l^{\left(3\right)}\left(Y_{i}\boldsymbol{\theta}^{T}\boldsymbol{X}_{i}\right)\\
\left|NLL_{i}^{\left(3\right)}\left(\boldsymbol{\theta}\right)\left[v_{1},v_{2},v_{3}\right]\right| & \leq\max_{a}\left|l^{\left(3\right)}\left(a\right)\right|\prod_{i=1}^{3}\boldsymbol{X}_{i}^{T}v_{i}\\
 & \leq\max_{a}\left|l^{\left(3\right)}\left(a\right)\right|\left\Vert X_{i}\right\Vert _{2}^{3}
\end{align}
where the random variable $\left\Vert X_{i}\right\Vert _{2}^{3}$
has bounded expected value.

We can thus apply Cor.\ref{cor:Bernstein-von-Mises:-IID case} to
this model.

\subsection{Details on the sampling approximations}

In order to approximate normalizing constants, I investigated two
methods:
\begin{enumerate}
\item I first considered a direct evaluation of the modified expression
for the KL divergence derived in the proof of Prop.\ref{prop: KL approx KL_var}
(more precisely in Lemma \ref{lem:Another-formula-for KL(g,f)}):
\begin{align}
KL\left(g_{LAP},f\right) & =\E_{\boldsymbol{\theta}\sim g_{LAP}\left(\boldsymbol{\theta}\right)}\left[\phi_{f}\left(\boldsymbol{\theta}\right)-\phi_{g}\left(\boldsymbol{\theta}\right)\right]+\log\left\{ \E_{\boldsymbol{\theta}\sim g_{LAP}\left(\boldsymbol{\theta}\right)}\left[\exp\left(\phi_{g}\left(\boldsymbol{\theta}\right)-\phi_{f}\left(\boldsymbol{\theta}\right)\right)\right]\right\} \\
 & \approx\frac{1}{s}\sum_{i=1}^{s}\left[\phi_{f}\left(\boldsymbol{\theta}_{i}\right)-\phi_{g}\left(\boldsymbol{\theta}_{i}\right)\right]+\log\left\{ \frac{1}{s}\sum_{i=1}^{s}\left[\exp\left(\phi_{g}\left(\boldsymbol{\theta}_{i}\right)-\phi_{f}\left(\boldsymbol{\theta}_{i}\right)\right)\right]\right\} 
\end{align}
where the $\boldsymbol{\theta}_{i}$ are $s$ IID samples from $g\left(\boldsymbol{\theta}\right)$.\\
This yields an asymptotically correct estimator of $KL\left(g_{LAP},f\right)$
as $s\rightarrow\infty$ but could have possibly bad properties for
finite $s$, depending on the properties of the random variable: $\exp\left(\phi_{g}\left(\boldsymbol{\theta}_{g}\right)-\phi_{f}\left(\boldsymbol{\theta}_{g}\right)\right)$,
namely whether it has a finite variance that is furthermore small
compared to $s$.
\item I also considered an indirect evaluation of the KL divergence based
on first approximating the normalization constant:
\begin{equation}
Z_{f}=\int\exp\left(-\phi_{f}\left(\boldsymbol{\theta}\right)\right)
\end{equation}
I evaluated this quantity by first generating samples from $f\left(\boldsymbol{\theta}\right)$
using the NUTS MCMC algorithm (\citet{hoffman2014no}) using an implementation
from \href{https://github.com/mfouesneau/NUTS}{M. Fouesnau on github}.
Then, I used the following identity:
\begin{align}
\frac{1}{Z_{f}} & =\frac{\int g_{LAP}\left(\boldsymbol{\theta}\right)d\boldsymbol{\theta}}{\int\exp\left(-\phi_{f}\left(\boldsymbol{\theta}\right)\right)d\boldsymbol{\theta}}\\
 & =\int\frac{g_{LAP}\left(\boldsymbol{\theta}\right)\exp\left(\phi_{f}\left(\boldsymbol{\theta}\right)-\phi_{f}\left(\boldsymbol{\theta}\right)\right)d\boldsymbol{\theta}}{\int\exp\left(-\phi_{f}\left(\boldsymbol{\theta}\right)\right)d\boldsymbol{\theta}}\\
 & =\int f\left(\boldsymbol{\theta}\right)\exp\left(\phi_{f}\left(\boldsymbol{\theta}\right)-\phi_{g}\left(\boldsymbol{\theta}\right)\right)d\boldsymbol{\theta}\\
 & =\E_{\boldsymbol{\theta}\sim f\left(\boldsymbol{\theta}\right)}\left[\exp\left(\phi_{f}\left(\boldsymbol{\theta}\right)-\phi_{g}\left(\boldsymbol{\theta}\right)\right)\right]
\end{align}
and the corresponding empirical approximation based on the samples
to approximate $Z_{f}$.\\
The KL approximation was then computed using the straightforward sampling
approximation of the expression:
\begin{equation}
KL\left(g_{LAP},f\right)=\E_{\boldsymbol{\theta}\sim g_{LAP}\left(\boldsymbol{\theta}\right)}\left[\phi_{f}\left(\boldsymbol{\theta}\right)-\phi_{g}\left(\boldsymbol{\theta}\right)\right]+\log\left\{ Z_{f}\right\} 
\end{equation}
This approximation thus combines IID samples from $g_{LAP}\left(\boldsymbol{\theta}\right)$
and MCMC samples from $f\left(\boldsymbol{\theta}\right)$.
\end{enumerate}
In the experiments, the two approximations of the KL divergence ended
up giving similar values (Appendix Fig.\ref{fig: APPENDIX the two kl approx are similar}).
The analysis in the main text was thus carried using only the NUTS-based
approximation of the KL divergence (considering instead the other
approximation does not modify the results; see Appendix Fig.\ref{fig: APPENDIX KL to n for the other approx}).

In the sampling approximations, 50.000 samples from $g_{LAP}\left(\boldsymbol{\theta}\right)$
were used. Similarly, 60.000 samples from the NUTS Markov Chain were
constructed if $p\in\left\{ 10,30,100\right\} $. For $p\in\left\{ 300,1000\right\} $,
260.000 samples were used instead, to account for increased correlation.
The first 10.000 samples were discarded. To minimize storage into
RAM for $p\in\left\{ 300,1000\right\} $, only one out of five of
the remaining 250.000 samples were retained for final analysis, yielding
50.000 samples.

Due to large number of experiments, the 50.000 MCMC samples were inspected
automatically. Samples from the chain were deemed appropriate if the
max of the empirical auto-correlation of $\phi_{f}\left(\boldsymbol{\theta}_{t}\right)$
for offset $\delta\in\left[30,100\right]$ was smaller than $0.05$.
This condition was fulfilled throughout our experiments and never
resulted in a dataset being rejected.

\begin{figure}
\begin{centering}
\includegraphics[width=13cm]{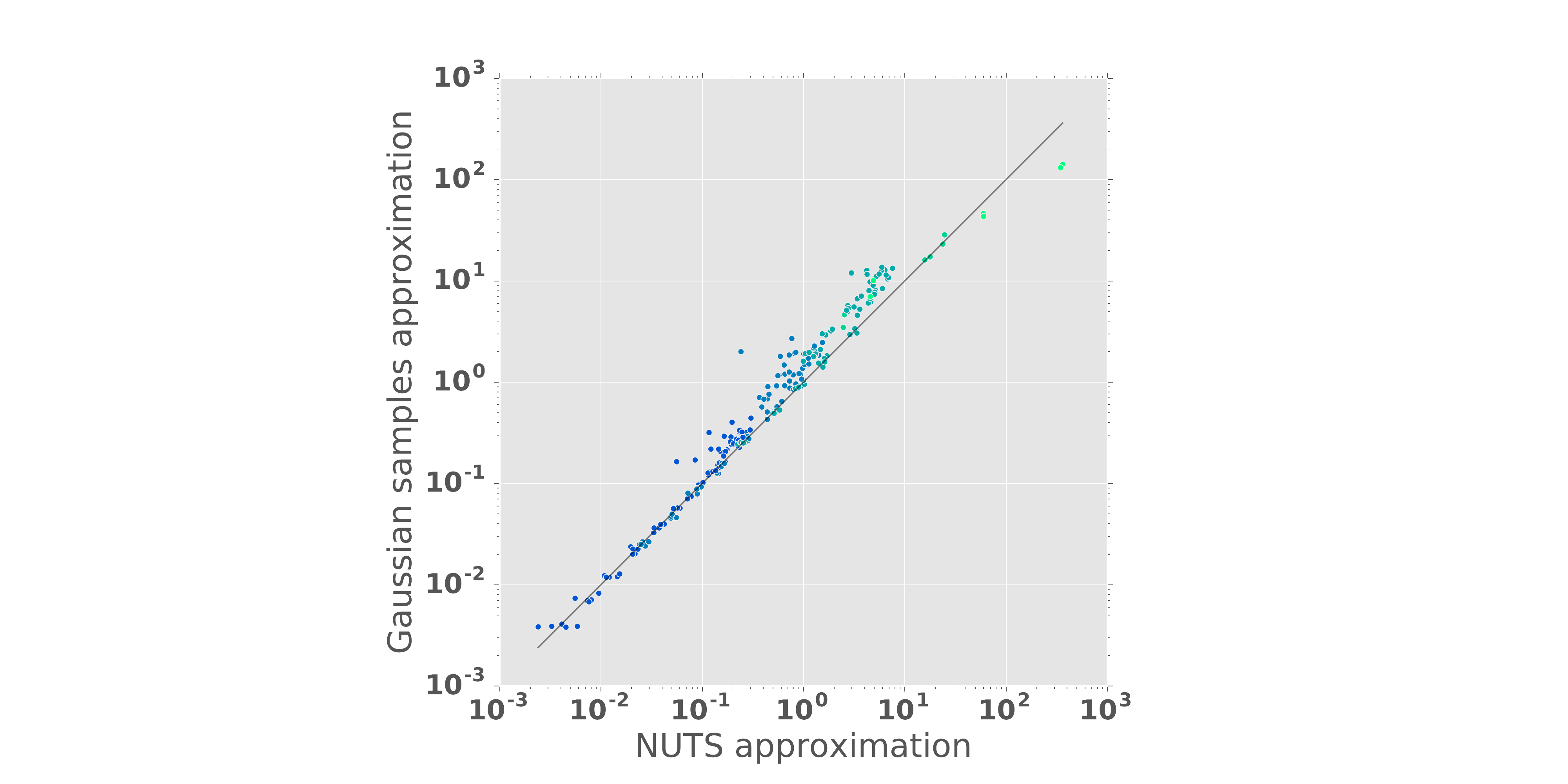}
\par\end{centering}
\caption{\label{fig: APPENDIX the two kl approx are similar}The two approximation
of $KL(g_{LAP},f)$ are similar. As explained in the text, two numerical
approximations of the KL divergence were considered. The difference
was the computation of the normalization constant $\log(Z_{f})$.
This figure corresponds to a simple scatter plot of both approximations,
superposed with the identity line. Colors are the same as in Fig.\ref{fig: evolution of KL with n}
and Fig.\ref{fig: quality of KL var and KL var + LSI}. The results
presented in the main text relied on the NUTS approximation.}
\end{figure}
\begin{figure}
\begin{centering}
\includegraphics[width=13cm]{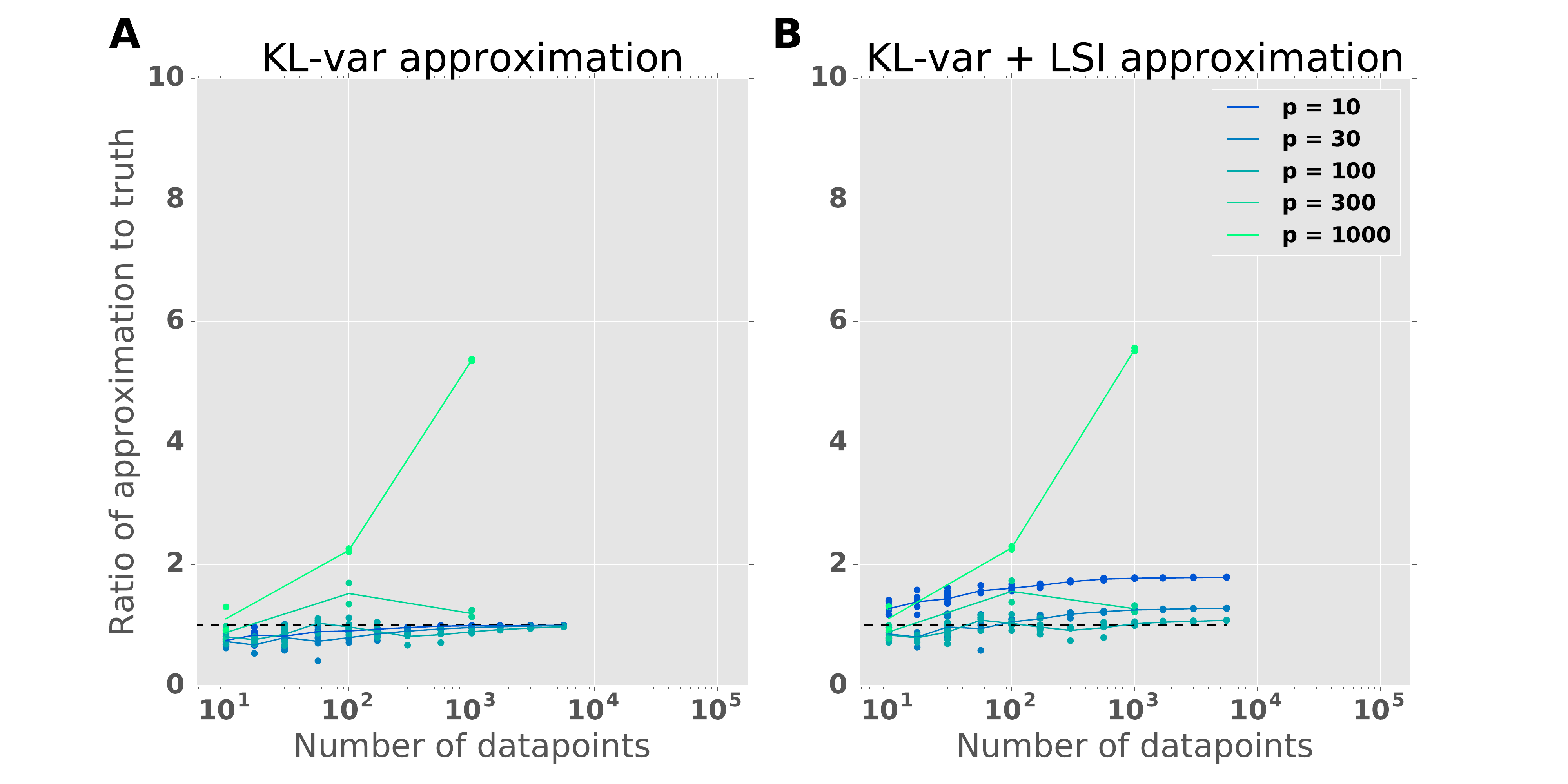}
\par\end{centering}
\caption{\label{fig: APPENDIX KL to n for the other approx}Evolution of the
KL-var and \textquotedbl KL-var + LSI\textquotedbl{} approximations
with $n$, when measured with respect to the other numerical approximation
of $KL(g_{LAP},f)$. Panels A and B as in Fig.\ref{fig: quality of KL var and KL var + LSI}.
The behavior is qualitatively similar as in Fig.\ref{fig: quality of KL var and KL var + LSI}.}
\end{figure}
\newpage{}

\subsection{Additional figures}

Due to space constraints, we present here the evolution of the other
approximations suggested by Cor.\ref{cor:Computable-approximations-of KL(g,f)}.

We also present the evolution of the quality of the approximations
as a function of the true KL divergence $KL\left(g_{LAP},f\right)$.

\begin{figure}
\centering{}\includegraphics[height=20cm]{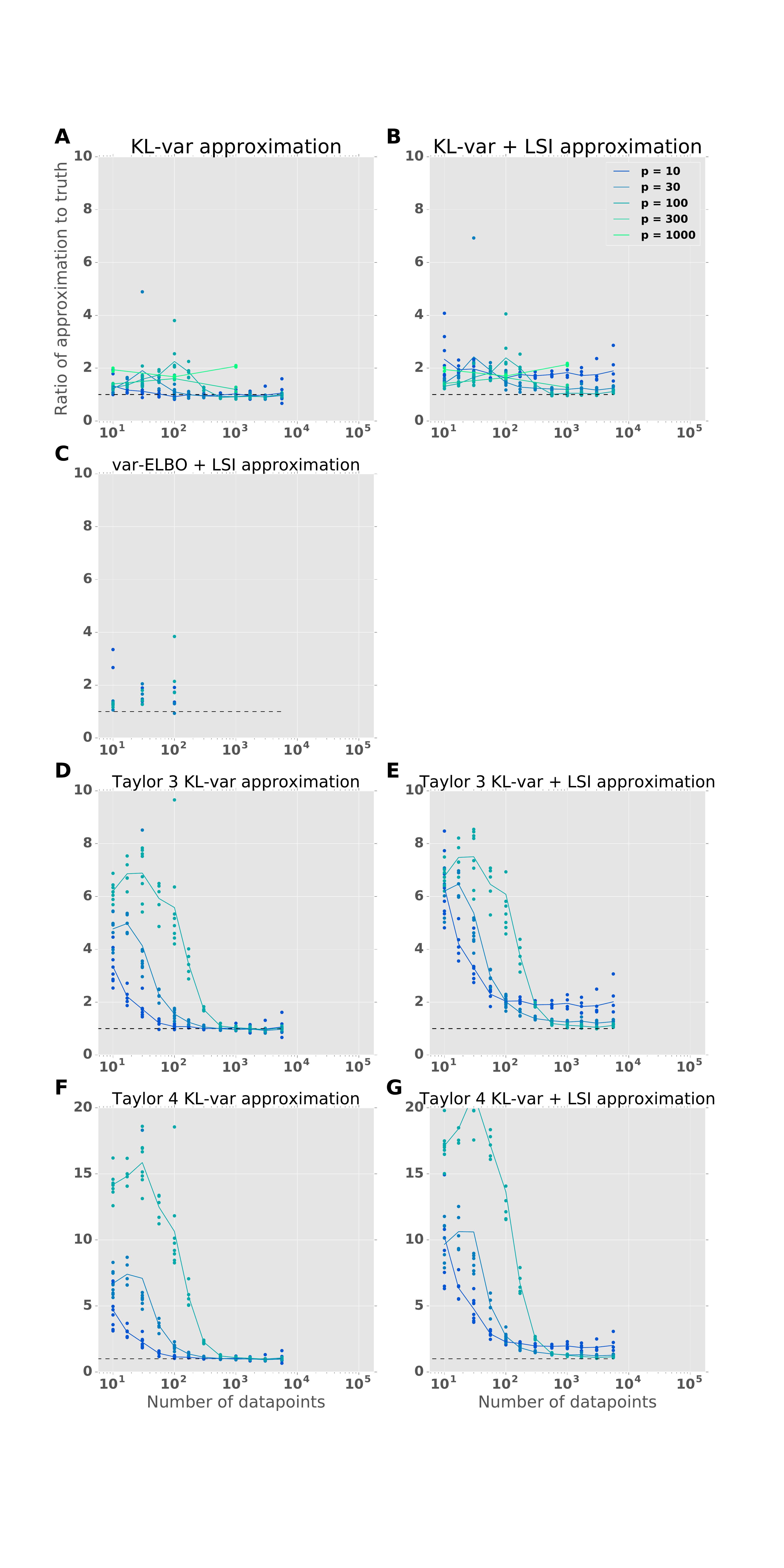}
\end{figure}
\begin{figure}
\contcaption{

\label{fig: APPENDIX evolution of all approximations}Evolution of
the ratio \textquotedbl approximation of KL / KL\textquotedbl{} for
all approximations proposed in Cor.\ref{cor:Computable-approximations-of KL(g,f)}.
A: KL-var approximation (repeated from Fig.\ref{fig: quality of KL var and KL var + LSI}).
B: \textquotedbl KL-var + LSI\textquotedbl{} approximation (repeated
from Fig.\ref{fig: quality of KL var and KL var + LSI}). C: \textquotedbl var
ELBO + LSI\textquotedbl{} approximation. Note that due to the computational
costs associated with this approximation, it was not computed for
larger values of $p$ and $n$. D,E: Taylor approximations to third
order of the KL-var and \textquotedbl KL-var + LSI\textquotedbl{}
approximations. F,G: Taylor approximations to fourth order of the
KL-var and \textquotedbl KL-var + LSI\textquotedbl{} approximations.
Please note the change of scale in panels F and G. The Taylor approximations
severely overestimate the true KL divergence for small values of $n$
and only become correct as $n$ is high enough. Note that due to computational
constraints associated with computing large tensors, the Taylor approximations
were not computed for the largest values of $p$.

 }
\end{figure}
\begin{figure}
\centering{}\includegraphics[height=20cm]{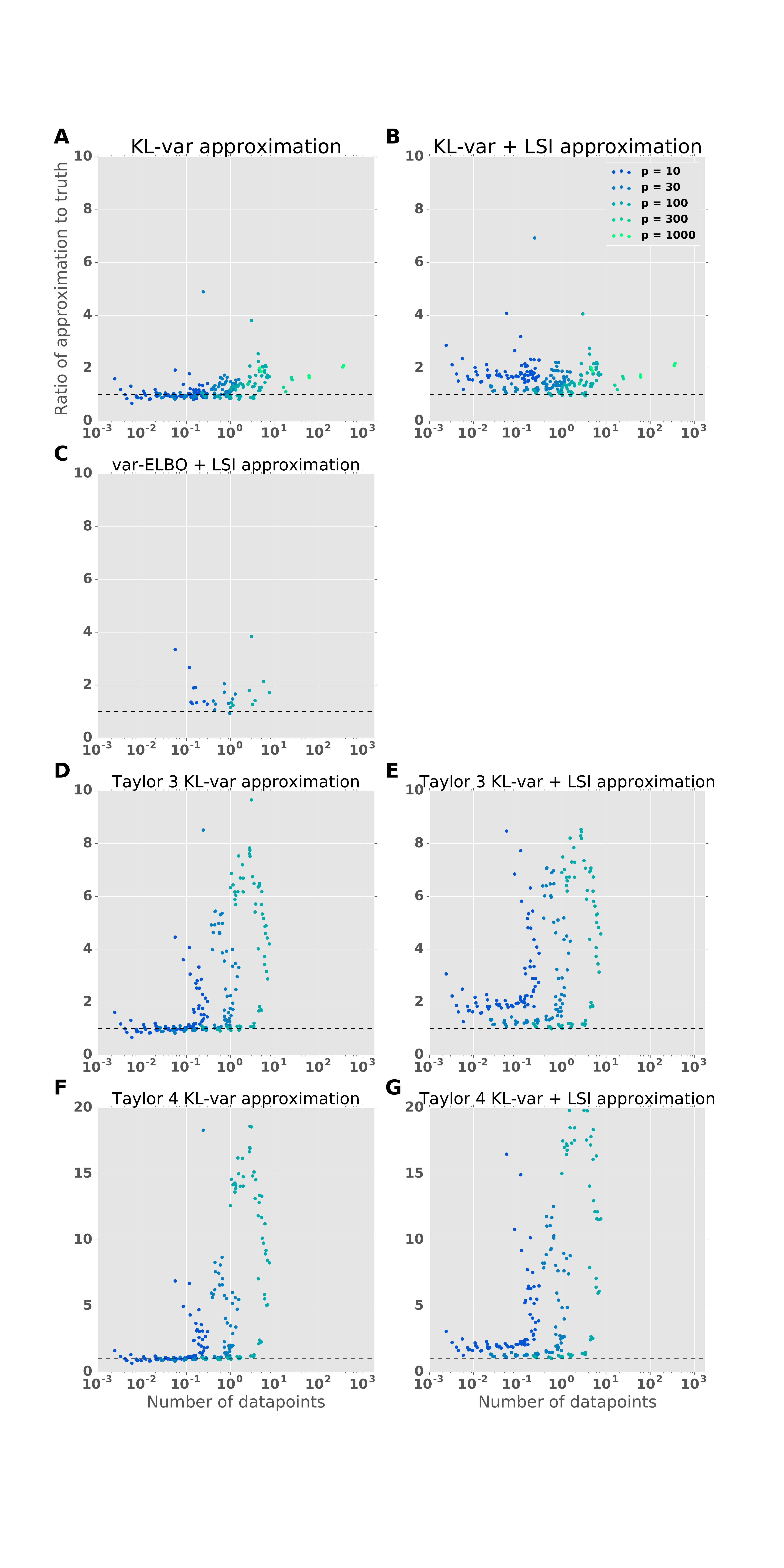}
\end{figure}
\begin{figure}
\contcaption{

\label{fig: APPENDIX evolution against true KL-1}Evolution of the
ratio \textquotedbl approximation of KL / KL\textquotedbl{} for all
approximations proposed in Cor.\ref{cor:Computable-approximations-of KL(g,f)}
but as a function of the true KL divergence. All panels as in Supplementary
Fig.\ref{fig: APPENDIX evolution of all approximations}. Once again,
note the change of scale in panels F and G. This representation shows
that the KL-var approximation, and the other approximations proposed
in Cor.\ref{cor:Computable-approximations-of KL(g,f)} are consistently
tight whenever the true KL divergence is low and that they are otherwise
upper-bounds. Further investigations are required in order to determine
whether it is possible to prove that this relationship holds in a
wider variety of situations.

 }
\end{figure}

\end{document}